\documentclass[a4paper,11pt,oneside,reqno]{amsart}
\usepackage{geometry}
\usepackage{tikz}\usetikzlibrary{patterns}
\usepackage{amsmath, amsthm, amssymb}
\usepackage{fancyhdr}
\usepackage{hyperref}
\hypersetup{colorlinks=true, linkcolor=blue,citecolor=blue, urlcolor=blue}
\usepackage[abbrev,lite,nobysame]{amsrefs}
\usepackage{times}
\usepackage{bbm}
\usepackage{subcaption}
\usepackage{mathtools}
\usepackage{pdfpages}
\usepackage{comment}

\DeclareMathAlphabet{\mathcal}{OMS}{cmsy}{m}{n}
\definecolor{redr}{RGB}{120,0,28}
\newcommand{\virg}[1]{``#1''}

\theoremstyle{plain}
\newtheorem{theorem}{Theorem}[section]
\newtheorem{lemma}[theorem]{Lemma}
\newtheorem{proposition}[theorem]{Proposition}
\newtheorem{corollary}[theorem]{Corollary}

\theoremstyle{definition}
\newtheorem{remark}[theorem]{Remark}

\newtheorem*{namedthm*}{\namedthmname}

\renewcommand{\div}{\mathrm{div}}
\newcommand{\dt}{\partial_t}
\newcommand{\dx}{\partial_x}
\newcommand{\dy}{\partial_y}
\newcommand{\dX}{\partial_X}
\newcommand{\dY}{\partial_Y}
\newcommand{\dtt}{\partial_{tt}}

\newcommand{\dyy}{\partial_{yy}}
\newcommand{\dXX}{\partial_{XX}}

\newcommand{\M}{M}
\newcommand{\dd}{\mathrm{d}}
\newcommand{\Dt}{\frac{\dd}{\dd t}}
\newcommand{\uu}{\boldsymbol{u}}
\newcommand{\vv}{\boldsymbol{v}}

\newcommand{\R}{\mathbb{R}}
\newcommand{\T}{\mathbb{T}}

\newcommand{\p}{p}

\newcommand{\norma}[2]{\left\lVert #1 \right\rVert_{#2}}
\newcommand{\hrho}{\widehat{\rho}}
\newcommand{\halpha}{\widehat{\alpha}}

\newcommand{\hXi}{\widehat{\Xi}}

\newcommand{\hOmega}{\widehat{\Omega}}
\newcommand{\hA}{\widehat{A}}
\newcommand{\hR}{\widehat{R}}

\renewcommand{\Re}{\operatorname{Re}}
\newcommand{\jap}[1]{\left\langle #1 \right\rangle}
\newcommand{\norm}[1]{\left\lVert #1 \right\rVert}

\usepackage{etoolbox}
\patchcmd{\section}{\scshape}{\large\bfseries}{}{}
\makeatletter
\renewcommand{\@secnumfont}{\bfseries}
\makeatother
\numberwithin{equation}{section}

\usepackage{etoolbox}
\patchcmd{\subsubsection}{\itshape}{\itshape\bfseries}{}{} 

\usepackage{xpatch}
\makeatletter
\xapptocmd{\@sect}{\csname #1mark\endcsname{#7}}{}{}

\usepackage{doc}
\makeatletter
\def\maketitle{\par
	\begingroup
	\def\thefootnote{\fnsymbol{footnote}}%
	\setcounter{footnote}\z@
	\def\@makefnmark{\hbox to\z@{$\m@th^{\@thefnmark}$\hss}}%
	\long\def\@makefntext##1{\noindent
		\ifnum\c@footnote>\z@\relax
		\hbox to1.8em{\hss$\m@th^{\@thefnmark}$}##1%
		\else
		\hbox to1.8em{\hfill}%
		\parbox{\dimexpr\linewidth-1.8em}{\raggedright ##1}%
		\fi}
	\if@twocolumn\twocolumn[\@maketitle]%
	\else\newpage\global\@topnum\z@\@maketitle\fi
	\thispagestyle{titlepage}\@thanks\endgroup
	\setcounter{footnote}\z@
	\gdef\@date{\today}\gdef\@thanks{}%
	\gdef\@author{}\gdef\@title{}}

\author{Paolo Antonelli \textsuperscript{*}}
\author{Michele Dolce \textsuperscript{*,$\dagger$}}
\author{Pierangelo Marcati\textsuperscript{*}}
\title{Linear stability analysis of the homogeneous Couette flow in a 2D isentropic compressible fluid \footnote{\normalfont GSSI - Gran Sasso Science Institute, Viale Francesco Crispi 7, 67100, L'Aquila, IT} \footnote{\normalfont Imperial College London, Department of Mathematics, London, SW7 2AZ, UK}}
\address{GSSI - Gran Sasso Science Institute, Viale Francesco Crispi 7, 67100, L'Aquila, IT}
\email{paolo.antonelli@gssi.it}
\address{Imperial College London, Department of Mathematics, London, SW7 2AZ, UK}
\email{m.dolce@imperial.ac.uk}
\address{GSSI - Gran Sasso Science Institute, Viale Francesco Crispi 7, 67100, L'Aquila, IT}
\email{pierangelo.marcati@gssi.it}

	\mathtoolsset{showonlyrefs=true}
	
\begin{document} 
	\maketitle
	\begin{abstract}
		 In this paper, we study the linear stability properties of perturbations around the homogeneous Couette flow for a 2D isentropic compressible fluid in the domain $\mathbb{T}\times \mathbb{R}$.
		 
		  In the inviscid case there is a generic Lyapunov type instability for the density and the irrotational component of the velocity field. More precisely, we prove that their $L^2$ norm grows as $t^{1/2}$ and this confirms previous observations in the physics literature. Instead, the solenoidal component of the velocity field experience inviscid damping, meaning that it decays to zero even in the absence of viscosity.  
		 
		For a viscous compressible fluid, we show that the perturbations may have a transient growth of order $\nu^{-1/6}$ (with $\nu^{-1}$ being proportional to the Reynolds number) on a time-scale $\nu^{-1/3}$, after which it decays exponentially fast. This phenomenon is also called enhanced dissipation and our result appears to be the first to detect this mechanism for a compressible fluid, where an exponential decay for the density is not a priori trivial given the absence of dissipation in the continuity equation.
	\end{abstract}
\tableofcontents
	\section{Introduction}
	\label{sec:intro}
		   We consider the isentropic compressible Navier-Stokes system 
	\begin{align}
		\label{eq:continuity} &\dt \tilde{\rho} +\div (\tilde{\rho} \uu)=0,\ \ \ \text{for }  (x,y)\in\mathbb{T}\times \mathbb{R}, \ t\geq 0,\\
		\label{eq:momentum} &\dt (\tilde{\rho} \uu )+\div (\tilde{\rho} \uu\otimes \uu)+\frac{1}{\M^2}\nabla p(\tilde{\rho})=\nu\Delta\uu+\lambda \nabla \div(\uu) ,
	\end{align}
	 in a periodic strip where $\mathbb{T}=\mathbb{R}/\mathbb{Z}$. Here, $\tilde{\rho}$ is the density of the fluid, $\uu$ the velocity, $p(\tilde{\rho})$ the pressure, $\M$ is the Mach number and $\nu,\lambda\geq 0$ are the shear and bulk viscosities coefficients respectively. The shear viscosity is proportional to the inverse of the Reynolds number. When $\nu=\lambda=0$ we are reduced to the Euler system.
	
	A stationary solution to \eqref{eq:continuity}-\eqref{eq:momentum} is given by the \textit{homogeneous Couette flow}, namely a shear flow with a linear velocity profile $u_E=(y,0)$ with constant density $\rho_E=1$. We are interested in studying the linear stability properties of this flow. Therefore, we consider a perturbation around it which is given by
	\begin{equation*}
		\widetilde{\rho}=\rho_E+\rho, \qquad
		\uu=\uu_E+\vv,
	\end{equation*}
	for $\tilde{\rho}, \uu$ satisfying \eqref{eq:continuity}-\eqref{eq:momentum}.  
	
	The linearized system around the homogeneous Couette flow read as follows 
		\begin{align}
			\label{eq:lincouette1}&\dt \rho+y\dx \rho+\div (\vv)=0, \ \ \ \text{for }  (x,y)\in \mathbb{T}\times \mathbb{R}, \ t\geq 0,\\
			\label{eq:lincouette2}&\dt \vv+y\dx \vv+\begin{pmatrix}
				v^y\\ 0
			\end{pmatrix}+\frac{1}{M^2}\nabla \rho=\nu \Delta \vv+\lambda \nabla \div(\vv),
		\end{align}
	where we set $p'(1)=1$. 
	
The study of linear stability properties of particular solutions to the equations governing the motion of a fluid is a classical topic in \textit{hydrodynamic stability theory} \cite{drazin2004hydrodynamic,yaglom2012hydrodynamic}.	

For an incompressible fluid, the linear analysis for the Couette flow was already done by Kelvin \cite{kelvin1887stability} in the 1887. Other classical results have been obtained via an eigenvalue (or normal mode) analysis in many different cases, however, the classical stability analysis in general does not agree with the numerical and physical observations \cite{drazin2004hydrodynamic,yaglom2012hydrodynamic,trefethen1993hydrodynamic}. For instance, Trefethen et al. \cite{trefethen1993hydrodynamic} observed that a common feature in these problems is the non-normality of the operators involved. In particular, this implies the possibility of large \textit{transient growths} (which are not captured via a pure eigenvalue analysis) that can take out the dynamics from the linear regime before the stability mechanisms takes over. The Couette flow is the simpler flow where these phenomenon are present, therefore the stability analysis of this particular case is the prototypical example to understand some of the mechanisms involved in the dynamics.

The analysis of the full nonlinear problem is extremely challenging even in the simpler cases. For an incompressible and homogeneous fluid in the Euler regime, Arnold \cite{arnold1965conditions} obtained an elegant stability result for a particular class of shear flows. However, some relevant flows, such as the Couette one, do not belong to this class. In the last ten years, the problem received a renewed attention. For the Couette flow, Lin and Zeng \cite{lin2011inviscid} showed the existence of stationary solutions which are not shear flows and are arbitrarily close to the Couette flow in a sufficiently low regularity space ($H^{3/2^{-}}$ for the vorticity), meaning that a perturbation around Couette may not converge towards a shear flow. A breakthrough in the understanding of the nonlinear stability properties of the planar Couette flow has been made by Bedrossian and Masmoudi \cite{bedrossian2015inviscid}. In particular, in the domain $\T\times \R$ they proved the asymptotic stability of the vorticity in a high-regularity space (Gevrey-$1/2^-$) which implies the \textit{inviscid damping} of the velocity field. Namely, the vorticity is \textit{mixed} by the background flow and the velocity field strongly converge in $L^2$ to a shear flow \textit{close} to Couette with polynomials rates of convergence. This phenomenon share analogies with the \textit{Landau damping} \cite{caglioti1998time,MV11,bedrossian2016landau,grenier2020landau}. For more general shear flows the analysis is highly non-trivial even at a linear level \cite{wei2018linear,WZZ17,jia2020linear,zillinger2016linear} and nonlinear inviscid damping results have been obtained very recently \cite{ionescu2020nonlinear,masmoudi2020nonlinear}. Linear inviscid damping results are available also for the Couette flow \cite{yang2018linear} and shear near Couette in a 2D inhomogeneous incompressible fluid  \cite{bianchini2020linear}.

When viscosity is present more stability results are available, for instance the nonlinear Couette case was studied by Romanov in the '70s \cite{romanov1973stability}. The stability mechanism present at the inviscid level can also combine with the dissipation and one observes an \textit{enhanced dissipation} of some components of the perturbations around the equilibrium. This is possible since the advection causes an energy cascade towards small spatial scales where dissipation takes over. In addition, due to the observed transient growths, a question of great interest in the viscous problem is a quantification of \textit{transition thresholds}, namely how small the initial perturbation has to be with respect to the viscosity parameter. On this side, several numerical studies predicted a power law dependence and estimated the exponents below which stability is possible. In the last ten years, enhanced dissipation and transition thresholds results have been proved in several cases and we refer to \cite{bedrossian2019stability} for a detailed literature review on the known results until 2017. More recent results, including vortices, inhomogeneous incompressible fluids and passive scalar problems related to fluid dynamics can be found in \cite{bedrossian2020he,cotizelati2020separation,cotizelati2020enhancedNS,zelati2019stochastic,chen2018transition,chen2020transition,chen2019linear,gallay2020enhanced,deng2020stability,li2020pseudo,masmoudi2020stability,zillinger2020boussinesq,zillinger2020enhanced}. 

In the compressible case the literature is significantly less developed with respect to the incompressible one. The extension of the standard stability analysis to the compressible case has been already considered starting from the '40s  \cite{lees1946investigation,eckart1963extension,blumen1970shear,blumen1975shear,drazin1977shear,subbiah1990stability}.	In the review paper \cite{holm1985nonlinear} there is the extension of Arnold's method for a 2D isentropic compressible fluid. However, at least to our knowledge, a complete non-modal mathematical analysis of the Couette flow was not considered previously in the literature. The linearization around the Couette flow in the 2D isentropic compressible Euler dynamics was instead considered in the physics literature, both from the numerical point of view and from the theoretical one, a highly incomplete list of papers includes \cite{bakas2009mechanisms,chagelishvili1994hydrodynamic,chagelishvili1997linear,hau2015comparative,hanifi1996transient} and references therein. In particular, the 2D inviscid problem (with an additional Coriolis forcing term) has been considered as a first model to understand the formation of \textit{spiral arms} in a rotating disk galaxy by Goldreich and Lynden-Bell \cite{goldreich1965gravitational,goldreich1965ii}. In \cite[Sec. 5-6]{goldreich1965ii} they directly consider the linearized initial value problem and they derive a second order ODE satisfied by the density in the Fourier space. From this equation, appealing to some formal approximation, they deduce an instability phenomenon that appears specifically due to the compressibility of the flow. More precisely, they obtain that $|\rho(t)|\sim O(t^{1/2})$. The problem (without the Coriolis force) was then studied also by Chagelishvili et al. in \cite{chagelishvili1994hydrodynamic,chagelishvili1997linear} where, with analogous computations, it is observed that $$|\rho(t)|+|\vv(t)|\sim O(t^{1/2}).$$ In \cite{bodo2005spiral} the analysis of Goldreich and Lynden-Bell was revisited, supplemented with numerical simulations and are also highlighted similarities to the case studied in \cite{chagelishvili1997linear}. Recently, more refined numerical simulations and analysis can be found in \cite{bakas2009mechanisms,hau2015comparative}, where the results are justified with some formal asymptotic expansion. We notice that in the results mentioned above the asymptotic expansions are formally justified if $M\ll1$, whereas in astrophysical applications the interesting regimes are for $M>1$ \cite{bodo2005spiral,makita2000two}.
	
	For the viscous compressible plane Couette flow, Glatzel in 1988 \cite{glatzel1989linear} has investigated linear stability properties via a normal mode analysis, see also \cite{duck1994linear,hu1998linear}. Hanifi et al. in \cite{hanifi1996transient} have numerically investigated a transient growth mechanism in the \textit{non-isothermal} case, showing that the maximum transient growth scales as $O(\nu^{-2})$ and increases with increasing Mach
	number, see also the more recent result \cite{malik2008linear}. Then, Farrell and Ioannou in \cite{farrell2000transient} considered the linear problem \eqref{eq:lincouette1}-\eqref{eq:lincouette2} and showed a rapid transient energy growth, that at large Mach numbers greatly exceed the expected one in the incompressible case, which is then damped due to the effect of viscosity. By some heuristic argument, the authors have also observed that the transient growth is due to purely inviscid and compressible effects in agreement with \cite{chagelishvili1994hydrodynamic,chagelishvili1997linear}. In addition, in the numerical simulations shown in \cite[Fig. 1]{farrell2000transient} the author consider as parameters $M=\{0,50\}$ and $\nu^{-1}=5000$ (where $M=0$ correspond to the incompressible dynamics) and in both cases the perturbation decay on a time-scale much faster with respect to the standard diffusive one (which is $O(\nu^{-1})$). In the viscous incompressible case it is well known \cite{bedrossian2019stability} that the dissipation become effective on a time-scale $O(\nu^{-\frac13})$, whereas for the compressible case a precise quantification of the time-scale is not known. Lastly, in the mathematics literature we mention the more recent results obtained by Kagei \cite{kagei2011asymptotic,kagei2012asymptotic}, where the Couette flow is generated by the top plate of an infinite channel, in dimension $n$, moving along the $x$-direction with constant velocity and with the bottom plate fixed. The author proves an asymptotic stability result for small Mach and Reynolds numbers. The conditions on the parameters have been relaxed by Li and Zhang \cite{li2017stability} imposing Navier-slip boundary conditions at the bottom plate.
	\subsection{Statement of the results}
	In this paper, we confirm  and make more precise the linear inviscid instability phenomena found in the above mentioned literature. In particular, we are able to prove that the solenoidal component of the velocity field experience inviscid damping whereas the irrotational component and the density have a linear Lyapunov instability for a \textit{generic} class of initial data. Then, in the viscous case, we confirm the observations made in \cite{farrell2000transient} by showing that dynamics is qualitatively the same to the inviscid case up to a time-scale $O(\nu^{-\frac13})$, after which viscosity become effective and the perturbations decay exponentially fast. This is not a priori trivial in view of the absence of viscosity in the continuity equation. 
	
	Before stating our main result, for any velocity field $\vv$, we denote 
	\begin{equation*}
		\alpha=\div (\vv), \qquad \omega=\nabla^\perp\cdot \vv.
	\end{equation*}
	The Helmholtz projection operators are defined in the usual way, namely 
	\begin{equation}
		\label{Helmholtzint}
		\vv=(v^x,v^y)^T=\nabla \Delta^{-1}\alpha+\nabla^\perp \Delta^{-1}\omega:= Q[\vv]+P[\vv],
	\end{equation} 
	where $\nabla^\perp=(-\dy,\dx)^T$. The system \eqref{eq:lincouette1}-\eqref{eq:lincouette2}, in terms of $(\rho, \alpha, \omega)$ read as
	\begin{align}
		\label{eq:contcouetteint}&\dt \rho +y\dx\rho+\alpha=0,\qquad  \ \ \ \text{for }  (x,y)\in \mathbb{T}\times \mathbb{R}, \ t\geq 0, \\
		\label{eq:divcouetteint}
		&\dt \alpha+y\dx \alpha+2\dx v^y+\frac{1}{\M^2}\Delta\rho=(\nu+\lambda)\Delta\alpha,\\
		&\label{eq:vorticitycouetteint}	\dt \omega+y\dx \omega-\alpha=\nu\Delta \omega.	
	\end{align}	
	The second component of the velocity $v^y$ can be recovered by means of  the Helmholtz decomposition 
	\begin{equation}\label{eq:v2}
		v^y=\dy(\Delta^{-1})\alpha+\dx(\Delta^{-1})\omega,
	\end{equation}
	hence \eqref{eq:contcouetteint}-\eqref{eq:vorticitycouetteint} is a closed system in terms of the variables $(\rho,\alpha,\omega)$. Notice that when $\nu=0$ one has 
	\begin{equation}
		\label{eq:consrom}
		(\dt+y\dx) (\rho+\omega)=0,
		\end{equation}
	 namely, there is an extra conservation law along the flow. We comment about this important feature in the sequel.
	
	In the following, we are going to denote
	\begin{equation}
		\label{def:f0}
		f_0(y)=\frac{1}{2\pi}\int_{\mathbb{T}}f(x,y)dx.
	\end{equation}
	For the sake of brevity, we do not explicit the dependence of the bounds with respect to the Sobolev norms of the initial data. We simply write $C_{in}=C_{in}(\rho^{in},\alpha^{in},\omega^{in})$ to indicate a suitable combination of Sobolev norms of the initial data. Those constants may also depend on quantities increasing with respect to the Mach number, namely $(1+M)^\beta$ or $\exp( M^\beta)$ for some $\beta\geq 1$.
	A more precise statement of the theorem below will be given in Sections \ref{sec:invcompcouette} and \ref{sec:visccompcouette}. 
	
	\begin{theorem}
		\label{th:NScouetteintro}
		Let  $\nu,\lambda \geq 0$ and $M>0$ be such that $\nu+\lambda \leq 1/2$ and $M\max\{(\nu+\lambda)^{\frac12},\nu^{\frac13}\} \leq 1$. Let $\rho^{in} \in H^{7}(\mathbb{T}\times \mathbb{R})$ and $\alpha^{in}, \omega^{in} \in H^{6}(\mathbb{T}\times \mathbb{R})$. Then, the $x$-average of the solution satisfy
		\begin{align}
			\label{bd:k01}
			\norm{\alpha_0(t)}_{L^2}+\frac{1}{M}\norm{\dy \rho_0(t)}_{L^2}\leq&\  \frac{C_{in}}{(1+\nu t)^\frac12},\\
			\label{bd:k02}\norm{v^y_0(t)}_{L^2}+\frac{1}{M}\norm{\rho_0(t)}_{L^2}+\norm{\omega_0(t)}_{L^2}\leq &\ C_{in}.
		\end{align} 
		For the fluctuations around the $x$-average the following inequalities holds:
		\begin{align}
			\label{bd:compenh}
			&\norma{(Q[\vv]-Q[\vv]_0)(t)}{L^2}+\frac{1}{\M}\norma{(\rho-\rho_0)(t)}{L^2}\leq \ \langle t \rangle^{\frac12} e^{-\frac{1}{32}\nu^\frac13 t}C_{in},\\
			\label{bd:P1enh}&\norma{(P[\vv]^x-P[\vv]^x_0)(t)}{L^2}\leq \ M\frac{e^{-\frac{1}{64}\nu^\frac13 t}}{\langle t \rangle^{\frac12}} C_{in}+\frac{e^{-\frac{1}{12}\nu^\frac13 t}}{\langle t\rangle}(\norm{\omega^{in}+\rho^{in}}_{H^1}),\\
			\label{b:P2enh}&\norma{P[\vv]^y(t)}{L^2}\leq \ M\frac{e^{-\frac{1}{64}\nu^\frac13 t}}{\langle t\rangle^\frac32} C_{in}+\frac{e^{-\frac{1}{12}\nu^\frac13 t}}{\langle t \rangle^2}(\norm{\omega^{in}+\rho^{in}}_{H^2}).
		\end{align}
		Let $\nu=\lambda=0$  and $s\geq 0$. Up to a generic (in the sense of Baire) set of initial data $\rho^{in},\alpha^{in},\omega^{in} \in H^s(\T \times \R)$ one has 
		\begin{equation}
			\label{bd:lwcomp}
			\norma{(Q[\vv]-Q[\vv]_0)(t)}{L^2}+\frac{1}{\M}\norma{(\rho-\rho_0)(t)}{L^2}\geq\langle t \rangle^{\frac12 } C_{in}. 
		\end{equation}
	\end{theorem}
	Notice that when $\M=0$ (and $\alpha^{in}=\rho^{in}=0$), formally the estimates \eqref{bd:P1enh}-\eqref{b:P2enh} give the same result that one has in the incompressible case \cite{bedrossian2019stability}.
	
	We remark that the dynamics of the $x$-averages decouples with respect to fluctuations around it, as we will show in Section \ref{sec:zeromode} where we comment more about the evolution of the zero $x$-mode.
	
	In the following, we discuss the results given in the theorem above and we outline the strategy of proof by considering separately the inviscid and the viscous case, which we investigate in Section \ref{sec:invcompcouette} and Section \ref{sec:visccompcouette} respectively.
	\subsection{{Inviscid case}}
	For $\nu=\lambda =0$, the estimates \eqref{bd:compenh} and \eqref{bd:lwcomp} give the first rigorous justification to the growth predicted in \cite{bakas2009mechanisms,chagelishvili1994hydrodynamic,chagelishvili1997linear,hau2015comparative}, where, in order to implement a WKB asymptotic analysis,  the authors had to restrict themselves to a small Mach number regime. The result in the inviscid case was announced in our unpublished note \cite{antonelli2020linear}. We emphasize that the result stated in Theorem \ref{th:NScouetteintro} is actually more general since it removes the smallness  assumption on the Mach number. We also see that only the density and the irrotational part of the velocity field are growing, whereas in \eqref{bd:P1enh}-\eqref{b:P2enh} we show an inviscid damping result for the solenoidal component of the velocity, with slower decay with respect to the incompressible case. Indeed, this slow down of the inviscid damping is exactly compensated by the time growth of the compressible part of the fluid, as we explain in Remark \ref{rem:potvort} below.
		\begin{remark}
		For some particular initial data the lower bound in \eqref{bd:lwcomp} may not be valid. However, as we shall see in Proposition \ref{prop:lwdensity}, we are able to explicitly construct an arbitrary small perturbation of the initial data, at any fixed frequency $k,\eta$, for which the lower bound holds true. Therefore the set to exclude is nowhere dense in any Sobolev space in which the initial data is taken, which implies a \textit{generic} (in the sense of Baire category) Lyapunov type instability.
	\end{remark}
	\subsubsection{On the conservation of $\rho+\omega$}
			\label{rem:potvort}
	As observed in \eqref{eq:consrom}, when $\nu=0$ then $\rho+\omega$ is conserved along the Couette flow, namely 
	\begin{equation}
		\label{eq:conrrom}
		(\rho+\omega)(t,x+yt,y)=(\rho^{in}+\omega^{in})(x,y).
	\end{equation}
	This relation was clearly observed also in \cite{goldreich1965gravitational,goldreich1965ii,chagelishvili1994hydrodynamic,chagelishvili1997linear,bakas2009mechanisms,hau2015comparative}. In particular, in \cite{chagelishvili1997linear,bodo2005spiral} the authors notice that this conservation law cause an ``\textit{emergence of acoustic waves from vortices}". Indeed, \eqref{eq:conrrom} immediately connects compressible and incompressible phenomena. Namely, an increase of the vorticity need to be compensated by a decrease for the density and the other way around. In addition, Theorem \ref{th:NScouetteintro} implies that the density and the irrotational part of the velocity exhibit a growth in time even when the initial perturbation satisfies $\rho^{in}=\alpha^{in}=0$. This can be seen from the linearized equations \eqref{eq:contcouetteint}-\eqref{eq:vorticitycouetteint}, where the identity \eqref{eq:v2} for $v^y$ yields a source term, depending on the vorticity, in the equation for the divergence \eqref{eq:divcouetteint}. 
	 This interplay between density and vorticity is also the cause of the slow-down of the inviscid damping for the solenoidal component of the velocity with respect to the homogeneous incompressible case.
 
 We also point out that \eqref{eq:conrrom} can be seen as the linear analogue of the \textit{potential vorticity}, i.e. $\widetilde{\omega}/\widetilde{\rho}$, being transported along the flow, where $\widetilde{\omega}=\nabla^\perp \cdot \uu$ for $\widetilde{\rho}, \uu$ satisfying \eqref{eq:continuity}-\eqref{eq:momentum}. Indeed, a direct computation shows that 
		\begin{equation}
           \label{eq:potvort}
			\dt \left(\frac{\widetilde{\omega}}{\widetilde{\rho}}\right)+\uu\cdot \nabla \left(\frac{\widetilde{\omega}}{\widetilde{\rho}}\right)=0.
		\end{equation}  
		Then, since we are considering perturbations around the Couette flow we have 
		\begin{equation}
			\widetilde{\omega}=-1+\omega, \qquad \widetilde{\rho}=1+\rho.
		\end{equation}Consequently, writing down \eqref{eq:potvort} in Lagrangian coordinates we deduce 
		\begin{equation*}
			\omega(t,\mathbf{X}(\mathbf{x}_{in},t))+\frac{1-\omega^{in}(\mathbf{x}_{in})}{1+\rho^{in}(\mathbf{x}_{in})}\rho(t,\mathbf{X}(\mathbf{x}_{in},t))=\frac{\rho^{in}(\mathbf{x}_{in})+\omega^{in}(\mathbf{x}_{in})}{1+\rho^{in}(\mathbf{x}_{in})},
		\end{equation*}
		where $\mathbf{X}$ is the flow associated to $u$, given by 
		\begin{align*}
			&\frac{d}{dt} \mathbf{X}(t,\mathbf{x}_{in})=\uu(\mathbf{X}(t,\mathbf{X}_{in}))\\
			&\mathbf{X}(0,\mathbf{x}_{in})=\mathbf{x}_{in}.
		\end{align*}
	 Assuming that we are in perturbative regime, namely $|\omega^{in}|\ll 1$ and $|\rho^{in}|\ll 1$ one has $(1-\omega^{in})/(1+\rho^{in})\approx 1$. Hence, by the previous heuristic argument at the nonlinear level,  we see why, at least formally, the conservation of $\rho+\omega$ can be considered as a linear approximation of the conservation of the potential vorticity for perturbations around the Couette flow with constant density. 
		
		The exact conservation along the Couette flow of $\rho+\omega$ plays a central role in our analysis in the inviscid case.

	\subsubsection*{Strategy of proof} Let us now briefly discuss the strategy of proof for the Theorem \ref{th:NScouetteintro} when $\nu=\lambda=0$. First of all, we remove the transport terms by defining the change of coordinates dictated by the background shear. Then, on this reference frame we have the exact conservation of the quantity $\rho + \omega$, so that we are able to reduce the degrees of freedom for the system \eqref{eq:contcouetteint}-\eqref{eq:vorticitycouetteint} and write a $2\times 2 $ system only involving the density and the divergence in the moving frame. Taking its Fourier transform in all the space variables, it can be studied as a $2\times 2$ non-autonomous  dynamical system at any fixed frequency $k,\eta$. Performing a suitable symmetrization via time dependent Fourier multipliers,  we can infer an energy estimate useful to deduce some property of the associated solution operator. Once the dynamics at any fixed frequency is understood, Theorem \ref{th:NScouetteintro} follows as a consequence and can be proved by going back to the original variables. 
	
	We present a more precise statement of Theorem \ref{th:NScouetteintro} in Theorems \ref{maintheorem} and \ref{maintheoremlwz}, where we consider separately the upper and lower bounds respectively. 
	
	\subsection{{Viscous case}}
	Theorem \ref{th:NScouetteintro} for $\nu>0$  gives a rigorous mathematical justification for the observations made in \cite{farrell2000transient}. At least to our knowledge, it appears to be the first enhanced dissipation estimate in the compressible case. In the bound \eqref{bd:compenh} we see the possibility of a large transient growth of order $O(\nu^{-\frac16})$ on a time scale $O(\nu^{-\frac13})$. This growth is due to the instability mechanism found in the inviscid case. Instead, the bounds \eqref{bd:P1enh}-\eqref{b:P2enh} combines inviscid damping and enhanced dissipation for the solenoidal component of the velocity. 
	
	The numerical observations made in \cite{farrell2000transient,hanifi1996transient,malik2008linear} shows that the transient growth increases at increasing Mach number, see for example \cite[Fig. 9]{hanifi1996transient}. In Theorem \ref{th:NScouetteintro} we have not an explicit dependence since, as previously mentioned, we are neglecting constants which can grow exponentially fast with respect to the Mach number. It may be of interest to estimate the dependence on the Mach number in an optimal way. For instance, improving the constants up to $O(M^\beta)$ for some $\beta\geq 1$, would imply that the density may experience a transient growth of order $O(M^{\beta+1}\nu^{-\frac16})$.
	
	\begin{remark}[Restrictions on the Mach number]
		In Theorem \ref{th:NScouetteintro} we have to restrict our analysis to the case of Mach numbers satisfying 	$M\leq  \min \{(\nu+\lambda)^{-\frac12},\nu^{-\frac13}\}$. 	For $\nu,\lambda \ll 1$ the last assumption is not really restrictive since in most physical applications $M\leq 1$ and in the astrophysical context $M\sim 10-50$ \cite{bodo2005spiral,makita2000two}. However, as we explain in Remark \ref{rem:condM}, the condition can be easily relaxed to
		\begin{equation}
			\label{condMachdelta}
			M\leq \min \{(\nu+\lambda)^{-\frac12},\delta^{-1}\nu^{-\frac13}\}
		\end{equation}
		for any $0<\delta\leq 1$, at the price of deteriorating the decay rates by a factor $\delta^{-1}$, namely instead of $e^{-c\nu^\frac13 t}$ one has $e^{-\delta c\nu^\frac13 t}$. The hypothesis $M\leq (\nu+\lambda)^{-\frac12}$, instead, is more rigid. 
	\end{remark}
	\begin{remark}[Absence of shear viscosity]
		If we set $\nu=0$ and $\lambda>0$ the dissipation is present only in the equation for the divergence. This particular case is not immediately covered by Theorem \ref{th:NScouetteintro} since one may infer more properties. We briefly discuss more about this point in Remark \ref{rem:nu0}   
	\end{remark}
	\begin{remark}[Regularity of the initial data]
		In Theorem \ref{th:NScouetteintro}, since we want to combine the inviscid and viscous dynamics, we are not interested in providing sharp regularity assumptions on the initial data. It is indeed natural to trade regularity for time decay in the inviscid problem, as also observed in the incompressible case \cite{bedrossian2019stability}.
	\end{remark}
	In view of the previous remark, we stress that in the viscous case it is not necessary to lose regularity and we are able to infer the following.
	\begin{theorem}
		\label{th:NSnoloss} Let $\nu>0$, $\lambda\geq 0$ and $M>0$ be such that $\nu+\lambda \leq 1/2$ and $M\leq \min \{(\nu+\lambda)^{-\frac12},\nu^{-\frac13}\}$. Assume that $\rho^{in}\in H^{1}(\T\times \R)$, $\alpha^{in},\omega^{in}\in L^2(\T\times \R)$. Then 
		\begin{equation}
			\label{bd:arhoomen}
			\begin{split}
				&\norm{(\alpha-\alpha_0)(t)}_{L^2}+\frac{1}{M}\norm{\nabla(\rho-\rho_0)(t)}_{L^2}+\norm{(\omega-\omega_0)(t)}_{L^2}\\ &\lesssim \nu^{-\frac12}e^{-\frac{\nu^{\frac13}}{64}t}\left(\norm{\alpha^{in}-\alpha_0^{in}}_{L^2}+\frac{1}{M}\norm{\nabla (\rho^{in}-\rho^{in}_0)}_{L^2}+\norm{(\omega^{in}-\omega^{in}_0)}_{L^2}\right).
			\end{split}
		\end{equation}
		In addition,  the following inequality holds
		\begin{equation}
			\label{bd:vrhoen}
			\begin{split}
				&\norm{(\vv-\vv_0)(t)}_{L^2}+\frac{1}{M}\norm{(\rho-\rho_0)(t)}_{L^2}\\
				&\lesssim \nu^{-\frac16}e^{-\frac{\nu^{\frac13}}{64}t}\left(\norm{\alpha^{in}-\alpha_0^{in}}_{L^2}+\frac{1}{M}\norm{\nabla (\rho^{in}-\rho^{in}_0)}_{L^2}+\norm{(\omega^{in}-\omega^{in}_0)}_{L^2}\right)
			\end{split}
		\end{equation}
	\end{theorem}
	In the bound \eqref{bd:arhoomen}, at the price of having worst estimates with respect to the one in Theorem \ref{th:NScouetteintro}, we see that we do not lose derivatives to get the exponential decay for the quantities on the left-hand side. Then, the bound \eqref{bd:vrhoen} does not straightforwardly follow by \eqref{bd:arhoomen}. It is indeed a consequence of a careful choice of some Fourier multipliers used to prove \eqref{bd:arhoomen}. In addition, since $\langle t \rangle^{\frac12}\lesssim \nu^{-\frac16}\exp((\nu^{\frac13}t)/64)$, the bound \eqref{bd:vrhoen} agrees with \eqref{bd:compenh} in terms of order of magnitude of the maximal possible growth. The estimates and the method of proof of Theorem \ref{th:NSnoloss} can be useful to extend this linear result to prove a transition threshold in Sobolev spaces for the fully nonlinear case, which we aim at studying.
  \subsubsection*{Strategy of proof} 	We now comment about the strategy of proof of Theorem \ref{th:NScouetteintro} and Theorem \ref{th:NSnoloss}. Being similar, we outline here the main ideas in both cases. 
	
	When  viscosity is present, we have to overcome two main difficulties. First of all, as can be seen by summing up \eqref{eq:contcouetteint} and \eqref{eq:vorticitycouetteint}, the conservation of $\rho+\omega$ along the Couette flow no longer holds, which is a crucial point in the inviscid case. Therefore, we cannot reduce the analysis to the study of a $2\times 2 $ system. The second point is that, since we do not have a dissipative term in \eqref{eq:contcouetteint}, it is not a priori trivial to have decay for the density. However, we will be able to recover the exponential decay via a weighted energy estimate where it is crucial to exploit the coupling between $\rho$ and $\alpha$,
	
	As done in the inviscid case, we first remove the transport terms via the standard change of coordinates and we perform the Fourier transform in both space variables, leading us to the study of a $3\times3$ system in the Fourier space. It is then crucial to replace the vorticity with another auxiliary quantity, i.e. $\rho+\omega-\nu M^2\alpha$, which satisfy a more complicated equation with respect to $\omega$ but has a better structure to make use of this variable in energy estimates. Then, we are able to define a weighted energy functional in terms of $(\rho,\alpha, \rho+\omega-\nu M^2\alpha)$ for which we can infer a Gr\"onwall's type estimate. The weights are suitable time-dependent Fourier multipliers. The main difference between the proof of Theorem \ref{th:NScouetteintro} and Theorem \ref{th:NSnoloss} is the choice of the weights.
	\medskip
	
	\subsubsection*{Outline of the paper}
	We begin our analysis with the study of the dynamics of the $x$-averages in Section \ref{sec:zeromode}. In Section \ref{sec:invcompcouette} we consider the inviscid problem in order to prove Theorem \ref{th:NScouetteintro} when $\nu=\lambda=0$. In Section \ref{sec:visccompcouette} we turn our attention to the viscous case. Here, we first prove Theorem \ref{th:NScouetteintro} in Subsection \ref{subsec:thNS1}. Then, in Subsection \ref{subsec:thNS2} we present the proof of Theorem \ref{th:NSnoloss}.
	\subsubsection*{Notations}
	In this paper, when using the symbol $\lesssim$ we are neglecting constants which do not depend on $\nu$ but may depend on $(1+M)^\beta$ or $\exp(M^\beta)$ for some $\beta\geq 1$. However, we keep track of constants which goes to zero as $M\to 0$.
	
	When it will be clear from the context whether we are working in the physical space or in the frequency space, by an abuse of notation, we will not distinguish between pseudo-differential operators and their own symbols.
	
	We denote the Fourier transform as 
	\begin{align*}
		&\widehat{f}(k,\eta)=\frac{1}{2\pi}\iint_{\T\times\mathbb{R}}e^{-i(kx+\eta y)}f(x,y)dxdy, \\
		&f(x,y)=\frac{1}{2\pi}\sum_{k \in \mathbb{Z}}\int_\mathbb{R}e^{i(kx+\eta y)}\widehat{f}(k,\eta)d\eta,
	\end{align*}
	We say that $f\in H^{s_1}_xH^{s_2}_y$ whenever
	\begin{equation*}
		\norm{f}_{H^{s_1}_xH^{s_2}_y}^2=\sum_k\int \langle k\rangle^{2s_1}\langle \eta \rangle^{2s_2} |\hat{f}|^2(k,\eta)d\eta< +\infty,
	\end{equation*}
	whereas we denote the norm in the usual $H^s(\T\times \R)$ space as
	\begin{equation*}
		\norm{f}_{H^{s}}^2=\sum_k\int \langle k,\eta\rangle^{2s} |\hat{f}|^2(k,\eta)d\eta.
	\end{equation*}

	Let $Z(t)=(Z_1(t),Z_2(t))^T:[t_0,+\infty)\times\mathbb{C}^{2}\to \mathbb{C}^2$ and $\mathcal{L}(t): [t_0,+\infty)\times\mathbb{C}^{2\times 2}\to \mathbb{C}^{2\times 2}$. Given the following 2D non-autonomous dynamical system
	\begin{equation}
		\label{eqGamma}
		\Dt Z=\mathcal{L}(t)Z,
	\end{equation}
	we define the standard Picard iteration 
	\begin{equation}
		\label{def:mhiGamma}\begin{split}
			&\Phi_\mathcal{L}(t,t_0)=\mathbbm{1}+\sum_{n=1}^{\infty}\mathcal{I}_n(t,t_0),\\
			&\mathcal{I}_{n+1}(t,t_0)=\int_{t_0}^{t}\mathcal{L}(\tau)\mathcal{I}_n(\tau,t_0)d\tau,\qquad
			\mathcal{I}_1(t,t_0)=\int_{t_0}^t\mathcal{L}(\tau)d\tau.
		\end{split}
	\end{equation}
	$\Phi_\mathcal{L}$ is the solution operator associated to $\mathcal{L}$. In particular it satisfies the group property, $\Phi_\mathcal{L}(t,t_0)=\Phi_\mathcal{L}(t,s)\Phi_\mathcal{L}(s,t_0)$ for any $t,s\geq0$.
	
	In order to remove the transport term from the equations, we will always make the following change of coordinates
	\begin{equation}
		\label{def:movframe}
		X=x-yt,\quad 	     Y=y.
	\end{equation}
	In particular, the differential operators change as follows
	\begin{equation}
		\label{def:DeltaL}
		\begin{split}
			&\dx=\dX,\\ &\dy = \dY-t\dX,\\     
			&\Delta= \Delta_L:=\dXX+(\dY-t\dX)^2.
		\end{split}
	\end{equation}
	In the new reference frame, which we shall often refer to as the \textit{moving frame}, we also define the functions 
	\begin{equation}
		\label{def:RAO}
		\begin{split}
			R(t,X,Y)&=\rho(t,X+tY,Y),\\
			A(t,X,Y)&=\alpha(t,X+tY,Y),\\
			\Omega(t,X,Y)&=\omega(t,X+tY,Y).
		\end{split}
	\end{equation}
	We denote the symbol associated to $-\Delta_L$ as
	\begin{align}
		\label{def:sigma}
		\p(t,k,\eta)&=k^2+(\eta-kt)^2.
	\end{align}  
	Moreover
	\begin{equation}
		\label{def:p'}
		(\dt p)(t,k,\eta)=-2k(\eta-kt)
	\end{equation}
	is the symbol associated to the operator  $2\dX(\dY-t\dX)$.
	\section{Dynamics of the $k=0$ modes} \label{sec:zeromode}
	
	In this section, we investigate in detail the dynamics of the $x$-averages of the perturbations. Due to the structure of the shear flow and the fact that the equations are
	linear, it is clear that the zero mode in $x$ has an independent dynamics with respect to other modes. Consequently, in our analysis we can decouple the evolution of the $k = 0$ mode from the rest of the perturbation.
	
	The system \eqref{eq:contcouetteint}-\eqref{eq:divcouetteint} when projected onto the $k=0$ frequency, recalling \eqref{def:f0}, read as follows
	\begin{align}
		\label{eq:rho0} &\dt \rho_0=-\alpha_0,\\
		\label{eq:alpha0}&\dt \alpha_0=(\nu+\lambda)\dyy \alpha_0-\frac{1}{M^2}\dyy \rho_0,\\
		\label{eq:om0}&\dt \omega_0=\alpha_0+\nu \dyy \omega_{0}.
	\end{align}
	Since $v_0^y$ satisfy the same equation of $\alpha_0=\dy v_0^y$,  we will identify $\dy^{-1}\alpha_0=v_0^y$.
	
	When $\nu=\lambda=0$ the dynamics of $(\rho_0,\alpha_0)$ is given by a standard $1$-D wave equation, namely 
	\begin{equation}
		\label{eq:waverho0}
		\dtt \rho_0-M^{-2}\dyy\rho_0=0, \qquad \text{in }\R,
	\end{equation}
	and by adding \eqref{eq:rho0} to \eqref{eq:om0} we get
	\begin{equation*}
		\dt(\rho_0+\omega_0)=0,
	\end{equation*} 
	hence $\omega_0=\rho^{in}_0+\omega^{in}_0-\rho_0$. Therefore, when $\nu=\lambda=0$ the dynamics of the $k=0$ mode is completely determined by solving \eqref{eq:waverho0}. By the explicit representation formula for \eqref{eq:waverho0}, we know that there is not decay for the zero modes. 
	
	When $\nu>0$, one has an explicit representation of the solution in the Fourier space. However, from this formula it is not immediate to infer decay properties of solutions to \eqref{eq:rho0}-\eqref{eq:om0}. We are then going to derive decay properties of the $k=0$ by using an energy method in a similar way to what was done by Guo and Wang in \cite{guo2012decay}. In order to perform the energy estimate, it will be convenient to replace the equation \eqref{eq:om0} with 
	\begin{equation}
		\label{eq:Xi0}
		\dt(\rho_0+\omega_0-\nu M^2\alpha_0)=\nu\dyy(\rho_0+\omega_0-\nu M^2\alpha_0)-\lambda\nu M^2\dyy \alpha_0.
	\end{equation} 
	In particular, we have the following.
	\begin{theorem}
		\label{th:zeromode}
		Let $\nu,\lambda \geq0$ and $\rho^{in},\alpha^{in},\omega^{in}$ be the initial data of \eqref{eq:contcouetteint}-\eqref{eq:vorticitycouetteint}. Then, the solution $(\rho,\alpha,\omega)$ can be decomposed as $\rho=\rho_0+\rho_{\neq}, \ \alpha=\alpha_0+\alpha_{\neq}, \ \omega=\omega_0+\omega_{\neq}$ where $(\rho_{\neq},\alpha_{\neq},\omega_{\neq})$ satisfy \eqref{eq:contcouetteint}-\eqref{eq:vorticitycouetteint} and $(\rho_{0},\alpha_0,\omega_0)$ satisfy \eqref{eq:rho0}-\eqref{eq:om0}. For the $k=0$ mode we have the following: for any $\ell \geq 0$, let 
		\begin{equation}
			\label{def:El0}
			\begin{split}
				\mathcal{E}^\ell(t)=&\norm{\dy^\ell \alpha_0(t)}^2_{L^2}+\norm{\dy^{\ell-1} \alpha_0(t)}^2_{L^2}+\norm{\dy^\ell(\omega_0+\rho_0-\nu M^2\alpha_0)}_{L^2}^2\\
				&+\frac{1}{M^2}(\norm{\dy^{\ell+1}\rho_0(t)}_{L^2}^2+\norm{\dy^{\ell}\rho_0(t)}_{L^2}^2),	
			\end{split}
		\end{equation}
		where $\dy^{-1}\alpha_0=v_0^y$. If $M(\nu+\lambda)^{\frac12}\leq 1$ and $\mathcal{E}^\ell_{in},\mathcal{E}^0_{in}<+\infty$ then 
		\begin{align}
			\label{bd:Eldecay} \mathcal{E}^\ell(t)\leq&\frac{4\mathcal{E}^\ell_{in}}{(\nu C_{in}^\ell t+1)^{\ell}}, 
		\end{align}
		where $C_{in}^0=0$ and $C_{in}^\ell=C\max\{1,(\mathcal{E}^\ell_{in}/\mathcal{E}^0_{in})^{\frac{1}{\ell}}\}$ for $\ell\geq 1$ and some constant $C$ which does not depend on $\ell,\nu,\lambda$.
		In addition we have that 
		\begin{equation}
			\label{eq:zerozero}
			\rho_0^{in}=\alpha_0^{in}=\omega_0^{in}=0 \Longrightarrow \rho_0(t)=\alpha_0(t)=\omega_0(t)=0.
		\end{equation}
	\end{theorem}
	\begin{remark}\label{rem:nozero}
		In view of the theorem above, it is equivalent to study the dynamics of $(\rho-\rho_0,\alpha-\alpha_0,\omega-\omega_0)$ or $(\rho,\alpha,\omega)$ assuming that $\rho_{0}^{in}=\alpha_{0}^{in}=\omega_{0}^{in}=0$. In the rest of the paper, for simplicity of notation, we will always consider the second case.
	\end{remark}
	\begin{remark}
		For any $N\geq 0$, from the previous theorem we infer that 
	\begin{align}
		\norm{\alpha_0(t)}_{H^N}+\frac{1}{M}\norm{\dy \rho_0(t)}_{H^N}\lesssim& \frac{\sqrt{\mathcal{E}^N_{in}}}{(1+\nu t)^\frac12},\\
		\norm{\omega_{0}(t)}_{H^N}\lesssim &\sqrt{\mathcal{E}^N_{in}}.
	\end{align}
	Hence, $\alpha_0$ and $\dy \rho_0$ have the same decay as if $\rho_0, v_0^y$ had satisfied the standard $1$-$D$ heat equation, see for instance \cite[Theorem 1.1]{guo2012decay}.
	\end{remark}
	
	We now present the proof of Theorem \ref{th:zeromode}.
	\begin{proof}
		First of all, $(\rho_{\neq},\alpha_{\neq},\omega_{\neq})$ satisfy \eqref{eq:contcouetteint}-\eqref{eq:vorticitycouetteint} since $\dx (f_0)=0$. The proof of \eqref{eq:zerozero} follows by the linearity of the system \eqref{eq:rho0}-\eqref{eq:om0}.  To prove \eqref{bd:Eldecay}, we define 
		\begin{align*}
			E^\ell(t)=\frac12 \left(\mathcal{E}^\ell(t)-\frac{(\nu+\lambda)}{2} \jap{\dy^\ell \rho_0(t),\dy^\ell \alpha_0(t)}\right).
		\end{align*}
		Since $M(\nu+\lambda)\leq1$ we have 
		\begin{equation}
			\label{bd:coerc0}
			\frac14\mathcal{E}^\ell(t)\leq E^\ell(t)\leq \mathcal{E}^\ell(t),
		\end{equation}
		namely the functional $E^\ell$ is coercive. Then, by a direct computation we get 
		\begin{align*}
			\Dt E^\ell(t)&+(\nu+\lambda)\left(\norm{\dy^{\ell+1}\alpha_0}^2_{L^2}+\norm{\dy^{\ell}\alpha_0}^2_{L^2}+\frac{1}{4M^2}\norm{\dy^{\ell+1}\rho_0}^2_{L^2}\right)\\
			&+\nu \norm{\dy^{\ell+1}(\rho_0+\omega_0-\nu M^2\alpha_0)}^2_{L^2}\\
			=&\frac{\nu+\lambda}{4}\norm{\dy^\ell \alpha_0}^2_{L^2}+\frac{(\nu+\lambda)^2}{4}\jap{\dy^{\ell+1}\alpha_0,\dy^{\ell+1}\rho_0}\\
			&+\lambda \nu M^2\jap{\dy^{\ell+1}\alpha_0,\dy^{\ell+1}(\rho_0+\omega_0-\nu M^2\alpha_0)}.
		\end{align*}
		Using again that $M^2(\nu+\lambda)\leq 1$, we have 
		\begin{align*}
			\frac{(\nu+\lambda)^2}{4}|\jap{\dy^{\ell+1}\alpha_0,\dy^{\ell+1}\rho_0}|\leq& \frac{(\nu+\lambda)^2}{8} \norm{\dy^{\ell+1}\alpha_0}^2_{L^2}+M^2(\nu+\lambda)\frac{(\nu+\lambda)}{8M^2}\norm{\dy^{\ell+1}\rho_0 }^2_{L^2}\\
			\leq&\frac{(\nu+\lambda)}{8}  \norm{\dy^{\ell+1}\alpha_0}^2_{L^2}+\frac{(\nu+\lambda)}{8M^2}\norm{\dy^{\ell+1}\rho_0 }^2_{L^2},
		\end{align*}
		and 
		\begin{align*}
			\lambda \nu M^2|\jap{\dy^{\ell+1}\alpha_0,\dy^{\ell+1}(\rho_0+\omega_0-\nu M^2\alpha_0)}|\leq& \frac{\nu}{2}\norm{\dy^{\ell+1}\alpha_0}^2_{L^2}\\
			&+\frac{\nu}{2}\norm{\dy^{\ell+1}(\rho_0+\omega_0-\nu M^2\alpha_0)}^2_{L^2}.
		\end{align*}
		Consequently we infer 
		\begin{equation}
			\label{bd:El1}
			\begin{split}
				\Dt E^\ell(t)&+\frac{(\nu+\lambda)}{8}\left(\norm{\dy^{\ell+1}\alpha_0}^2_{L^2}+\norm{\dy^{\ell}\alpha_0}^2_{L^2}+\frac{1}{M^2}\norm{\dy^{\ell+1}\rho_0}^2_{L^2}\right)\\
				&+\frac{\nu}{2}\norm{\dy^{\ell+1}(\rho_0+\omega_0-\nu M^2\alpha_0)}^2_{L^2}\leq 0.
			\end{split}
		\end{equation}
		Therefore, by combining \eqref{bd:coerc0} with \eqref{bd:El1} we prove that for any $\ell\geq 0$
		\begin{equation}
			\label{bd:Elmon}\mathcal{E}^\ell(t)\lesssim \mathcal{E}^\ell_{in}.
		\end{equation}
		
		To prove \eqref{bd:Eldecay}, we need to reconstruct  some power of the energy functional by providing lower bounds for the positive terms appearing in \eqref{bd:El1}. Hence, we first recall the following interpolation inequality, see \cite[Lemma A.4]{guo2012decay}, 
		\begin{equation}
			\label{bd:trivinter}
			\norm{\dy^\ell f}_{L^2}\leq \norm{\dy^{\ell+1}f}_{L^2}^{\frac{\ell}{\ell+1}}\norm{f}_{L^2}^{\frac{1}{\ell+1}}.
		\end{equation}
	In addition, by \eqref{bd:Elmon} we know that 
		\begin{equation}
			\label{bd:rho0al0l2}
			\norm{\alpha_0}_{L^2}^2+\norm{\dy^{-1}\alpha_0}_{L^2}^2+\frac{1}{M^2}\norm{\rho_0}_{L^2}^2+\norm{\rho_0+\omega_0-\nu M^2\alpha_0}_{L^2}^2\lesssim \mathcal{E}^0_{in}.
		\end{equation}
		Therefore, for $\ell \geq 1$ from \eqref{bd:trivinter} and \eqref{bd:rho0al0l2} we get 
		\begin{align}
			\norm{\dy^{\ell+1}\alpha_0}_{L^2}+\frac{1}{M^2}\norm{\dy^{\ell+1}\rho_0}_{L^2}&\gtrsim  \norm{\dy^\ell \alpha_0}_{L^2}^{1+\frac{1}{\ell}}\norm{\alpha_0}^{-\frac{1}{\ell}}_{L^2}+\frac{1}{M^2}\norm{\dy^\ell \rho_0}_{L^2}^{1+\frac{1}{\ell}}\norm{\rho_0}^{-\frac{1}{\ell}}_{L^2}\\
			\gtrsim & \left(\norm{\dy^\ell \alpha_0}_{L^2}^{1+\frac{1}{\ell}}+(\frac{1}{M^2}\norm{\dy^\ell \rho_{0}}_{L^2})^{1+\frac{1}{\ell}}\right)(\mathcal{E}^0_{in})^{-\frac{1}{2\ell}}.
			\label{bd:a0r0ell}
		\end{align}
		Similarly we have 
		\begin{align}
			\label{bd:a0-1}
			&\norm{\dy^{\ell}\alpha_0}_{L^2}=\norm{\dy^{\ell+1}(\dy^{-1}\alpha_0)}_{L^2}\gtrsim\norm{\dy^{\ell-1} \alpha_0}_{L^2}^{1+\frac{1}{\ell}}(\mathcal{E}^0_{in})^{-\frac{1}{2\ell}},\\
			\label{bd:Xi0}
			&\norm{\dy^{\ell+1}(\rho_0+\omega_0-\nu M^2\alpha_0)}_{L^2}\gtrsim\norm{\dy^{\ell}(\rho_0+\omega_0-\nu M^2\alpha_0)}_{L^2}^{1+\frac{1}{\ell}}(\mathcal{E}^0_{in})^{-\frac{1}{2\ell}}. 
		\end{align}
		In account of \eqref{bd:Elmon}, we observe also that 
		\begin{equation}
			\label{bd:rhoell+1}
			\begin{split}
				\frac{1}{M^2}\norm{\dy^{\ell+1}\rho_0}_{L^2}=&\frac{1}{M^2}\norm{\dy^{\ell+1}\rho_0}_{L^2}^{1+\frac{1}{\ell}}\norm{\dy^{\ell+1}\rho_0}_{L^2}^{-\frac{1}{\ell}}\\
				\gtrsim& (\frac{1}{M^2}\norm{\dy^{\ell+1}\rho_0}_{L^2})^{1+\frac{1}{\ell}}(\mathcal{E}^\ell_{in})^{-\frac{1}{2\ell}}.
			\end{split}
		\end{equation} 
		In particular, by combining the estimates \eqref{bd:a0r0ell}-\eqref{bd:rhoell+1} we infer 
		\begin{align*}
			\norm{\dy^{\ell+1}\alpha_0}^2_{L^2}&+\norm{\dy^{\ell}\alpha_0}^2_{L^2}+\frac{1}{M^2}\norm{\dy^{\ell+1}\rho_0}^2_{L^2}+\norm{\dy^{\ell+1}(\rho_0+\omega_0-\nu M^2\alpha_0)}^2_{L^2}\\
			\gtrsim& \left(\norm{\dy^{\ell}\alpha_0}^2_{L^2}+\norm{\dy^{\ell-1}\alpha_0}^2_{L^2}+\frac{1}{2M^2}\norm{\dy^{\ell}\rho_0}^2_{L^2}\right)^{1+\frac{1}{\ell}}(\mathcal{E}^0_{in})^{-\frac{1}{\ell}}\\
			&+\norm{\dy^{\ell}(\rho_0+\omega_0-\nu M^2\alpha_0)}_{L^2}^{1+\frac{1}{\ell}}(\mathcal{E}^0_{in})^{-\frac{1}{2\ell}}+(\frac{1}{2M^2}\norm{\dy^{\ell+1}\rho_0}_{L^2}^2)^{1+\frac{1}{\ell}}(\mathcal{E}^\ell_{in})^{-\frac{1}{\ell}}\\
			\gtrsim& (\mathcal{E}^\ell(t))^{1+\frac{1}{\ell}}(\max\{\mathcal{E}^\ell_{in},\mathcal{E}^0_{in}\})^{-\frac{1}{\ell}}.
		\end{align*}
		Consequently, appealing to \eqref{bd:coerc0}, combining the bound above with \eqref{bd:El1} we have 
		\begin{equation}
			\Dt E^\ell(t)+\nu C (\max\{\mathcal{E}^\ell_{in},\mathcal{E}^0_{in}\})^{-\frac{1}{\ell}}E^\ell(t)^{1+\frac{1}{\ell}}\leq 0,
		\end{equation}
		where $C$ is a constant independent of $\ell$. Hence, from Gr\"onwall's Lemma we get
		\begin{equation}
			E^\ell(t)\leq E^\ell_{in}(\nu\widetilde{C}_{in}^\ell t+ 1)^{-\ell},
		\end{equation}
		where $\widetilde{C}_{in}^{\ell}=C(E^\ell_{in})^{\frac{1}{\ell}}(\max\{\mathcal{E}^\ell_{in},\mathcal{E}^0_{in}\})^{-\frac{1}{\ell}}$, whence proving \eqref{bd:Eldecay} in view of \eqref{bd:coerc0}.
	\end{proof}
	\section{The inviscid case} \label{sec:invcompcouette}
	In this section we investigate in detail the inviscid case. We are going to prove the results stated in Theorem \ref{th:NScouetteintro} when $\nu=\lambda=0$, for which it is convenient to treat separately the analysis for the upper and lower bounds, respectively given in Theorem \ref{maintheorem} and  Theorem \ref{maintheoremlwz}.  As observed in Remark \ref{rem:nozero}, we can remove the $x$-average from the dynamics, so we will prove the results only for initial perturbations without the $k=0$ mode, namely ${\rho}_{0}^{in}={\alpha}_{0}^{in}={\omega}_{0}^{in}=0$. 
	
	To proceed with the analysis of the system \eqref{eq:contcouetteint}-\eqref{eq:vorticitycouetteint}, we eliminate the transport term with the change of coordinates \eqref{def:movframe} and we use the notation defined in \eqref{def:DeltaL}-\eqref{def:RAO}. 
	
	By adding \eqref{eq:contcouetteint} to \eqref{eq:vorticitycouetteint}, we find out that $\rho+\omega$ is transported by the Couette flow. Hence, defining 
	\begin{equation}
		\label{def:Xi}
		\Xi(t,X,Y):=R(t,X,Y)+\Omega(t,X,Y),
	\end{equation}
	we have that $\dt \Xi=0$. Consequently 
	\begin{equation}
		\label{eq:decOmega}
		\Omega(t,X,Y)=\Xi^{in}(X,Y)-R(t,X,Y),
	\end{equation}
	where $\Xi^{in}=\omega^{in}+\rho^{in}$. In view of \eqref{eq:v2}, we also have 
	\begin{align*}
		V^y=&(\dY-t\dX)\Delta_L^{-1}A+\dX\Delta_L^{-1}\Omega \\
		=&(\dY-t\dX)\Delta_L^{-1}A+\dX\Delta_L^{-1}\Xi^{in}-\dX\Delta_L^{-1}R.
	\end{align*}
	We can thus rewrite the system \eqref{eq:contcouetteint}-\eqref{eq:vorticitycouetteint} in the moving frame only in terms of $A$ and $R$ as follows 
	\begin{align}
		\label{eq:R}\dt R=&-A,\\
		\begin{split}
			\label{eq:A}\dt A=&-2\dX(\dY-t\dX)(\Delta_L^{-1})A+\left(-\frac{1}{\M^2}\Delta_L+2\dXX(\Delta_L^{-1})\right)R\\ &-2\dXX(\Delta_L^{-1})\Xi^{in}.
		\end{split}
	\end{align}
	We stress again the importance of the identity \eqref{eq:decOmega}, which not only allow us to study the system in terms of density and divergence but also relates compressible and incompressible effects, see Section \ref{rem:potvort}.
	
	In view of the particular choice of the domain, it is now natural to perform the analysis in the Fourier space.
	\subsection{Analysis in the Fourier space}
	We first take the Fourier transform in both space variables of the system \eqref{eq:R}-\eqref{eq:A}, which become a non-autonomous $2\times 2$ dynamical system at each fixed frequency $(k,\eta)$. Then, by properly weighting the density and the divergence we characterize the solution operator of the associated homogeneous problem, i.e. $\Xi^{in}=0$, which is a key point in order to prove Theorem \ref{th:NScouetteintro}.
	
	By taking the Fourier transform of the system \eqref{eq:R}-\eqref{eq:A}, recalling the notation introduced in \eqref{def:sigma}-\eqref{def:p'}, we get that 
	\begin{align}
		\label{eq:hR0} &\dt \hR=-\hA\\
		\label{eq:hA0}&\dt \hA=\frac{\dt \p}{\p}\hA+\bigg(\frac{\p}{\M^2}+\frac{2k^2}{\p}\bigg)\hR-\frac{2k^2}{\p}\hXi^{in}.
	\end{align}
	Since in what follows we consider $k, \eta$ as fixed parameters, we will omit their dependence for the quantities under study.  In Figure \ref{fig:plotinvcomp1} we show some numerical simulations of the system above. 
	\begin{remark}[Transient decay]
		\label{rem:trandec}
		From \eqref{eq:hA0}, since for $t<\eta/k$ one has $\dt p<0$, the first term in the right-hand side of \eqref{eq:hA0} acts as a damping term for $\widehat{A}$. Instead, $\dt p>0$ for $t>\eta/k$, hence it induces a growth on $\widehat{A}$. In the incompressible case, the velocity may experience a transient growth,  here, we see that the divergence may have a \textit{transient decay}, see Figures \ref{fig:transAM1} and \ref{fig:transAM50}. To balance the growth generated by this term we need to properly weight $\hR$ and $\hA$. 
	\end{remark}
	\begin{figure}[t]
		\centering 
		\begin{subfigure}{0.32\textwidth}
			\includegraphics[width=\linewidth]{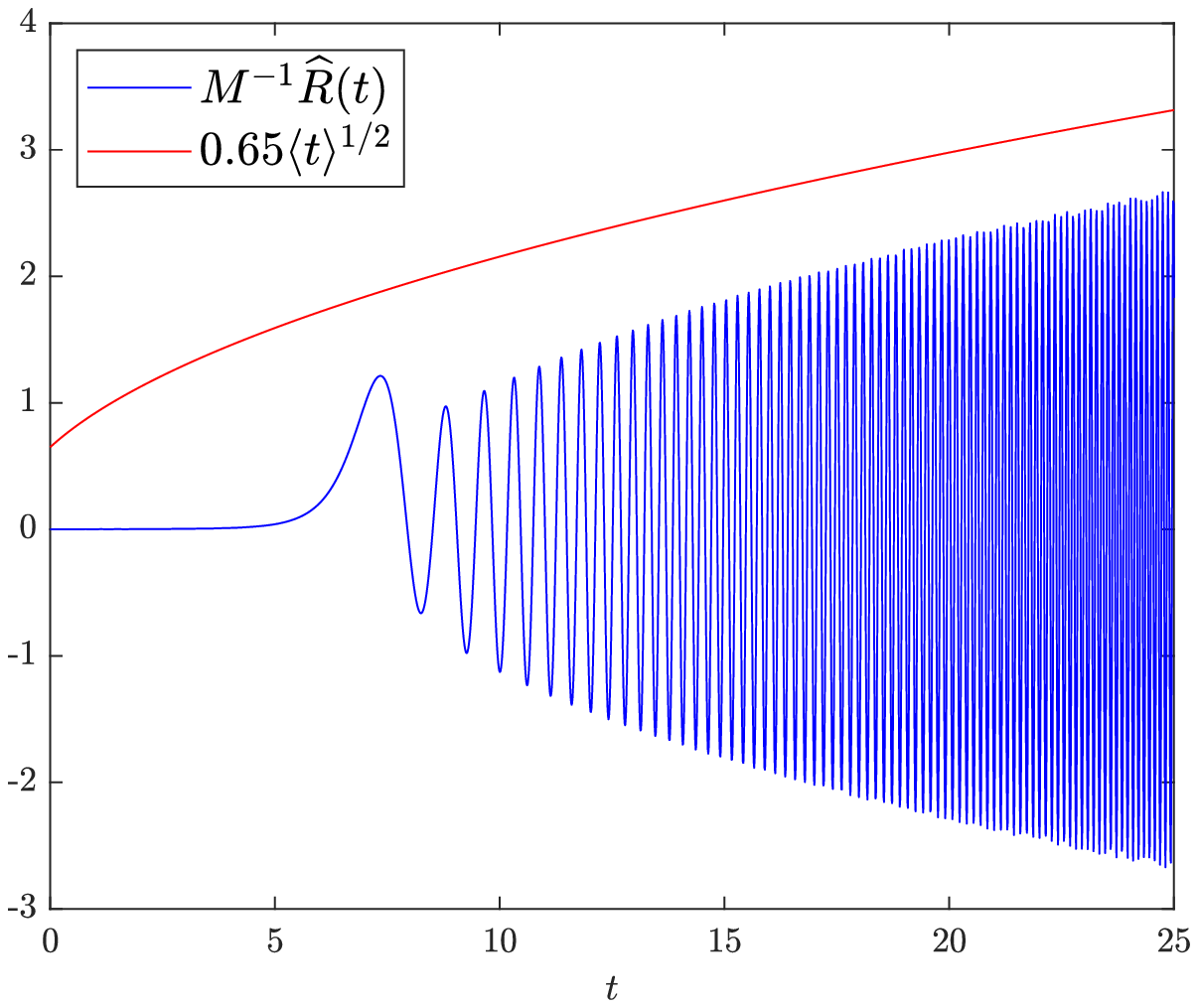}
			\caption{$M=1$}
			\label{fig:RM1}
		\end{subfigure}\hfil 
		\begin{subfigure}{0.32\textwidth}
			\includegraphics[width=\linewidth]{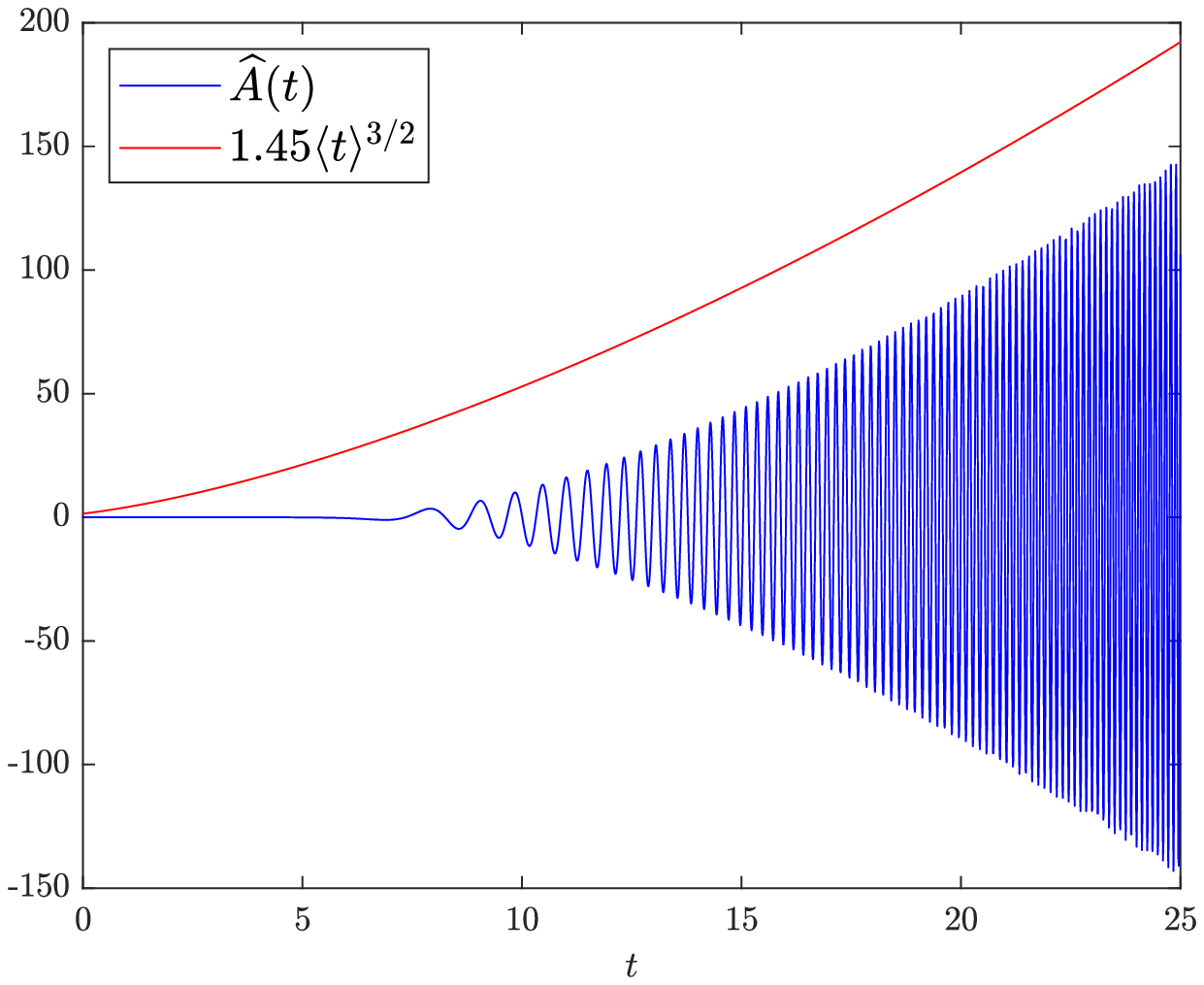}
			\caption{$M=1$}
			\label{fig:AM1}
		\end{subfigure}\hfil 
		\begin{subfigure}{0.32\textwidth}
			\includegraphics[width=\linewidth]{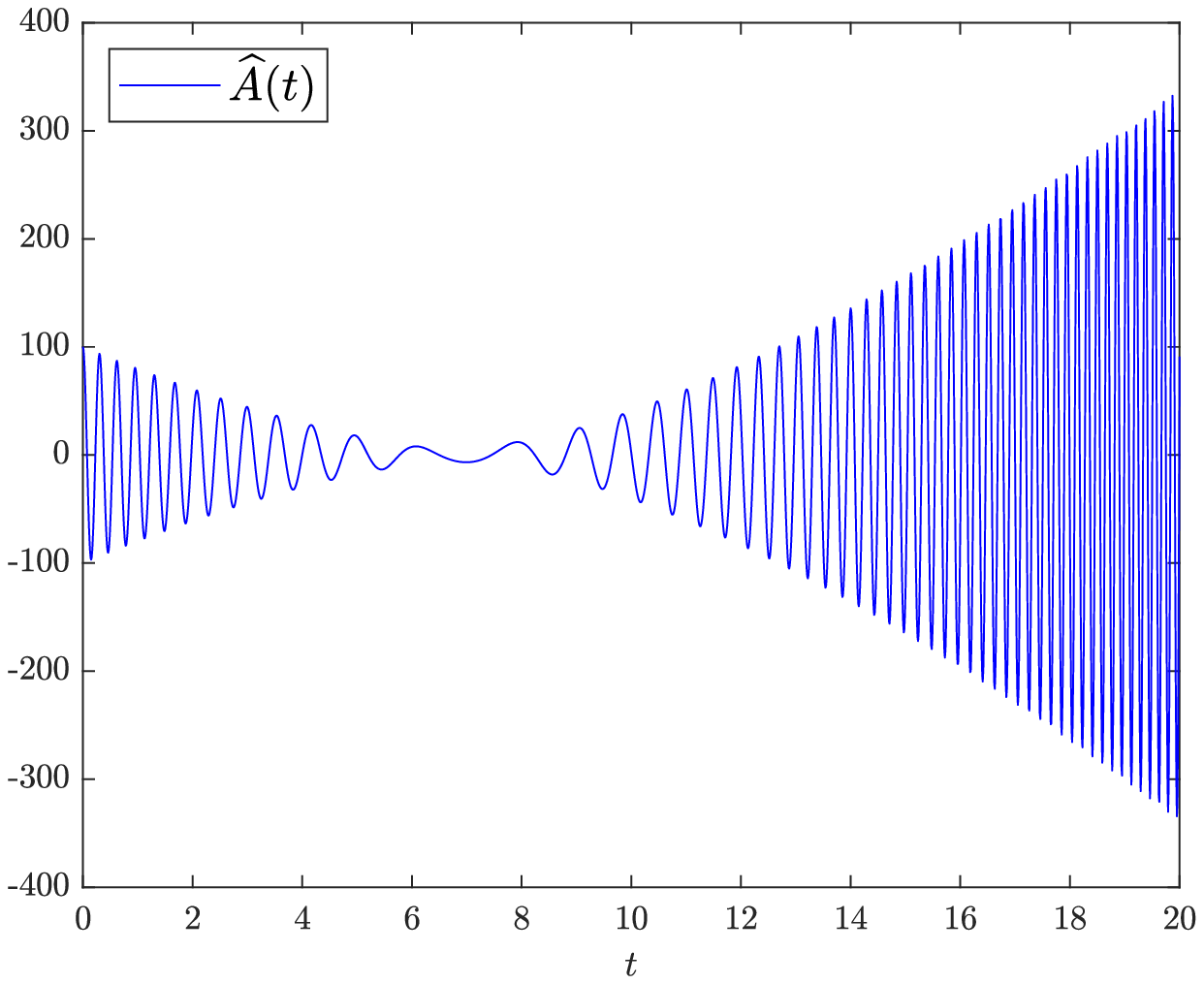}
			\caption{$M=1$}
			\label{fig:transAM1}
		\end{subfigure}
		\medskip
		\centering 
		\begin{subfigure}{0.32\textwidth}
			\includegraphics[width=\linewidth]{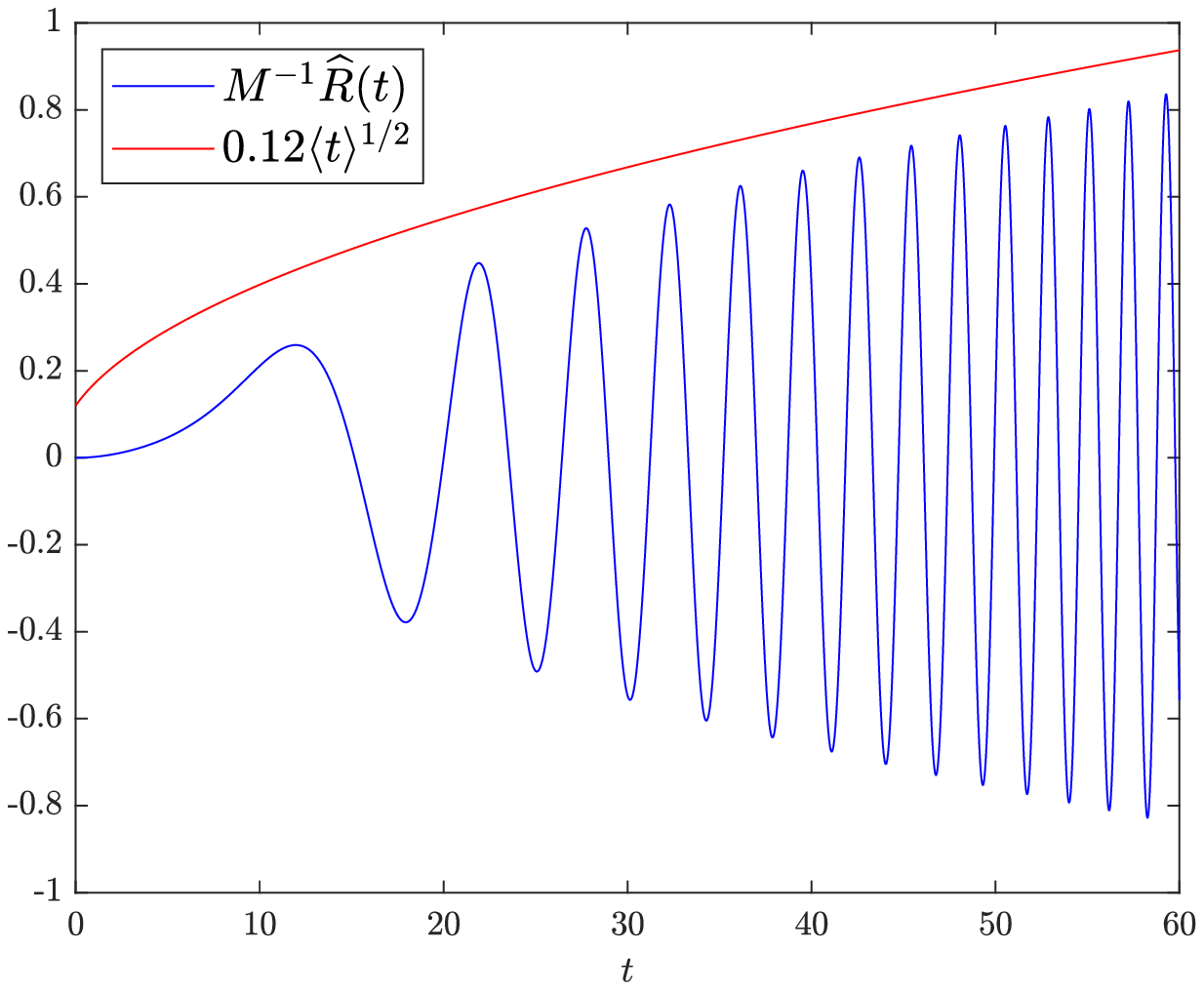}
			\caption{$M=50$}
			\label{fig:RM50}
		\end{subfigure}\hfil 
		\begin{subfigure}{0.32\textwidth}
			\includegraphics[width=\linewidth]{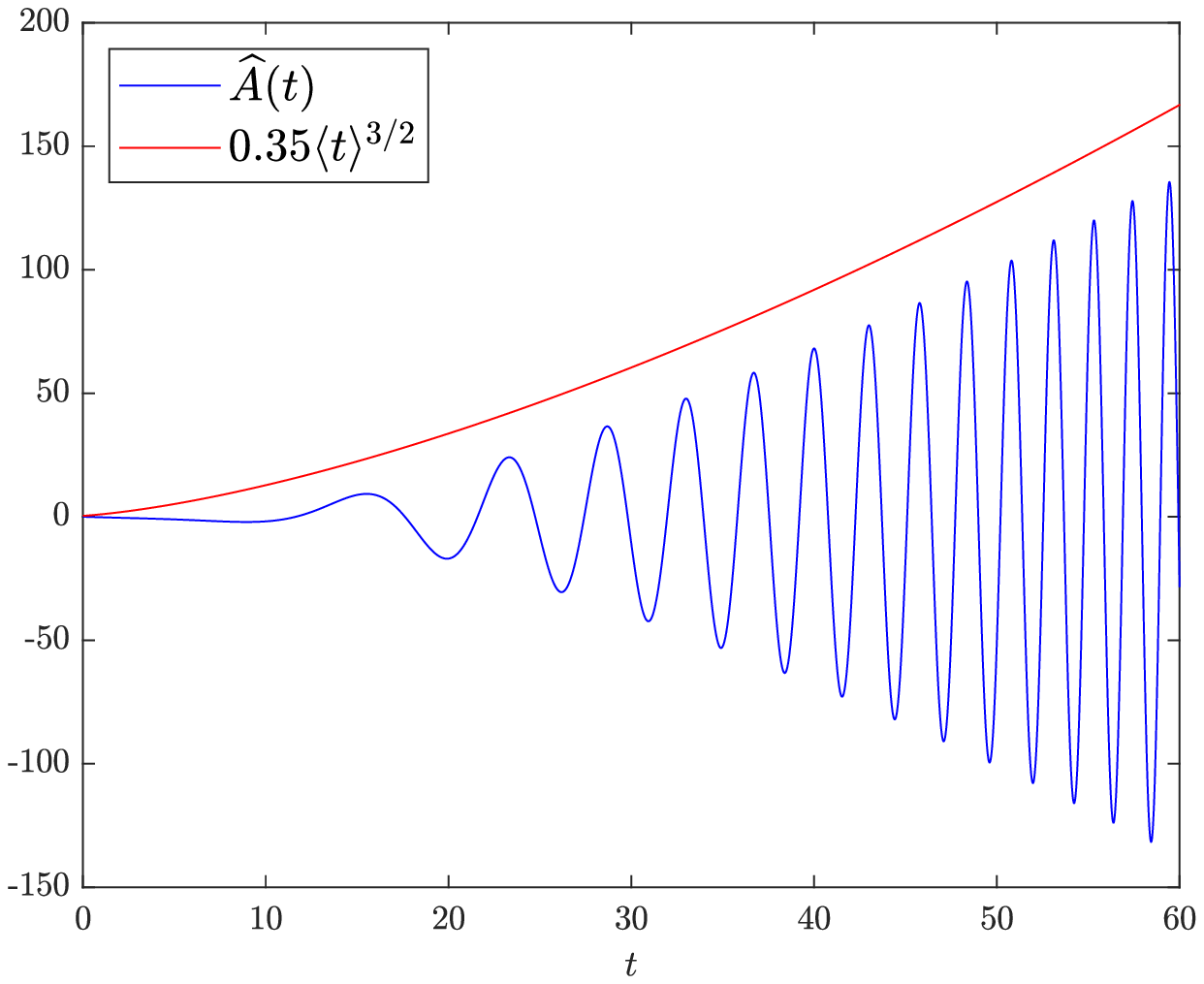}
			\caption{$M=50$}
			\label{fig:AM50}
		\end{subfigure}\hfil 
		\begin{subfigure}{0.32\textwidth}
			\includegraphics[width=\linewidth]{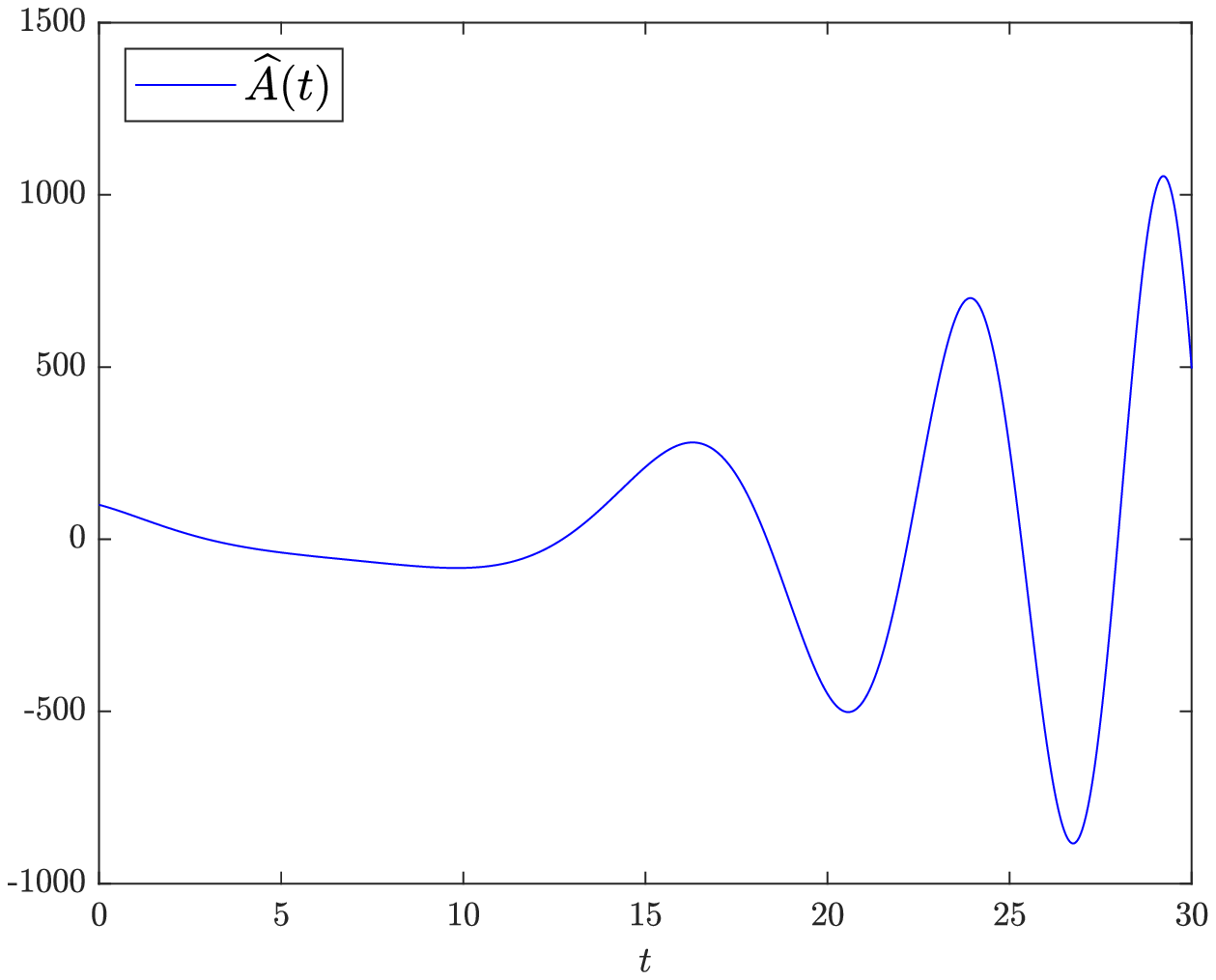}
			\caption{$M=50$}
			\label{fig:transAM50}
		\end{subfigure}
		\caption{Numerical simulations of the system \eqref{eq:hR0}-\eqref{eq:hA0} at fixed frequencies $k=3, \eta=21$ for different values of the Mach number. The red lines are the expected asymptotic behaviours. In the figures \ref{fig:RM1}, \ref{fig:AM1}, \ref{fig:RM50} and \ref{fig:AM50} we consider  $\widehat{R}^{in}=\widehat{A}^{in}=0$ and $\widehat{\Xi}^{in}=5$. In the figures \ref{fig:transAM1} and \ref{fig:transAM50} we set $\widehat{R}^{in}=20$, $\widehat{A}^{in}=50$ and $\widehat{\Xi}^{in}=5$. Notice the transient decay for the divergence, see Remark \ref{rem:trandec}, up to times close to the critical one, namely $t= 7$. Numerical simulation performed with MATLAB R2019.}
		\label{fig:plotinvcomp1}
	\end{figure}
	\begin{remark}[Wave equation for $\widehat{R}$]
		\label{rem:waveR}	Combining the equations \eqref{eq:hR0}-\eqref{eq:hA0} we have that 
		\begin{equation}
			\partial_{tt} \widehat{R}-\frac{\dt p}{p}\dt \widehat{R}+\bigg(\frac{\p}{\M^2}+\frac{2k^2}{\p}\bigg)\hR=\frac{2k^2}{\p}\hXi^{in}.
		\end{equation}
		In the physics literature the equation above is solved approximately for $M\ll 1$ \cite{bakas2009mechanisms,chagelishvili1994hydrodynamic,chagelishvili1997linear,goldreich1965ii,hau2015comparative},  for example in \cite{bakas2009mechanisms} is used a WKB approximation. Indeed, for $\Xi^{in}=0$ and assuming $M\ll 1$ one can say that the previous equation is approximated by 
		\begin{equation}
			M^2\partial_{tt} \hR_{app}=-p\hR_{app},
		\end{equation} 
		then, making a WKB ansatz, i.e. $\hR_{app}(t)=\exp(\delta^{-1}\sum_{n=0}^{+\infty}\delta^{n}S_n(t))$, a first order approximation satisfy 
		\begin{equation}
			\label{eq:wkbR}
			\hR_{app}(t)\approx S_1(t)=p^{\frac14}(t)\left(c_1e^{\frac{i}{M}\int_0^t \sqrt{p}(\tau)d\tau}+c_2e^{-\frac{i}{M}\int_0^t \sqrt{p}(\tau)d\tau} \right).
		\end{equation}
		In particular, recalling that $p=k^2+(\eta-kt)^2$, the previous formal analysis suggest that $|\hR|^2$ should grow linearly in time. In the following, we essentially prove the validity of this  asymptotic behaviour without the aid of any formal approximation.
	\end{remark}
	\subsubsection{Symmetrization}
	In order to study the system \eqref{eq:hR0}-\eqref{eq:hA0}, 	we want to look for a proper symmetrization of the system above. So  we define
	\begin{align}
		\label{def:hZ}
		Z(t)=(Z_1(t),Z_2(t))^T=\begin{pmatrix}
			&\displaystyle\frac{\hR}{\M\p^{\frac14}}(t),
			&\displaystyle\frac{\hA}{\p^{\frac34}}(t)
		\end{pmatrix}^T.
	\end{align}
	Observe that if we are able to get a uniform bound on $|Z|$, in view of the weight on $R$, we will match the asymptotic behaviour predicted by \eqref{eq:wkbR}. 
	By a direct computation we find that $Z(t)$ satisfy
	\begin{equation}
		\label{eq:dtZ}
		\begin{split}
			\Dt Z(t)=&L(t)Z(t)+F(t)\hXi^{in},\\
			Z(0)=&Z^{in}
		\end{split}
	\end{equation}
	where  
	\begin{equation}
		\label{def:LF}
		L(t)= \begin{bmatrix}
			\displaystyle -\frac{\dt p}{4\p} & \displaystyle -\frac{\sqrt{\p}}{\M} \\
			\displaystyle \frac{\sqrt{\p}}{\M} +\frac{2\M k^2}{\p^{3/2}}&\displaystyle \frac{\dt p}{4\p}
		\end{bmatrix},\ \ \ F(t)=\begin{pmatrix}
			0 \\\displaystyle -\frac{2k^2}{\p^{7/4}}
		\end{pmatrix}
	\end{equation}
	and 
	\begin{equation}
		\label{def:Zin}
		Z^{in}=\left(\frac{1}{M(k^2+\eta^2)^{\frac14}}\hR^{in},\frac{1}{(k^2+\eta^2)^{\frac34}}\hA^{in}\right)^T
	\end{equation}
	A key property coming from the choice of the weights on $R,A$ given in the definition \eqref{def:hZ} is that the matrix $L(t)$ is trace-free.
	
	We now have to deal with a non-autonomous 2D dynamical system. The solution of \eqref{eq:dtZ}, given by Duhamel's formula, is
	\begin{equation}
		\label{eq:solZ}
		Z(t)=\Phi_L(t,0)\bigg(Z^{in}+\int_{0}^{t}\Phi_L(0,s)F(s)\hXi^{in}ds\bigg),
	\end{equation}
	where $\Phi_L$ is the solution operator defined in \eqref{def:mhiGamma}. Notice that $\Phi_L\neq \exp(L)$ since $L(t)L(s)\neq L(s)L(t)$.    Therefore, everything is reduced in studying properties of the operator $\Phi_L$, which is equivalent to the study of the homogeneous problem associated to \eqref{eq:dtZ}. 
	\subsubsection*{Properties of $\Phi_L$}
	In order to investigate properties of \eqref{eq:dtZ} when $\Xi^{in}=0$, let us first consider the following toy example
	\begin{equation*}
		\begin{split}
			\Dt Z(t)=&\begin{bmatrix}
				-a & -b \\ d & a
			\end{bmatrix}Z(t),
			\\
			Z(0)=&Z^{in},
		\end{split}
	\end{equation*}
	where $a,b,d\in \R$, $b,d \neq0$, $bd>0$ and $Z^{in}\in \R^2$. Then, one can check that $Z(t)$ satisfy 
	\begin{equation*}
		E(t)=\sqrt{\frac{d}{b}}|Z_1|^2(t)+\sqrt{\frac{b}{d}}|Z_2|^2(t)+2\frac{a}{\sqrt{db}}Z_1(t)Z_2(t)=E(0).
	\end{equation*}
	In particular, if $\sqrt{bd}>a$ then a trajectory in the phase space is an ellipse determined by the equation above. In the non-autonomous case, we cannot expect to have immediately a conserved quantity. However, we have the following lemma which play a crucial role in our subsequent analysis.
	\begin{lemma}
		\label{keylemma}
		Let $Z(t)$ be a solution to \eqref{eq:dtZ} with $\hXi^{in}=0$. Define 
		\begin{equation}
			\label{def:abd}
			a(t)=\frac14 \frac{\dt p}{p}, \qquad b(t)=\frac{\sqrt{p}}{M}, \qquad d(t)=\frac{\sqrt{p}}{M}+\frac{2Mk^2}{p^{3/2}}.
		\end{equation}
		and 
		\begin{equation}
			\label{def:tildeE}
			E(t)=\left(\sqrt{\frac{d}{b}}|Z_1|^2\right)(t)+\left(\sqrt{\frac{b}{d}}|Z_2|^2\right)(t)+2\left(\frac{a}{\sqrt{db}}\Re(Z_1\bar{Z}_2)\right)(t).
		\end{equation}
		Then, there exists constants $c_1,C_1,c_2,C_2>0$ independent of $k,\eta$ such that 
		\begin{equation}
			\label{bd:enkeylemma}
			c_1E(0)\leq E(t)\leq C_1 E(0),
		\end{equation} 
		and
		\begin{equation}
			\label{bd:upplowkey}
			c_2|Z^{in}|\leq \big|Z(t)|\leq C_2|Z^{in}|.
		\end{equation}
		In addition, let $\Re(Z_1(t))=r(t)\cos(\theta(t))$ and $\Re(Z_2(t))=r(t)\sin(\theta(t))$ (or the imaginary part), then we have
		\begin{equation}
			\label{eq:dttheta}
			\Dt \theta(t)= \frac{\sqrt{p}}{M}+\frac{2Mk^2}{p^{3/2}}\cos(\theta(t))^2+\frac14 \frac{\dt p}{p}\sin(2\theta(t)).
		\end{equation}
	\end{lemma}
	This Lemma shows that the trajectories of the homogeneous problem associated to \eqref{eq:dtZ} are contained inside an annular region of the $Z$-plane and rotate with an angular velocity given by $\theta$. In particular, since $d/b\to 1$ and $a/\sqrt{db}\to 0$ as $t\to\infty$, the limit cycle is a circle. Approximating \eqref{eq:dttheta} and retaining the leading order terms one may infer a dispersion relation like $M^{-1}\sqrt{k^2+(\eta-kt)^2}$, which was also observed in \cite{bakas2009mechanisms,hau2015comparative} and is the one suggested by the WKB approximation, see \eqref{eq:wkbR}. However, dispersive properties require a more delicate analysis which we do not pursue here.
	
	We now present the proof of the Lemma \ref{keylemma}.
	\begin{proof} \label{sec:proofkeylemma}
		First of all, we define  
		\begin{equation}
			\label{def:zetabeta}
			\zeta= \sqrt{\frac{d}{b}}, \qquad \beta=\sqrt{bd}.
		\end{equation} 
		Then, with the notation introduced in \eqref{def:abd}, for $\hXi^{in}=0$ the system \eqref{eq:dtZ} become
		\begin{align}
			\label{xdot}\Dt Z_1&=-aZ_1-bZ_2,\\
			\label{ydot}\Dt Z_2&=dZ_1+aZ_2.
		\end{align}
		Notice that 
		\begin{equation}
			\label{bd:zeta}
			1\leq \zeta^2=1+\frac{2M^2k^2}{p^2}\leq 1+2M^2,
		\end{equation}
		hence, multiplying \eqref{xdot} by $\zeta$ and dividing \eqref{ydot} by $\zeta$ we obtain that 
		\begin{align}
			\label{gammadotx}
			\zeta \Dt Z_1&=-a\zeta Z_1-\beta Z_2,\\
			\label{oneovergammadoty}\frac{1}{\zeta}\Dt Z_2&=\beta Z_1+\frac{a}{\zeta}Z_2.
		\end{align}
		Now we multiply \eqref{gammadotx} by $\bar{Z}_1$, \eqref{oneovergammadoty} by $\bar{Z}_2$ and we add the two equations to have that 
		\begin{equation}
			\label{enestpart}
			\frac{\zeta}{2}\Dt |Z_1|^2+\frac{1}{2\zeta}\Dt |Z_2|^2=-a\left(\zeta |Z_1|^2-\frac{1}{\zeta}|Z_2|^2\right).
		\end{equation}
		Then, since the matrix $L$ is trace-free, just observe the following 
		\begin{equation}
			\label{dtxy}
			\Dt \operatorname{Re}(Z_1\bar{Z}_2)=\beta \big(\zeta |Z_1|^2-\frac{1}{\zeta}|Z_2|^2\big).
		\end{equation}
		Plugging \eqref{dtxy} into \eqref{enestpart} we get that 
		\begin{equation*}
			\frac{\zeta}{2}\Dt |Z_1|^2+\frac{1}{2\zeta}\Dt |Z_2|^2+\frac{a}{\beta}\Dt\operatorname{Re}(Z_1\bar{Z}_2)=0.
		\end{equation*}
		Hence, in view of \eqref{def:zetabeta}, we have that $E(t)$, see \eqref{def:tildeE}, satisfy 
		\begin{equation}
			\label{eq:dttildeE}
			\Dt E=\Dt\bigg(\log(\zeta)\bigg) \zeta |Z_1|^2+\Dt\bigg(\log(\frac{1}{\zeta})\bigg)\frac{1}{\zeta}|Z_2|^2+2\Dt\bigg(\frac{a}{\beta}\bigg) \operatorname{Re}(Z_1\bar{Z}_2).
		\end{equation}
		Then, since $|\dt p|\leq 2|k|\sqrt{p}$, observe that 
		\begin{equation}
			\label{bd:abeta}
			\frac{|a|}{\beta}\leq \frac12 \frac{|k|}{\sqrt{p}}\left(\frac{ p}{M^2}+\frac{2k^2}{p}\right)^{-\frac12}\leq \frac12 \frac{|k|}{\sqrt{p}} \frac{\sqrt{p}}{\sqrt{2}|k|}=\frac{1}{2\sqrt{2}}. 
		\end{equation}
		Therefore, by calling $\widetilde{E}(t)=\zeta |Z_1|^2+\zeta^{-1} |Z_2|^2$, we obtain 
		\begin{equation}
			\label{bd:abetaE}
			-\frac12 \widetilde{E}(t)\leq 2(\frac{|a|}{\beta}\operatorname{Re}(Z_1\bar{Z}_2))(t)\leq \frac{1}{2}\widetilde{E}(t),
		\end{equation}
		consequently it follows that
		\begin{equation}
			\label{bd:coercivity}
			\frac12 \widetilde{E}(t)\leq E(t)\leq \frac{3}{2}\widetilde{E}(t).
		\end{equation}
		By combining \eqref{eq:dttildeE} with \eqref{bd:coercivity}, we get 
		\begin{equation}
			\label{bd:upperEtilde}
			\begin{split}
				\Dt E \leq& \frac32 \left(\left|\Dt \log(\zeta)\right|+\left|\Dt \left(\frac{a}{\beta}\right)\right|\right) \widetilde{E},\\
				\leq&\frac94\left(\left|\Dt \log(\zeta)\right|+\left|\Dt \left(\frac{a}{\beta}\right)\right|\right) E.
			\end{split}
		\end{equation}
		Analogously, we have
		\begin{equation}
			\label{bd:lowerEtilde}
			\Dt {E} \geq -\frac14 \left(\left|\Dt \log(\zeta)\right|+\left|\Dt \left(\frac{a}{\beta}\right)\right|\right) {E}.
		\end{equation}
		In order to apply the Gr\"onwall's Lemma, it remains to provide a uniform bound for the integral in time of the terms in bracket of \eqref{bd:lowerEtilde}. For the first one, observe that since $\dt \zeta^2=-4M^2k^2(\dt p)p^{-3}$ changes sign only in $t=\eta/k$ one has 
		\begin{equation}
			\label{bd:gron1}
			\begin{split}
				\int_{0}^\infty \left|\Dt \log(\zeta)\right|d\tau=&\frac12\int_{0}^\infty \left|\Dt \log(\zeta^2)\right|d\tau\\
				=&\frac12\log\left(\frac{\zeta^2(\eta/k)}{\zeta^2(0)}\right)+\frac12\log\left(\frac{\zeta^2(\eta/k)}{\zeta^2(+\infty)}\right)\\ 
				\leq& \ \log(1+2M^2),
		\end{split}	\end{equation} 
		where we have also used \eqref{bd:zeta}. For the one involving $a/\beta$, since the bound \eqref{bd:abeta} is uniform with respect to $M$, we simply observe that being $a/\beta$ a bounded rational function may change sign only a finite number of times $n_0$, so that 
		\begin{equation}
			\label{bd:gron2}
			\int_0^{+\infty}\left|\Dt \left(\frac{a}{\beta}\right)\right|d\tau\leq \frac{n_0}{\sqrt{2}}.
		\end{equation}
		Therefore, by combining \eqref{bd:upperEtilde}, \eqref{bd:lowerEtilde} with \eqref{bd:gron1} and \eqref{bd:gron2}, applying Gr\"onwall's Lemma we infer  
		\begin{equation}
			\label{boundEtilde}
			c_1{E}(0)\leq {E}(t)\leq C_1{E}(0),
		\end{equation}
		whence proving \eqref{bd:enkeylemma}. In addition, in view of \eqref{bd:coercivity}, from \eqref{boundEtilde} we get 
		\begin{equation}
			\label{bd:tildeE0}
			\widetilde{c}_2\widetilde{E}(0)\leq \widetilde{E}(t)\leq \widetilde{C}_2\widetilde{E}(0).
		\end{equation}
		Then,  thanks to \eqref{bd:zeta} we know that 
		\begin{equation}
			\label{bd:hZM}
			(1+2M^2)^{-1}\widetilde{E}(t)\leq |\Phi_L(t,0)Z^{in}|^2=|Z(t)|^2\leq(1+2M^2)\widetilde{E}(t), 
		\end{equation}
		and combining the inequalities above with \eqref{bd:tildeE0} we prove \eqref{bd:upplowkey}.	
		
		To prove \eqref{eq:dttheta} observe that the coefficients of the system \eqref{eq:dtZ} are all real valued. Therefore, being the system linear, the real and imaginary part decouples. Then \eqref{eq:dttheta} directly follows by the fact that $r^2\dot{\theta}=x\dot{y}-\dot{x}y$.
	\end{proof}
	\begin{remark}\label{rem:Minv}
		From the proof of Lemma \ref{keylemma}, in view of the bounds \eqref{bd:gron1} and \eqref{bd:hZM}, we also observe that the constants appearing in \eqref{bd:upplowkey} satisfy $c_2,C_2\approx \langle M\rangle^\beta$ for some $\beta>1$ explicitly computable.
	\end{remark} 
	\subsubsection{Upper and lower bounds}
	
	We can now present a more precise statement of the Theorem \ref{th:NScouetteintro} in the inviscid case by considering separately the upper and lower bounds. Regarding the upper bounds we have the following.
	\begin{theorem}
		\label{maintheorem}
		Let $\rho^{in}, \ \omega^{in}\in H^1_xH^2_y$ and $\alpha^{in}\in H^{-\frac{1}{2}}_xL^2_y$ be the initial data of \eqref{eq:contcouetteint}-\eqref{eq:vorticitycouetteint} with $\rho^{in}_0=\omega^{in}_0=\alpha^{in}_0=0$. 
		Then the following inequality hold
		\begin{equation}
			\label{inq:energybound}
			\begin{split}
				\norma{Q[\vv](t)}{L^2}+\frac{1}{\M}\norma{\rho(t)}{L^2}\lesssim \jap{t}^\frac12\bigg(\norm{\frac{\rho^{in}}{M}}_{L^2}+\norm{\alpha^{in}}_{H^{-1}}+\norm{\rho^{in}+\omega^{in}}_{H^1}\bigg).
			\end{split}
		\end{equation}
		For the solenoidal component of the velocity we have 
		\begin{equation}
			\label{inq:P1}\begin{split}
				\norma{P[\vv]^x(t)}{L^2}\lesssim&\ \frac{\M}{\langle t\rangle^\frac12}\bigg(\norma{ \frac{\rho^{in}}{\M}}{H^{-\frac12}_xL^2_y}+\norma{\alpha^{in}}{H^{-\frac12}_xH^{-1}_y} +\norma{ \rho^{in}+\omega^{in}}{H^{-\frac12}_xH^{\frac12}_y}\bigg)\\
				&+\frac{1}{\langle t\rangle}\norma{\rho^{in}+\omega^{in}}{H^{-1}_xH^1_y},
			\end{split}
		\end{equation}
		\begin{equation}
			\begin{split}
				\label{inq:P2}\norma{P[\vv]^y(t)}{L^2}\lesssim&\ \frac{\M}{\langle t\rangle^\frac32}\bigg(\norma{ \frac{\rho^{in}}{\M}}{H^{-\frac12}_xH^{1}_y}+\norma{\alpha^{in}}{H^{-\frac12}_xL^2_y}+\norma{ \rho^{in}+\omega^{in}}{H^{-\frac12}_xH^{\frac{3}{2}}_y}\bigg)\\
				&+\frac{1}{\langle t\rangle^2}\norma{\rho^{in}+\omega^{in}}{H^{-1}_xH^2_y}.
			\end{split}
		\end{equation}
	\end{theorem}          	    
	In view of the analysis in the frequency space that can be given through the Lemma \ref{keylemma}, it is also possible to give an estimate on general Sobolev norms.
	Consequently, we could choose a suitable Sobolev space where also the acoustic part decays, which in particular implies a mixing phenomenon.
	\begin{corollary}
		\label{corollary}
		Let $s\geq 1/2$, $\rho^{in},\omega^{in} \in H^{s-\frac12}$ and $\alpha^{in}\in H^{s-\frac32}$ be the initial data of \eqref{eq:contcouetteint}-\eqref{eq:vorticitycouetteint} with $\rho^{in}_0=\omega^{in}_0=\alpha^{in}_0=0$.	     Then 
		\begin{equation}
			\label{decay}
			\norma{Q[\vv](t)}{H^{-s}}+\frac{1}{\M}\norma{\rho(t)}{H^{-s}}\leq \frac{1}{\langle t \rangle^{s-\frac12}}C\big(\rho^{in},\alpha^{in},\omega^{in}\big),
		\end{equation}
		where the constant depends upon Sobolev norms of the initial data.
	\end{corollary}
	We will not prove the corollary above since its proof can be directly deduced by the proof of Theorem \ref{maintheorem}, which we present in detail.
	
	We now turn our attention to the lower bound \eqref{bd:lwcomp}. To state the results it is convenient to introduce 
	\begin{equation}
		\label{def:GammaZXi}
		\Gamma(t,Z^{in},\Xi^{in})=\widehat{Z}^{in}+\int_0^t\Phi_L(0,s)F(s)\widehat{\Xi}^{in}ds.
	\end{equation}
	We then have the following.
	\begin{theorem}
		\label{maintheoremlwz} Let $\rho^{in}, \ \omega^{in}\in L^2_xH^{-\frac12}_y$ and $\alpha^{in}\in H^{-\frac{3}{2}}_xH^{-2}_y$ with $\rho^{in}_0=\omega^{in}_0=\alpha^{in}_0=0$. Then the solution of \eqref{eq:contcouetteint}-\eqref{eq:vorticitycouetteint} with initial data $\rho^{in},\alpha^{in}, \omega^{in}$ satisfy
		\begin{align}
			\notag
			\norma{Q[\vv](t)}{L^2}+\frac{1}{\M}\norma{\rho(t)}{L^2}\gtrsim \langle t \rangle^\frac12\norm{\Gamma(t,Z^{in},\Xi^{in})}_{L^2_xH^{-1/2}_y},
		\end{align}
		where $Z^{in}$ is defined as in  \eqref{def:hZ} and $\Xi^{in}=\rho^{in}+\omega^{in}$.
	\end{theorem}
	Clearly, looking at \eqref{def:GammaZXi}, if $\Xi^{in}=0$, namely $\rho^{in}=-\omega^{in}$, we immediately have a growth in time for non trivial initial conditions. When $\Xi^{in}\neq 0$, it may happen that the right-hand side of the inequality above become zero for some $t$. For this reason, in the following proposition we show that the set of initial data for which the right-hand side of the bound in Theorem \ref{maintheoremlwz} vanishes at some time has empty interior in any Sobolev space in which the initial data are taken.
	
	\begin{proposition}
		\label{prop:lwdensity}
		Let $s_1\in \R$ and $s_2\geq -1/2$. Given $\rho^{in},\omega^{in}\in H^{s_1}_xH^{s_2}_y$ and $\alpha^{in} \in H^{s_1-\frac32}_xH^{s_2-\frac32}_y$, let $
		\Gamma(t,Z^{in},\Xi^{in})$ be defined as in \eqref{def:GammaZXi}, where $Z^{in}$ is defined as in  \eqref{def:hZ} and $\Xi^{in}=\rho^{in}+\omega^{in}$. 
		
		Then, for any $\epsilon>0$ sufficiently small, there exists $\rho_\epsilon^{in},\alpha_\epsilon^{in},\omega_\epsilon^{in}$ such that
		\begin{align}
			\label{bd:closeHs}
			\norma{\rho^{in}-\rho_{\epsilon}^{in}}{ H^{s_1}_xH^{s_2}_y}+\norma{\omega^{in}-\omega_{\epsilon}^{in}}{ H^{s_1}_xH^{s_2}_y}+\norma{\alpha^{in}-\alpha_{\epsilon}^{in}}{H^{s_1-\frac32}_xH^{s_2-\frac32}_y}\leq 2 \epsilon,
		\end{align}
		and, by defining $Z^{in}_\epsilon, \Xi^{in}_\epsilon$ accordingly, the following inequality holds
		\begin{equation}
			\label{bd:densiti}
			\inf_{t\geq 0}\norma{\Gamma(t,Z^{in}_\epsilon,\Xi^{in}_\epsilon)}{L^2_xH^{-1/2}_y}\geq  \frac{\epsilon}{2}.
		\end{equation}
	\end{proposition}
	\begin{remark}
		In the proof of Proposition \ref{prop:lwdensity}, given at the end of this section, we construct the perturbation $(\rho_\epsilon^{in},\alpha_\epsilon^{in},\omega_\epsilon^{in})$  at any fixed frequency $k,\eta$, satisfying a non-degeneracy condition related to \eqref{bd:densiti}.
	\end{remark}
	In the following, we prove Theorem \ref{maintheorem} and Theorem \ref{maintheoremlwz}.
	
	\begin{proof}[Proof of Theorem \ref{maintheorem}]
		We first prove the bounds for the solenoidal component of the velocity field, namely \eqref{inq:P1} and \eqref{inq:P2}. By \eqref{eq:decOmega}, we have 
		\begin{equation*}
			|\hOmega|(t,k,\eta)\leq|\hR|(t,k,\eta)+|\hXi^{in}|(k,\eta).
		\end{equation*}
		Then, thanks to the Biot-Savart law, we prove \eqref{inq:P1} as follows,
		\begin{align*}
			\norma{P[\vv]^x(t)}{L^2}^2&=\norma{(\dy \Delta^{-1}\omega)(t)}{L^2}^2\\
			&=\sum_k\int \frac{(\eta-kt)^2}{\p^{2}}|\hOmega(t)|^2d\eta\\
			&\lesssim \sum_k\int \M^2\frac{(\eta-kt)^2}{\p^{3/2}}\left|\frac{\hR(t)}{\M\p^{1/4}}\right|^2+\frac{(\eta-kt)^2}{\p^2}|\hXi^{in}|^2d\eta\\
			&\lesssim \sum_k\int \frac{\M^2}{\sqrt{\p}}(|Z^{in}|^2+|\hXi^{in}|^2)+\frac{1}{\p}|\hXi^{in}|^2d\eta.
		\end{align*}
		Now since $\jap{\eta/k-t}\langle \eta/k \rangle \gtrsim \langle t \rangle$ observe that 
		\begin{equation*}
			\frac{1}{\sqrt{p}}=\frac{1}{|k|\jap{\eta/k-t}}\frac{\jap{\eta/k}}{\jap{\eta/k}}\lesssim \frac{1}{\jap{t}}\frac{\jap{\eta}}{\jap{k}}.
		\end{equation*}
		Hence, recalling the definition of $Z^{in}$, see \eqref{def:Zin}, we infer  
		\begin{align*}
			\norma{P[\vv]^x(t)}{L^2}^2\lesssim&\frac{\M^2}{\langle t\rangle}\bigg(\norma{ \frac{\rho^{in}}{\M}}{H^{-\frac12}_xL^2_y}^2+\norma{\alpha^{in}}{H^{-\frac12}_xH^{-1}_y}^2 +\norma{ \rho^{in}+\omega^{in}}{H^{-\frac12}_xH^{\frac12}_y}^2\bigg)\\
			&+\frac{1}{\langle t\rangle^2}\norma{\rho^{in}+\omega^{in}}{H^{-1}_xH^1_y}^2,
		\end{align*} 
		Similarly for $P[\vv]^y$, we prove \eqref{inq:P2} as follows 
		\begin{align*}
			\norma{P[\vv]^y(t)}{L^2}^2=&\norma{\dx \Delta^{-1}\omega}{L^2}^2 \\
			\lesssim& \sum_k\int \M^2\frac{k^2}{\p^{3/2}}\left|\frac{\hR(t)}{\M \p^{1/4}}\right|^2+\frac{k^2}{\p^2}|\hXi^{in}|^2d\eta\\
			\lesssim&\frac{\M^2}{\langle t\rangle^3}\bigg(\norma{ \frac{\rho^{in}}{\M}}{H^{-\frac12}_xH^{1}_y}^2+\norma{\alpha^{in}}{H^{-\frac12}_xL^2_y}^2+\norma{ \rho^{in}+\omega^{in}}{H^{-\frac12}_xH^{\frac{3}{2}}_y}^2\bigg)\\
			&+\frac{1}{\langle t\rangle^4}\norma{\rho^{in}+\omega^{in}}{H^{-1}_xH^2_y}^2.
		\end{align*}
		To prove  \eqref{inq:energybound} first of all observe that, thanks to Lemma \ref{keylemma}, we have 
		\begin{equation}
			\label{inq:PhiL1}
			\begin{split}
				\int_0^\infty |\Phi_L(t,s)F(s)|ds&\lesssim \int_0^\infty|F(s)|ds\\
				&\lesssim \frac{1}{|k|^{\frac32}}\int_{0}^{\infty}\frac{ds }{(1+(\eta/k-s)^2)^{\frac74}}\lesssim \frac{1}{|k|^{\frac32}}.
			\end{split}
		\end{equation}
		Hence, by recalling the definition of $Z$, see \eqref{eq:solZ1}, combining Lemma \ref{keylemma} with \eqref{inq:PhiL1} we infer  
		\begin{equation}
			\label{bd:Zhat}
			|\widehat{Z}(t,k,\eta)|\lesssim |Z^{in}(k,\eta)|+|\hXi^{in}(k,\eta)| \quad \text{for any $t\geq0$}.
		\end{equation}
		Then, by the Helmholtz decomposition we have  
		\begin{align*}
			\norma{Q[\vv](t)}{L^2}^2+\frac{1}{\M^2}\norma{\rho(t)}{L^2}^2=&\norma{(\dx \Delta^{-1}\alpha)(t)}{L^2}^2+\norma{(\dy \Delta^{-1}\alpha)(t)}{L^2}^2+\frac{1}{\M^2}\norma{\rho(t)}{L^2}^2\\
			=&\sum_k\int \frac{|\halpha(t)|^2(t,k,\eta)}{k^2+\eta^2}+\frac{1}{\M^2}|\hrho(t)|^2(t,k,\eta)d\eta\\
			=&\sum_k\int \frac{|\hA(t)|^2}{\p}(t,k,\eta)+\frac{1}{\M^2}|\hR(t)|^2(t,k,\eta)d\eta.
		\end{align*}
		Therefore, by \eqref{bd:Zhat} and the fact that $p\leq\jap{t}^2\jap{k,\eta}^2$, we get
		\begin{align}
			\label{ideproofQ}\norma{Q[\vv](t)}{L^2}^2+\frac{1}{M^2}\norma{\rho(t)}{L^2}^2= \ & \sum_k\int \sqrt{\p}\bigg(\bigg|\frac{\hA(t)}{\p^{3/4}}\bigg|^2+\bigg|\frac{\hR(t)}{\M \p^{1/4}}\bigg|^2\bigg)d\eta\\
			\notag=& \sum_k \int \sqrt{p}|Z(t)|^2d\eta\\
			\lesssim& \jap{t}(\norm{Z^{in}}_{H^1}^2+\norm{\rho^{in}+\omega^{in}}_{H^1}^2)\\
			\notag=&	\jap{t}\left(\norm{\frac{\rho^{in}}{M}}_{L^2}^2+\norm{\alpha^{in}}^2_{H^{-1}}+\norm{\rho^{in}+\omega^{in}}_{H^1}^2\right),
		\end{align}
		hence concluding the proof of Theorem \ref{maintheorem}.
	\end{proof}
	We now prove Theorem \ref{maintheoremlwz}.
	\begin{proof}[Proof of Theorem \ref{maintheoremlwz}]
		Recall that the solution of \eqref{eq:dtZ} is given by the Duhamel's formula as 
		\begin{equation}
			\label{eq:solZ1}
			Z(t)=\Phi_L(t,0)\bigg(Z^{in}+\int_{0}^{t}\Phi_L(0,s)F(s)\hXi^{in}ds\bigg)=\Phi_L(t,0)\Gamma(t,Z^{in},\Xi^{in}),
		\end{equation}	     	
		where we have also used the definition of $\Gamma$ given in \eqref{def:GammaZXi}.
		By Lemma \ref{keylemma} we have 
		\begin{equation}
			\label{inq:boundZ}
			|Z(t)|\geq c|\Gamma(t,Z^{in},\Xi^{in})|,
		\end{equation}
		hence, combining the bound above with the identity in \eqref{ideproofQ} we get  
		\begin{align*}
			\norma{Q[\vv](t)}{L^2}^2+\frac{1}{\M^2}\norma{\rho(t)}{L^2}^2\gtrsim& \sum_k\int \sqrt{\p}|\Gamma(t,Z^{in},\Xi^{in})|^2 d\eta\\
			\gtrsim& \sum_k \int \langle \eta-kt \rangle |\Gamma(t,Z^{in},\Xi^{in})|^2d\eta \\
			\gtrsim&  \langle t\rangle\sum_k\int \frac{1}{\langle \eta \rangle}|\Gamma(t,Z^{in},\Xi^{in})|^2d\eta,
		\end{align*}
		where in the last two lines we have used that $\sqrt{p}\geq \jap{\eta-kt}$ and $\jap{\eta-kt}\jap{\eta}\gtrsim \jap{kt}$. Therefore we have  
		\begin{align*}
			\norma{Q[\vv](t)}{L^2}^2+\frac{1}{\M^2}\norma{\rho(t)}{L^2}^2   \gtrsim&\langle t \rangle \norma{\Gamma(t,Z^{in},\Xi^{in})}{L^2_xH^{-1/2}_y}^2
		\end{align*}
		and the proof is over.
	\end{proof}	
	
	Finally, we present the proof of the Proposition \ref{prop:lwdensity}.

	\begin{proof}[Proof of Proposition \ref{prop:lwdensity}]
		\label{sec:propZin}
		With a slight abuse of notations, from the definition \eqref{def:GammaZXi} we have  
		\begin{equation}
			\label{eq:Gamm}
			\widehat{\Gamma}(t,k,\eta)=\widehat{Z}^{in}(k,\eta)+\int_0^t\Phi_L(0,s)F(s,k,\eta)\widehat{\Xi}^{in}(k,\eta)ds.
		\end{equation}
		Now, let us fix the frequencies $k,\eta$. In this way, $t\mapsto \widehat{\Gamma}(t)$ is a regular curve in $\mathbb{C}^2$. We now want to construct a perturbation of the initial data. If $\widehat{\Xi}^{in}(k,\eta)=0$ there is nothing to prove. So we assume that $\widehat{\Xi}^{in}(k,\eta)\neq0$.
		
		First of all, by a computation similar to \eqref{inq:PhiL1}, we know that $\lim_{t\to \infty} \Gamma(t,k,\eta)=\Gamma^{\infty}(k,\eta)$. Let us first consider the case $\Gamma^\infty\neq 0$.
		
		We claim that in this case $\Gamma(t,k,\eta)$ vanishes at most in a finite number of times $t_i$ for $i=0,\dots, n$. 
		
		Indeed, since $|\Gamma^{\infty}|>0$ and the integral in \eqref{eq:Gamm} is converging, see \eqref{inq:PhiL1}, there is a $T(\Gamma^\infty,k,\eta)>0$ such that
		\begin{equation}
			\label{bd:Gammainfty}
			|\Gamma(t,k,\eta)|>\frac12|\Gamma^{\infty}(k,\eta)| \quad \text{for $t\geq T(\Gamma^\infty,k,\eta)$}.
		\end{equation}
		Hence, we know that $\Gamma$ may vanish only for $t\in[0,T(\Gamma^\infty,k,\eta)]$. Then, by \eqref{eq:Gamm} and the bound \eqref{bd:upplowkey} in Lemma \ref{keylemma}, we have 
		\begin{align}
			|\dt \Gamma(t,k,\eta)|=&|\Phi_L(0,t)F(t,k,\eta)\widehat{\Xi}^{in}(k,\eta)|\geq C|F(t,k,\eta)\widehat{\Xi}^{in}(k,\eta)|\\
			\geq &C(T,k,\eta)|\widehat{\Xi}^{in}(k,\eta)|,
		\end{align}
		where, from the definition of $F$ given in \eqref{def:LF}, we define 
		\begin{equation*}
			C(T,k,\eta)=C\min_{t\in[0,T]}|F(t,k,\eta)|=2C\left(|k|^{\frac32}\max_{t\in[0,T]}\jap{\eta/k-t}^{\frac72}\right)^{-1}>0,
		\end{equation*}
		and the last inequality follows since $(k,\eta)$ are fixed, $|k|\geq 1$ and $T<+\infty$. 
		Consequently, since in the compact set $[0,T(\Gamma^{\infty})]$ there are no points $t^*$ such that $\Gamma(t^*)=\dt \Gamma(t^*)=0$, exploiting also the continuity of $\dt \Gamma$, we have that $\Gamma$ vanishes at most in a finite number of times in the interval $[0,T(\Gamma^\infty)]$.\\
		Now we can construct the perturbation of the initial data. Consider the set $\{\dt \Gamma(t_i,k,\eta)\}$ for $i=0,\dots,n$. In view of the continuity of $\dt \Gamma$, there is an $\epsilon<\min\{|\Gamma^\infty|/2,1\}$ and at least one unit vector $\nu_\epsilon(k,\eta)=(\nu_{\epsilon}^1(k,\eta),\nu_{\epsilon}^2(k,\eta))$ which is not parallel to any $\dt \Gamma(t_i,k,\eta)$ such that 
		\begin{equation}
			\label{bd:Gamma}
			|\Gamma(t,k,\eta)+\epsilon e^{-(k^2+\eta^2)}\nu_\epsilon(k,\eta)|>\frac12\epsilon e^{-(k^2+\eta^2)}.
		\end{equation}
		By choosing 
		\begin{align}
			\widehat{\alpha}_\epsilon^{in}(k,\eta)&=\widehat\alpha^{in}(k,\eta)+\epsilon(k^2+\eta^2)^{\frac34} e^{-(k^2+\eta^2)}\nu_{\epsilon}^2(k,\eta),\\
			\widehat\rho_\epsilon^{in}(k,\eta)&=\widehat\rho^{in}(k,\eta )+\epsilon M(k^2+\eta^2)^{\frac14}e^{-(k^2+\eta^2)}\nu_{\epsilon}^1(k,\eta),\\
			\widehat\omega_\epsilon^{in}(k,\eta)&=\widehat\omega^{in}(k,\eta )-\epsilon M(k^2+\eta^2)^{\frac14}e^{-(k^2+\eta^2)}\nu_{\epsilon}^1(k,\eta),
		\end{align}
		we clearly satisfy \eqref{bd:closeHs}. In addition, we have 
		\begin{align}
			\widehat Z_\epsilon^{in}(k,\eta)=\widehat Z^{in}(k,\eta)+\epsilon e^{-(k^2+\eta^2)}\nu_\epsilon(k,\eta), \qquad \widehat \Xi^{in}_\epsilon(k,\eta)=\widehat \Xi^{in}(k,\eta).
		\end{align}
		Consequently
		\begin{align*}
			\widehat{\Gamma}^\epsilon(t,k,\eta)&=\widehat{Z}_\epsilon^{in}(k,\eta)+\int_0^t\Phi_L(0,s)F(s,k,\eta)\widehat{\Xi}_\epsilon^{in}(k,\eta)ds\\
			&=\Gamma(t,k,\eta)+\epsilon e^{-(k^2+\eta^2)}\nu_\epsilon(k,\eta).
		\end{align*}
		By combining \eqref{bd:Gammainfty} with \eqref{bd:Gamma} we get that 
		\begin{equation}
			\label{bd:Gammaeps}
			|\Gamma^\epsilon(t,k,\eta)|\geq \frac12\min \left(|\Gamma^\infty(k,\eta)|,\epsilon e^{-(k^2+\eta^2)}\right) \qquad \text{for any $t>0$.}
		\end{equation}
		Let us now turn to the case $\Gamma^\infty(k,\eta)=0$. First we choose
		\begin{equation}
			\alpha^{in}_1=\alpha^{in}+\epsilon(k^2+\eta^2)^{\frac34}e^{-(k^2+\eta^2)},
		\end{equation} 
		so that for the corresponding $\Gamma^1$ we get $|\Gamma^{1,\infty}(k,\eta)|=\epsilon e^{-(k^2+\eta^2)}$. At this point, we can repeat the previous argument.
		
		Resuming, by using Plancherel's Theorem, from the bound \eqref{bd:Gammaeps} we obtain \eqref{bd:densiti}, hence the proof is over.  
	\end{proof}
	\section{The viscous case}\label{sec:visccompcouette}
	
	In this section, we study the system \eqref{eq:contcouetteint}-\eqref{eq:vorticitycouetteint} in presence of viscosity, namely we assume that $\nu>0$ and $\lambda\geq 0$. As done in the previous section, we will prove the results only in the case $\rho_0^{in}=\alpha_0^{in}=\omega_0^{in}=0$, since the dynamics of the zero mode decouples with respect to fluctuations around it, see Section \ref{sec:zeromode}. 
	
	To remove the transport terms, we again consider the change of coordinates \eqref{def:movframe} and we use the notation introduced in \eqref{def:DeltaL}-\eqref{def:p'}. Defining $\mu=\nu+\lambda$, by taking the Fourier transform of the system \eqref{eq:contcouetteint}-\eqref{eq:vorticitycouetteint} in the moving frame, we have
	\begin{align}
		\label{eq:NShR}
		&\dt \hR=-\hA,\\
		\label{eq:NShA}&\dt \hA=\frac{\dt p}{p}\hA-\mu p \hA+\frac{1}{M^2}p\hR-\frac{2k^2}{p}\hOmega,\\
		\label{eq:NShO}&\dt \hOmega=\hA-\nu p\hOmega.
	\end{align}
	In the inviscid case, the conservation of $\Xi=R+\Omega$ was crucial in order to have a closed system in terms of $R,A$, which allow us to deal with a $2\times 2$ non-autonomous system of ODE's in the Fourier space. Also in the viscous case it turns out that it is convenient to replace $\Omega$ with another auxiliary variable, however, the conservation of $\Xi$ no longer hold since 
	\begin{equation}
		\label{eq:dtXiphys}
		\dt \hXi=-\nu p\hXi+\nu p \hR.
	\end{equation}
	In addition, the last term in the right-hand side of \eqref{eq:dtXiphys} may be a problem to perform energy estimates. Indeed, it is not possible to directly control $\nu p \hR$ in a straightforward energy estimate since we do not have a similar dissipative term for $R$. 
	\subsubsection*{The good unknown} To overcome this difficulty, we observe that 
	\begin{equation}
		\label{eq:dtXinuAphys1}
		\begin{split}
			\dt (\hXi-\nu M^2\hA)=&\ -\nu p(\hXi-\nu M^2\hA)+\nu(\mu-\nu)M^2p \hA\\
			&-\nu M^2\frac{\dt p}{p}\hA+2\nu M^2\frac{k^2}{p}\left(\hXi-\nu M^2\hA\right)\\
			&-2\nu M^2\frac{k^2}{p}\hR+2\nu^2 M^4\frac{k^2}{p}\hA,
		\end{split}
	\end{equation}
	where we have also used $\Omega=\Xi-R$. Although the equation \eqref{eq:dtXinuAphys1} looks more complicated with respect to \eqref{eq:dtXiphys}, it has a better structure to make use of $\Xi-\nu M^2A$ as an auxiliary variable in an energy estimate. This because the first term in the right-hand side of \eqref{eq:dtXinuAphys1} give us dissipation, the second term scales as the available dissipation for $A$ (notice that if $\lambda=0$ this term does not appear) and the remaining ones are lower order.  
	\begin{remark}
		One of the main difficulties to obtain an enhanced dissipation estimate is the absence of a diffusive term in the continuity equation, because otherwise it would have been sufficient to combine the equation \eqref{eq:dtXiphys} with an adaptation of the energy estimates given in the inviscid case. Instead, we need to take advantage of the underlying wave structure in the system. More precisely, we will be able to exploit the coupling between the density and the divergence to gain a dissipative term for the density. A similar strategy, inspired by the classical paper of Matsumura and Nishida \cite{matsumura1983initial}, has been already exploited by Guo and Wang \cite{guo2012decay} to prove decay time rates for the compressible Navier-Stokes equations with smooth and small initial data.
	\end{remark}
\begin{remark}[The $\nu=0$ case]
	\label{rem:nu0}
	When $\nu=0$ we have $\dt\Xi=0$ as in the inviscid case. This implies that we can replace $\Omega$ with $\Xi^{in}-R$ in \eqref{eq:NShA} so that the system is closed in terms of $R$ and $A$. Then the system to study is much simpler but the exponential decay is not valid in general. Indeed, if $\Xi^{in}\neq0 $ the forcing term appearing in the equation for the divergence give us the convergence towards an asymptotic state which is different, in general, from zero. If $\Xi^{in}=0$ we again have the exponential decay as in Theorem \ref{th:NScouetteintro} with $\lambda$ replacing $\nu$ in the bounds. We do not detail this case since it can be deduced with the same method that we are going to present in the general scenario.
\end{remark}
	This section is organized as follows. In Subsection \ref{subsec:thNS1} we prove Theorem \ref{th:NScouetteintro} whereas in Subsection \ref{subsec:thNS2} we prove Theorem \ref{th:NSnoloss}.
	
	Throughout all this section we make use of the following notation 
	\begin{equation}
		\label{def:Cin}
		C_{in,s}=\frac{1}{M}\norm{\rho^{in}}_{H^{s+1}}+\norm{\alpha^{in}}_{H^{s}}+\norm{\Xi^{in}-\nu M^2\alpha^{in}}_{H^{s}}.
	\end{equation}
	\subsection{Combining the dissipation enhancement with the inviscid mechanism}
	\label{subsec:thNS1}
	In this subsection, we prove Theorem \ref{th:NScouetteintro}, which combines the dissipation enhancement generated by the presence of the background shear flow with the inviscid dynamics. More precisely, our goal is to obtain estimates such that in the limit $\nu \to0$ we  recover the bounds of the inviscid case. However, as previously discussed, when viscosity is present there is the loss of a conservation law, meaning that we will need to consider a $3\times 3$ system in the Fourier space.  
	
	The key estimate which allow us to prove Theorem \ref{th:NScouetteintro} is given in the following proposition.
	\begin{proposition}
		\label{th:enhancedcomp}
		Let $s\geq 0$, $0<\mu \leq 1/2, \ M>0$ be such that $M\leq \min \{\mu^{-\frac12},\nu^{-\frac13}\}$. If $\rho^{in}\in H^{s+1}(\mathbb{T}\times \mathbb{R})$ and $\alpha^{in},\omega^{in}\in H^{s}(\mathbb{T}\times \mathbb{R})$ then
		\begin{equation}
			\label{bd:RA}
			\begin{split}
				\frac{1}{M}\norm{(p^{-\frac14}{\hR})(t)}_{H^s}&+\norm{(p^{-\frac34}{\hA})(t)}_{H^s}\\
				&+\norm{(p^{-\frac34}(\hXi-\nu M^2\hA))(t)}_{H^s}\lesssim e^{-\frac{1}{32}\nu^\frac13 t}C_{in,s},
			\end{split}
		\end{equation}
		where $C_{in,s}$ is defined in \eqref{def:Cin}.
	\end{proposition}
	In accordance with the inviscid case, one should not expect any weight on the auxiliary variable $\hXi-\nu M^2\hA$. However, the weight $p^{-\frac34}$ is introduced for technical reasons, since it helps to control the second term in the right-hand side of \eqref{eq:dtXinuAphys1}, which is not present if  $\lambda \neq 0$. We discuss more about this point in Remark \ref{rem:lambda0}.
	
	To recover a bound on the vorticity one could exploit the fact that $\Omega=(\Xi-\nu M^2A)+\nu M^2A-R$ and use the previous proposition to infer estimates on $P[\vv]$. This procedure, since $A$ and $\Xi-\nu M^2A$ have slower decay rates with respect to $R$ (formally, think of $p^{-1}$ as $t^{-2}$), would lead to worst decay rates with respect to the one given in Theorem \ref{th:NScouetteintro} for the  solenoidal component of the velocity. In particular, one cannot recover the estimates in the nonviscous and incompressible case by performing the formal limits $\nu \to 0$ and $M\to 0$ respectively.
	
	Instead, by solving the equation for $\Xi$, see \eqref{eq:dtXiphys}, via Duhamel's formula, from Proposition \ref{th:enhancedcomp} we infer the following.
	\begin{corollary}
		\label{cor:vorticity}
		Let $s\geq 0$, $\mu \leq 1/2, \ M>0$ be such that $M\leq \min \{\mu^{-\frac12},\nu^{-\frac13}\}$. If $\rho^{in}\in H^{s+\frac72}(\mathbb{T}\times \mathbb{R})$ and $\alpha^{in},\omega^{in}\in H^{s+\frac52}(\mathbb{T}\times \mathbb{R})$ then
		\begin{equation}
			\label{bd:Omega}
			\begin{split}
				\norm{\Omega(t)}_{H^s}\lesssim& \  M\langle t \rangle^{\frac12}e^{-\frac{1}{32}\nu^{\frac13}t}C_{in,s+\frac12}+M\jap{t}^\frac12e^{-\frac{1}{64}\nu^{\frac13}t}C_{in,s+\frac52}\\
				&+e^{-\frac{1}{12}\nu^{\frac13}t}\norm{\omega^{in}+\rho^{in}}_{H^s}
			\end{split}
		\end{equation}
	\end{corollary}
	\begin{remark}
		\label{rem:loss}
		Observe that in Proposition \ref{th:enhancedcomp} and in Corollary \ref{cor:vorticity} we are losing derivatives. The loss in \eqref{bd:RA} comes from the technical obstruction that forces us to introduce the weight $p^{-\frac34}$ for the variable $\Xi-\nu M^2A$. For $\lambda=0$ one does not have this loss of derivatives, see Remark \ref{rem:lambda0}. Instead, the loss of derivatives in \eqref{bd:Omega} seems to be necessary in view of the last term in the right-hand side of \eqref{eq:dtXiphys}, where we can control time-growth by paying regularity.
	\end{remark}
	In the following,  appealing to Proposition \ref{th:enhancedcomp} and Corollary \ref{cor:vorticity} we first prove Theorem \ref{th:NScouetteintro}, while the proofs of the proposition and the corollary are postponed to the end of this subsection.
	\begin{proof}[Proof of Theorem \ref{th:NScouetteintro}]
		We start with the proof of \eqref{bd:compenh}.  From the Helmholtz decomposition \eqref{Helmholtzint} we have 
		\begin{align}
			\norm{Q[\vv](t)}_{L^2}^2+\frac{1}{M^2}\norm{\rho(t)}_{L^2}^2=&\norm{(-\Delta)^{-\frac12}\alpha(t)}_{L^2}^2+\frac{1}{M^2}\norm{\rho(t)}_{L^2}^2\\
			=&\norm{(-\Delta_L)^{-\frac12}A(t)}_{L^2}^2+\frac{1}{M^2}\norm{R(t)}_{L^2}^2,
		\end{align}
		where in the last line we have done the change of variables $X=x-yt, \ Y=y$. By the Plancherel's Theorem and the fact that $p\leq\jap{t}^2\jap{k,\eta}^2$, we get 
		\begin{align}
			\norm{Q[\vv](t)}_{L^2}^2+\frac{1}{M^2}\norm{\rho(t)}_{L^2}^2=&\norm{p^{\frac12}(p^{-\frac34}\widehat{A})(t)}_{L^2}^2+\frac{1}{M^2}\norm{p^{\frac12}(p^{-\frac12}\widehat{R})(t)}_{L^2}^2\\
			\lesssim &\langle t \rangle \left(\norm{(p^{-\frac34}\widehat{A})(t)}_{H^1}^2+\frac{1}{M^2}\norm{(p^{-\frac12}\widehat{R})(t)}_{H^1}^2\right)\\
			\lesssim & \langle t \rangle e^{-\frac{1}{16}\nu^\frac13 t}(C_{in,1})^2,
		\end{align}
		where in the last line we have used \eqref{bd:RA}, hence proving \eqref{bd:compenh}.
		
		We now turn our attention to the solenoidal component of the velocity in order to prove \eqref{bd:P1enh} and \eqref{b:P2enh}. By using again the Helmholtz decomposition, we have 
		\begin{align}
			\norm{P[\vv]^x(t)}_{L^2}=&\norm{\dy\Delta^{-1}\omega(t)}_{L^2}=\norm{(\dY-t\dX)(\Delta_L^{-1}\Omega)(t)}_{L^2}\\
			\leq&\norm{((-\Delta_L)^{-\frac12}\Omega)(t)}_{L^2}. 
		\end{align}
		Therefore, since $p^{\frac12}\jap{kt}\geq \jap{\eta-kt}\jap{kt}\gtrsim \jap{\eta}$, we get 
		\begin{equation}
			\norm{P[\vv]^x(t)}_{L^2}\lesssim \frac{1}{\langle t\rangle}\norm{{\Omega}(t)}_{H^1}
		\end{equation}
		and combining the previous bound with \eqref{bd:Omega} we prove \eqref{bd:P1enh}. The bound \eqref{b:P2enh} follows analogously.
	\end{proof}

	In order to prove Proposition \ref{th:enhancedcomp}, we have to control a weighted energy functional. From the bounds on this energy functional the proof of Proposition \ref{th:enhancedcomp} readily follows.
	\subsubsection{\textbf{The weighted energy functional}}
	We need to introduce the following Fourier multiplier, already used in \cite{BGM15III,bedrossian2019stability,bedrossian2016enhanced,zillinger2020enhanced},
	\begin{align}
		\dt m(t,k,\eta)=&\frac{2\nu^{\frac13}}{\nu^\frac23\left(\frac{\eta}{k}-t\right)^2+1}m(t,k,\eta),\\
		m(0,k,\eta)=&\exp(2\arctan(\nu^{\frac13}\frac{\eta}{k}))
	\end{align}
	which is explicitly given by
	\begin{equation}
		\label{def:m}
		m(t,k,\eta)=\exp(2\arctan(\nu^{\frac13}(t-\frac{\eta}{k})).
	\end{equation}
	Clearly $m$ and $m^{-1}$ are bounded Fourier multipliers, therefore they generates an equivalent norm to the standard $L^2$.
	
	The multiplier $m$ is introduced since it enjoys the following crucial property
	\begin{equation}
		\label{bd:propm}
		\nu p(t,k,\eta)+\frac{\dt m}{m}(t,k,\eta)\geq \nu^{\frac13} \qquad \text{ for any } t\geq 0,\  k\in \mathbb{Z}\setminus\{0\}, \ \eta\in \R,
	\end{equation}
	which compensates the slow down of the enhanced dissipation mechanism close to the critical times $t=\eta/k$.
	We then consider the system given by $(R,A,\Xi-\nu M^2A)$, namely the equations \eqref{eq:NShR}, \eqref{eq:NShA} when replacing $\hOmega$ with $(\hXi-\nu M^2\hA)+\nu M^2\hA-\hR$, and  \eqref{eq:dtXinuAphys1}. Clearly, for the system under consideration the dynamics decouples in $k,\eta$, therefore we can perform estimates at each fixed frequency. Let $s\geq 0$, we define the following weighted variables  
	\begin{equation}
		\label{def:ZiNScomp1}
		\begin{split}
			&Z_1(t)=\frac{1}{M}\jap{k,\eta}^s(m^{-1} p^{-\frac14}\widehat{R})(t), \quad Z_2(t)=\jap{k,\eta}^s(m^{-1} p^{-\frac34}\widehat{A})(t), \\
			& Z_3(t)=\jap{k,\eta}^s(m^{-1} p^{-\frac34}(\widehat{\Xi}-\nu M^2\widehat{A}))(t). 
		\end{split}
	\end{equation}
	Besides the multiplier $\jap{k,\eta}^sm^{-1}$, we remark that $Z_1, Z_2$ are the quantities also used in the non viscous case in order to symmetrize the system, see Section \ref{sec:invcompcouette}. Instead, $Z_3$, as explained, is introduced as an auxiliary variable to close the energy estimate.
	
	Then, let $ 0< \gamma=\gamma(M,\nu)\leq 1/4$ be a parameter to be chosen later and consider the following energy functional
	\begin{align}
		\label{def:Et1}	E(t)=\frac12 \bigg(&\left(1+M^2\frac{(\dt p)^2}{p^3}\right)|Z_1|^2(t)+|Z_2|^2(t)+ |Z_3|^2(t)\\
		\label{def:Et2}&+(\frac{M}{2} \frac{\dt p}{p^{\frac32}}\Re(\bar{Z}_1Z_2))(t)-(2\gamma p^{-\frac12}\Re(\bar{Z}_1Z_2))(t)\bigg).
	\end{align}
	Since $|\dt p|<p$, it is immediate to check that the previous energy functional is coercive, namely 
	\begin{align}
		\label{bd:coerc1}
		E(t)\geq& \ \frac{1}{4}\left((1+M^2\frac{(\dt p)^2}{p^3})|Z_1|^2+|Z_2|^2+ 2|Z_3|^2\right)(t)\\ 
		\label{bd:coerc2}E(t)\leq&\  \left((1+M^2\frac{(\dt p)^2}{p^3})|Z_1|^2+|Z_2|^2+ |Z_3|^2\right)(t).
	\end{align}
	Since $m$ is a bounded Fourier multiplier, we also have 
	\begin{align}
		\label{equiv}
		\sum_{k\neq 0}\int E(t)d\eta\approx\frac{1}{M^2}\norm{p^{-\frac14}\widehat{R}(t)}_{H^s}^2+\norm{p^{-\frac34}\widehat{A}(t)}_{H^s}^2+\norm{p^{-\frac34}(\widehat{\Xi}-\nu M^2 \widehat{A})(t)}^2_{H^s}.
	\end{align}
	The latter equivalence tells us that a suitable estimate on $E(t)$ will imply the bound  \eqref{bd:RA} in Proposition \ref{th:enhancedcomp}. In particular, we aim at proving the following 
	\begin{lemma}
		\label{lemma:dtEcouette}
		Under the assumptions of Proposition \ref{th:enhancedcomp}, let $E(t)$ be defined as in \eqref{def:Et1} and $\gamma=\displaystyle \frac{\nu^\frac13M}{4}$, then 
		\begin{equation}
			\label{bd:Lemmacouette}
			\sum_k \int E(t) d\eta\lesssim e^{-\frac{\nu^\frac13}{16}t}(C_{in,s})^2,
		\end{equation}
		where $C_{in,s}$ is defined in \eqref{def:Cin}.
	\end{lemma}
	Thanks to the previous Lemma, we conclude the proof of Proposition \ref{th:enhancedcomp} as follows 
	\begin{align}
		\frac{1}{M^2}\norm{p^{-\frac14}\widehat{R}(t)}_{H^s}^2+&\norm{p^{-\frac34}\widehat{A}(t)}_{H^s}^2+\norm{p^{-\frac34}(\widehat{\Xi}-\nu M^2 \widehat{A})(t)}^2_{H^s}\\ \lesssim& \sum_k \int E(t) d\eta \lesssim e^{-\frac{\nu^\frac13}{16}t}(C_{in,s})^2. 
	\end{align}
	We now have to prove Lemma \ref{lemma:dtEcouette}.	
	\begin{proof}[Proof of Lemma \ref{lemma:dtEcouette}] We are going to prove the bound \eqref{bd:Lemmacouette} via a Gr\"onwall's inequality. Therefore, we have to first compute the time derivative of $E(t)$.
		
		First of all, observe that 
		\begin{align}
			\label{eq:Z1}\dt Z_1=&-\frac{\dt m}{m}Z_1-\frac14\frac{\dt p}{p}Z_1-\frac{1}{M}p^{\frac12}Z_2,\\
			\label{eq:Z2}\dt Z_2=&-\left(\frac{\dt m}{m}Z_2+\mu p\right)Z_2+\frac14\frac{\dt p}{p}Z_2+\left(\frac{1}{M}p^{\frac12}+2M\frac{k^2}{p^{\frac32}}\right)Z_1\\
			&-2\frac{k^2}{p}Z_3-2\nu M^2\frac{k^2}{p}Z_2,\\
			\label{eq:Z3}\dt Z_3=&-\left(\frac{\dt m}{m}+\nu p\right)Z_3 -\frac34 \frac{\dt p}{p}Z_3+\nu (\mu-\nu)M^2  pZ_2-\nu M^2 \frac{\dt p}{p}Z_2\\
			&-2\nu M^3 \frac{k^2}{p^{\frac32}}Z_1+2\nu M^2 \frac{k^2}{p}Z_3+2\nu^2 M^4\frac{k^2}{p}Z_2.
		\end{align}
		
		From  \eqref{eq:Z1} we directly compute that 
		\begin{equation}
			\label{dt:Z1}
			\frac12\Dt |Z_1|^2=-\frac{\dt m}{m}|Z_1|^2-\frac14\frac{\dt p}{p}|Z_1|^2-\frac1M p^{\frac12}\Re(\bar{Z}_1Z_2),
		\end{equation}
		and 
		\begin{equation}
			\label{dt:Z1M2}
			\begin{split}
				\frac{M^2}{2}\Dt \left|\frac{\dt p}{p^{\frac32}}Z_1\right|^2=&M^2\left(\frac{2k^2\dt p}{p^3}-\frac32\frac{(\dt p)^3}{p^4}\right)|Z_1|^2-M^2\frac{\dt m}{m}\frac{(\dt p)^2}{p^{3}}|Z_1|^2\\
				&-\frac{M^2}{4}\frac{(\dt p)^3}{p^4}|Z_1|^2-M \frac{(\dt p)^2}{p^{\frac52}}\Re(\bar{Z}_1Z_2).
			\end{split}
		\end{equation}
		By \eqref{eq:Z2} we get
		\begin{align}
			\label{dt:Z2}
			\frac12 \Dt |Z_2|^2=&-\left(\frac{\dt m}{m}+\mu p\right)|Z_2|^2+\frac14 \frac{\dt p}{p}|Z_2|^2+\frac1M p^{\frac12}\Re(Z_1\bar{Z}_2)\\
			&+2M\frac{k^2}{p^{\frac32}}\Re(Z_1\bar{Z}_2)-2\frac{k^2}{p}\Re(Z_3\bar{Z}_2)-2\nu M^2\frac{k^2}{p}|Z_2|^2.
		\end{align}
		From \eqref{eq:Z3} we have 
		\begin{align}
			\label{dt:Z3}
			\frac12 \Dt |Z_3|^2=&-\left(\frac{\dt m}{m}+\nu  p\right)|Z_3|^2-\frac34\frac{\dt p}{p}|Z_3|^2\\
			&+\nu(\mu-\nu)M^2p \Re(Z_2\bar{Z}_3)-\nu M^2 \frac{\dt p}{p}\Re(Z_2\bar{Z}_3)\\
			&-2\nu M^3\frac{k^2}{p^\frac32}\Re(Z_1\bar{Z}_3)+2\nu M^2\frac{k^2}{p}|Z_3|^2+2\nu^2 M^4\frac{k^2}{p}\Re(Z_2\bar{Z}_3).
		\end{align}
		Now we compute the time derivative of the mixed terms appearing in \eqref{def:Et2}.	
		The first term in \eqref{def:Et2} is introduced to cancel  the terms with $\frac14 \dt p/p$ coming from the summation of \eqref{dt:Z1} with \eqref{dt:Z2}, indeed observe that
		\begin{align}
			\label{dt:mix1}
			\frac{M}{4}\Dt (\frac{\dt p}{p^{\frac32}}\Re(\bar{Z}_1Z_2))=&\frac{M}{4}\left(\frac{2k^2}{p^\frac32}-\frac32 \frac{(\dt p)^2}{p^\frac52}\right)\Re(\bar{Z}_1Z_2)-\frac{M}{2}\frac{\dt m}{m}\frac{\dt p}{p^{3/2}}\Re(\bar{Z}_1Z_2)\\
			\label{dt:mix11}&-\frac14 \frac{\dt p}{p}(|Z_2|^2-|Z_1|^2)-\mu\frac{M}{4}\frac{\dt p}{p^{\frac12}}\Re(\bar{Z}_1Z_2)\\
			&+M^2\frac{k^2\dt p}{2p^3}|Z_1|^2-M\frac{k^2\dt p}{2p^{\frac52}}\Re(\bar{Z}_1Z_3)\\
			&-\nu M^3\frac{k^2\dt p}{2p^{\frac52}}\Re(\bar{Z}_1Z_2)
		\end{align}
		The second term in \eqref{def:Et2} give us a dissipative term for $Z_1$ as follows
		\begin{align}
			\label{eq:dtZ1Z2}
			-\gamma \Dt \left(p^{-\frac12}\Re(\bar{Z}_1Z_2)\right)=&\frac{\gamma}{2} \frac{\dt p}{p^{\frac32}}\Re(\bar{Z}_1Z_2)+2\gamma\frac{\dt m}{m}p^{-\frac12}\Re(\bar{Z}_1Z_2)\\
			&+\frac{\gamma}{M}|Z_2|^2-\frac{\gamma}{M}\left(1+2M^2\frac{k^2}{p^2}\right)|Z_1|^2\\
			&+\gamma \mu p^\frac12\Re(\bar{Z}_1Z_2)\\
			&+2\gamma \frac{k^2}{p^{\frac32}}\Re(\bar{Z}_1Z_3)+2\gamma \nu M^2\frac{k^2}{p^{\frac32}}\Re(\bar{Z}_1Z_2).
		\end{align}
		Hence, by rearranging the terms appearing in \eqref{dt:Z1}, \eqref{dt:Z1M2}, \eqref{dt:Z2}, \eqref{dt:Z3}, \eqref{dt:mix1} and  \eqref{eq:dtZ1Z2}  we have the following identity 
		\begin{align}
			\label{eq:dtE}
			\Dt E(t)=&-\left(\frac{\dt m}{m}+\mu p\right)|Z_2|^2-\left(\frac{\dt m}{m}+\frac{\gamma}{M}\big(1+2M^2\frac{k^2}{p^2}\big)\right)|Z_1|^2\\
			&-\left(\frac{\dt m}{m}+\nu p\right)|Z_3|^2+\sum_{i=1}^5\mathcal{D}_i+\sum_{i=1}^6\mathcal{I}_i,
		\end{align}
		where we define the \textit{dissipative error terms} as
		\begin{align}
			\label{def:D1}
			\mathcal{D}_1=&\frac{\gamma}{M}|Z_2|^2, \quad \mathcal{D}_2=\gamma \mu p^{\frac12}\Re(\bar{Z}_1Z_2), \quad \mathcal{D}_3=\nu(\mu-\nu)M^2 p\Re(\bar{Z}_2Z_3) \\
			\label{def:D2} \mathcal{D}_4=&-\nu M^2\frac{\dt p}{p}\Re(\bar{Z}_2Z_3), \quad \mathcal{D}_5=-\frac{\mu M}{4}\frac{\dt p}{p^{\frac12}}\Re(\bar{Z}_1Z_2),
		\end{align}
		which we need to control with the negative terms appearing in \eqref{eq:dtE}. Instead, the \textit{integrable error terms} are given by
		\begin{align}
			\label{def:I11}\mathcal{I}_1=&M^2\left(\frac{5k^2\dt p}{2p^3}-\frac{7}{4}\frac{(\dt p)^3}{p^4}\right)|Z_1|^2,\\
			\label{def:I21}\mathcal{I}_2=&\ M\left(\frac12(\frac{\gamma}{M}-\frac{\dt m}{m})\frac{\dt p}{p^\frac32}+(\frac52+2\gamma \nu M)\frac{k^2}{p^{\frac32}}-\frac{11}{8}\frac{(\dt p)^2}{p^{\frac52}}\right)\Re(\bar{Z}_1Z_2)\\
			\label{def:I22}&-\nu \frac{k^2M^3 \dt p}{2p^\frac52}\Re(\bar{Z}_1Z_2)+2\gamma \frac{\dt m}{m}p^{-\frac12}\Re(\bar{Z}_1Z_2),\\
			\label{def:I3}\mathcal{I}_3=&\ -\frac34 \frac{\dt p}{p}|Z_3|^2+2\nu M^2\frac{k^2 }{p}|Z_3|^2,\\
			\label{def:I4} \mathcal{I}_4=&\ M\left(2(\frac{\gamma}{M}-\nu M^2)\frac{k^2}{p^{\frac32}} -\frac{k^2\dt p}{2p^\frac52}\right)\Re(\bar{Z}_1Z_3),\\
			\label{def:I5} \mathcal{I}_5=&\ \left(-2\frac{k^2}{p}+2\nu^2M^4 \frac{k^2}{p}\right)\Re(\bar{Z}_2Z_3),\\
			\label{def:I6} \mathcal{I}_6=&\ -M^2\frac{\dt m}{m}\frac{(\dt p)^2}{p^3}|Z_1|^2-2\nu \frac{k^2M^2}{p}|Z_2|^2,
		\end{align} 
		which all involve Fourier multipliers integrable in time, as we explain below.
		
		Now we proceed by providing suitable bounds on the terms $\mathcal{D}_i$. To control $\mathcal{D}_1$, in view of the property \eqref{bd:propm}, we need to choose $\gamma$ such that $\gamma/M< \nu^{\frac13}$.	Therefore, we define 
		\begin{equation}
			\label{choicegamma}
			\gamma=\frac{M\nu^{\frac13}}{4}
		\end{equation} 
		and notice that $\gamma\leq 1/4$ by assumptions on $M$. To control the remaining terms, we also need to exploit the hypothesis 
		\begin{equation}
			\label{hyp:Mmugamma}
			\mu M^2 \leq 1.
		\end{equation}
		Since $\gamma \leq1/4$ and $M\mu\leq 1$, we bound $\mathcal{D}_2$ as 
		\begin{align}
			\label{bd:D2} |\mathcal{D}_2|\leq  \frac{\gamma}{2}\mu p|Z_2|^2+\frac{\gamma}{2M}(M\mu)|Z_1|^2\leq \frac{\mu}{8} p|Z_2|^2+\frac{\gamma}{2M}|Z_1|^2. 
		\end{align}
		To control $\mathcal{D}_3$, in view of the restriction $\mu M^2\leq 1$, we  have
		\begin{align}
			\label{bd:D3}
			|\mathcal{D}_3|&\leq \frac{\nu}{2}p|Z_2|^2+\frac{\nu}{2}p|Z_3|^2. 
		\end{align}
		Since $|\dt p|\leq 2|k|p^\frac12$, the bounds on $\mathcal{D}_4, \mathcal{D}_5$ are given by 
		\begin{align}
			\label{bd:D4}
			|\mathcal{D}_4|&\leq \frac{\nu}{4}|Z_3|^2+\nu M^4 \frac{(\dt p)^2}{p^2}|Z_2|^2\leq\frac{\nu}{4}|Z_3|^2+4M^2 \frac{k^2}{p}|Z_2|^2,\\
			\label{bd:D5}|\mathcal{D}_5|&\leq \frac{\mu}{16}p|Z_2|^2+\frac{\mu M^2k^2}{p}|Z_1|^2 
		\end{align}
		Notice that the last terms in the right-hand side of the last two inequalities need not to be absorbed with the negative terms in \eqref{eq:dtE}, being $k^2p^{-1}$ integrable in time.
		
		We now turn our attention to provide bounds for the terms $\mathcal{I}_i$, where we can exploit integrability in time of the Fourier multipliers. More precisely, we have two main contributions, one given by $\dt m/m$, which is clearly integrable in time and can also be absorbed with the negative terms appearing in \eqref{eq:dtE}. The second one is the multiplier $k^2p^{-1}$, appearing for example in $\mathcal{I}_3,\mathcal{I}_5$, whose integral in time is uniformly bounded with respect to $k,\eta$, namely
		\begin{equation}
			\int_0^t \frac{k^2}{p(\tau)} d\tau=\int_0^t\frac{d\tau}{(\frac{\eta}{k}-\tau)^2+1}=\ \left(\arctan(\frac{\eta}{k}-t)-\arctan(\frac{\eta}{k})\right).
		\end{equation}  
		In addition, for the term $\mathcal{I}_3$, since $\dt p>0$ for $t>\eta/k$, we have
		\begin{equation}
			\label{bd:I3}
			|\mathcal{I}_3|\leq -\frac34 \frac{\dt p}{p}\chi_{t<\frac{\eta}{k}}|Z_3|^2+\frac{2k^2M^2}{p}|Z_3|^2, 
		\end{equation}
		therefore we can also integrate in time the first term in the right-hand side of the last inequality. However, this will be the source of a loss of regularity as it will be clear later on.
		
		Then, since $|\dt p|\leq 2|k|p^\frac12$ and recalling \eqref{hyp:Mmugamma}, we roughly estimate the remaining terms as follows 
		\begin{align}
			\label{bd:I1} |\mathcal{I}_1|&\leq C_1M^2\frac{k^2}{p} |Z_1|^2,\\
			\label{bd:I2} |\mathcal{I}_2|&\leq  \left(C_2M\frac{k^2}{p}+\frac{1}{2}\frac{\dt m}{m}\right)(|Z_1|^2+|Z_2|^2),\\
			\label{bd:I4}|\mathcal{I}_4|&\leq C_4M\frac{k^2}{p}(|Z_1|^2+|Z_3|^2),\\
			\label{bd:I5} |\mathcal{I}_5|&\leq 2\frac{k^2}{p}(|Z_2|^2+|Z_3|^2),\\
			\label{bd:I6} \mathcal{I}_6&\leq 0,
		\end{align}
		where we perform the last trivial bound since we cannot gain much from $\mathcal{I}_6$.
		
		Therefore, thanks to the choice of $\gamma$ in \eqref{choicegamma}, the properties \eqref{bd:propm} and \eqref{bd:coerc1}, by combining \eqref{bd:D2}, \eqref{bd:D3}, \eqref{bd:D4}, \eqref{bd:I1}-\eqref{bd:I6} with \eqref{eq:dtE} we infer 
		\begin{align}
			\label{bd:dtE1}
			\Dt E(t)\leq& -\frac{\nu^\frac13}{16}\left((1+4M^2\frac{k^2}{p^2})|Z_1|^2+|Z_2|^2+|Z_3|^2\right)\\
			&+4\left(C_M\frac{k^2}{p}+\frac{1}{2}\frac{\dt m}{m}\right)E(t)-\frac34 \frac{\dt p}{p}\chi_{t<\frac{\eta}{k}}|Z_3|^2,
		\end{align}
		where $C_M$ is explicitly computable from the previous bounds. Then, since 
		\begin{equation}
			\label{bd:dtpp3}
			M^2\frac{(\dt p)^2}{p^3}\leq 4M^2 \frac{k^2}{p^2},
		\end{equation}
		and from \eqref{bd:coerc2} we know  that $|Z_3|^2\leq 2E(t)$, by \eqref{bd:dtE1} we get 
		\begin{equation}
			\Dt E(t)\leq -\frac{\nu^{\frac13}}{16}E(t)+\left(-\frac32 \frac{\dt p}{p}\chi_{t<\frac{\eta}{k}}+4C_M\frac{k^2}{p}+2\frac{\dt m}{m}\right)E(t),
		\end{equation}
		hence, applying  Gr\"onwall's Lemma we have 
		\begin{equation}
			\label{bd:E1}
			E(t)\leq \widetilde{C}_Me^{-\frac{ \nu^\frac13 }{16}t} \frac{p(0)^\frac32}{p(t)^{\frac32}\chi_{t<\frac{\eta}{k}}+p(\eta/k)^{\frac32}\chi_{t>\frac{\eta}{k}}}E(0)\leq \widetilde{C}_Me^{-\frac{ \nu^\frac13 }{16}t}\jap{k,\eta}^{3}E(0),
		\end{equation}
		where $\widetilde{C}_M=\exp(5\pi C_M)$. Clearly the term $\jap{k,\eta}^{3}$ is the one which cause the loss of regularity, coming from the bound \eqref{bd:I3} as we have stressed previously. 
		
		To conclude the proof of Lemma \ref{lemma:dtEcouette}, by summing in $k$ and integrating in $\eta$ in \eqref{bd:E1}, thanks to \eqref{bd:coerc2} and \eqref{def:ZiNScomp1} we have 
		\begin{equation}
			\label{bd:energy}
			\sum_{k\neq 0}\int E(t)d\eta \lesssim e^{-\frac{\nu^\frac13}{16}t}\left(\frac{1}{M^2}\norm{{R}^{in}}_{H^{s+1}}^2+\norm{{A}^{in}}_{H^{s}}^2+\norm{\Xi^{in}-\nu M^2A^{in}}_{H^{s}}^2\right),
		\end{equation}
		therefore, in view of \eqref{def:Cin}, the proof of Lemma \ref{lemma:dtEcouette} is over.
	\end{proof}
	\begin{remark}\label{rem:condM}
		Combining the choice of $\gamma$ with the restrictions \eqref{hyp:Mmugamma} we immediately recover the hypothesis on the Mach number made in Theorem \ref{th:NScouetteintro}, namely $	M\leq \min\{ \mu^{-\frac12}, \nu^{-\frac13}\}$. 	However, by choosing $\gamma= \delta M\nu^{\frac13}/4$ for $0<\delta \leq 1$, it would be sufficient that $	M\leq \min\{ \mu^{-\frac12}, \delta^{-1}\nu^{-\frac13}\}$,
		while in the exponential bound \eqref{bd:Lemmacouette} a factor $\delta$ will appear, namely we have $e^{-\delta\frac{\nu^\frac13}{16}t}$.	Therefore, we could slightly improve the range of available Mach numbers by deteriorating the decay rates.  
		
		In addition, from \eqref{bd:E1} we also see that the constants hidden when using the symbol $\lesssim$ grows exponentially fast with respect to $M$, which is clearly irrelevant for $M\approx1$ but deteriorates extremely the bounds for larger values of $M$. It should be possible to improve this dependency up to constants $O(\jap{M}^\beta)$ for some $\beta>1$ by considering exactly the energy functional used in the inviscid case, see Lemma \ref{keylemma} and Remark \ref{rem:Minv}, plus the terms due to the viscosity.
	\end{remark}
	\begin{remark}[Regularity in absence of bulk viscosity]
		\label{rem:lambda0}
		When $\lambda=0$, namely $\mu=\nu$, it is sufficient to consider the auxiliary variable 
		\begin{equation}
			\widetilde{Z}_3=\jap{k,\eta}^sm^{-1}(\hXi-\nu M^2\hA),
		\end{equation}
		which satisfy the following equation 
		\begin{equation}
			\begin{split}	
				\dt \widetilde{Z}_3=&-\left(\frac{\dt m}{m}+\nu p\right)\widetilde{Z}_3 -\nu M^2 \frac{\dt p}{p^\frac14}Z_2\\
				&-2\nu M^3 \frac{k^2}{p^{\frac34}}Z_1+2\nu M^2 \frac{k^2}{p}\widetilde{Z}_3+2\nu^2 M^4\frac{k^2}{p^{\frac14}}Z_2,
			\end{split}
		\end{equation}
		where $Z_1$ and $Z_2$ are defined in \eqref{def:ZiNScomp1}. We can then proceed as in the proof of Lemma \ref{lemma:dtEcouette}, clearly by defining the new error terms accordingly. For example, the most dangerous one can be controlled as follows 
		\begin{align*}
			\nu M^2 \frac{|\dt p|}{p^\frac14}\Re(\bar{Z}_2\widetilde{Z}_3)\leq& 4\nu M^4\frac{|\dt p|}{p^{\frac54}}|Z_2|^2+\frac{\nu}{16}p|\widetilde{Z}_3|^2\\
			\leq& 4\nu M^4\frac{|k|}{p^{\frac34}}|Z_2|^2+\frac{\nu}{16}p|\widetilde{Z}_3|^2, 
		\end{align*} 
		since $|k|p^{-\frac34}$ is integrable in time. However, by using $\widetilde{Z}_3$ we will not have the error term containing the multiplier $-\frac34\dt p/p$, see \eqref{def:I3}, meaning that in the bound analogous to \eqref{bd:E1} there is not $\jap{k,\eta}$, hence in Lemma \ref{lemma:dtEcouette} we do not lose derivatives (consequently also in Proposition \ref{th:enhancedcomp})
		
	\end{remark}
	We now turn our attention to the proof of Corollary \ref{cor:vorticity}.
	\begin{proof}[Proof of Corollary \ref{cor:vorticity}]
		In order to prove \eqref{bd:Omega}, let 
		\begin{equation*}
			\mathcal{L}_\nu(t,k,\eta)=\nu \int_0^tp(\tau,k,\eta)d\tau.
		\end{equation*}
	Since 
	\begin{equation}
		\frac13k^2t^2+k^2+\eta^2-\eta kt=\frac13k^2t^2+k^2 +\left(\frac12 kt-\eta\right)^2-\frac14 k^2t^2\geq \frac{1}{12}k^2t^2,
	\end{equation} 
 we get 
		\begin{equation}
			\label{bd:propenheat}
			\norm{e^{\mathcal{L}_\nu(t)}f}_{H^s}\leq e^{-\frac{1}{12}\nu t^3}\norm{f}_{H^s}\leq e^{-\frac{1}{12}\nu^{\frac13}t}\norm{f}_{H^s},
		\end{equation}  
		Therefore, solving \eqref{eq:dtXiphys} via Duhamel's formula we have 
		\begin{equation}
			\widehat{\Xi}(t)=e^{\mathcal{L}_\nu(t)}\Xi^{in}+\nu\int_0^t e^{\mathcal{L}_\nu(t-\tau)}p(\tau)\widehat{R}(\tau)d\tau.
		\end{equation}
		Appealing to \eqref{bd:propenheat}, we get 
		\begin{equation}
			\label{bd:Xi1}
			\norm{\widehat{\Xi}(t)}_{H^s}\leq e^{-\frac{1}{12}\nu^{\frac13}t}\norm{\widehat{\Xi}^{in}}_{H^s}+\nu \int_0^te^{-\frac{1}{12}\nu^{\frac13}(t-\tau)}\norm{p(\tau)\widehat{R}(\tau)}_{H^s}d\tau.
		\end{equation}
		To bound the integrand of the last equation, we exploit the bound obtained on $M^{-1}p^{-1/4}\widehat{R}$, see \eqref{bd:RA}. In particular, by using that $p\leq \jap{t}^2\jap{k,\eta}^2$, we have 
		\begin{align}
			\norm{p(\tau)\widehat{R}(\tau)}_{H^s}=&M\norm{p^{\frac54}(\tau)(M^{-1}p^{-\frac14}\widehat{R})(\tau)}_{H^s}\\
			\lesssim& M\langle \tau \rangle^{\frac52}\norm{(M^{-1}p^{-\frac14}\widehat{R})(\tau)}_{H^{s+\frac52}}\\
			\lesssim & M\langle \tau \rangle^{\frac52}e^{-\frac{1}{32}\nu^{\frac13}\tau}C_{in,s+\frac52}
		\end{align}
		where we recall the definition of $C_{in,s}$ given in \eqref{def:Cin}.     Consequently we get
		\begin{align}
			\nu \int_0^t &e^{-\frac{1}{12}\nu^{\frac13}(t-\tau)}\norm{p(\tau)\widehat{R}(\tau)}_{H^s}d\tau\lesssim\ M C_{in,s+\frac52} \nu \int_0^t e^{-\frac{1}{12}\nu^{\frac13}(t-\tau)}\langle \tau \rangle^{\frac52}e^{-\frac{1}{32}\nu^{\frac13}\tau}d\tau\\
			\lesssim&\ MC_{in,s+\frac52}\nu \jap{t}^\frac12 \int_0^te^{-\frac{1}{12}\nu^{\frac13}(t-\tau)}\nu^{-\frac23}(\nu^\frac13\jap{\tau})^2e^{-\frac{1}{32}\nu^{\frac13}\tau}d\tau\\
			\lesssim&\ MC_{in,s+\frac52}\nu^{\frac13} \jap{t}^\frac12 \int_0^te^{-\frac{1}{12}\nu^{\frac13}(t-\tau)}e^{-\frac{1}{64}\nu^{\frac13}\tau}d\tau\\
			\lesssim &\ MC_{in,s+\frac52}\nu^{\frac13} \jap{t}^\frac12e^{-\frac{1}{64}\nu^{\frac13}t}\int_0^te^{-(\frac{1}{12}-\frac{1}{64})\nu^{\frac13}(t-\tau)}d\tau\\
			\lesssim &\ MC_{in,s+\frac52}\jap{t}^\frac12e^{-\frac{1}{64}\nu^\frac13 t}.
		\end{align}
		Combining the previous estimate with \eqref{bd:Xi1} we obtain 
		\begin{equation}
			\label{bd:Xipf}
			\norm{{\Xi}(t)}_{H^s}\lesssim e^{-\frac{1}{12}\nu^{\frac13}t}\norm{{\Xi}^{in}}_{H^s}+M\jap{t}^\frac12 e^{-\frac{1}{64}\nu^\frac13 t}C_{in,s+\frac52}
		\end{equation}
		Finally, we directly recover the bound on $\Omega$ as follows
		\begin{align*}
			\norm{\Omega(t)}_{H^s}\leq &\norm{\Xi(t)}_{H^s}+\norm{R(t)}_{H^s}\\
			=&\norm{\Xi(t)}_{H^s}+M\norm{p^{\frac14}(M^{-1}p^{-\frac14}R)(t)}_{H^s}\\
			\lesssim &\norm{\Xi(t)}_{H^s}+M\langle t \rangle^{\frac12}\norm{M^{-1}p^{-\frac14}R(t)}_{H^{s+\frac12}}.
		\end{align*} 
		The proof of the Corollary \ref{cor:vorticity} then follows by combining the last bound with \eqref{bd:RA} and \eqref{bd:Xipf}.
	\end{proof}
	\subsection{Dissipation enhancement without loss of derivates}
	\label{subsec:thNS2}
	The purpose of this subsection is to prove Theorem \ref{th:NSnoloss}, where again we assume that $\rho_{in,0}=\alpha_{in,0}=\omega_{in,0}=0$. In the following, we first present a toy model to introduce the key Fourier multiplier which is crucial to avoid the loss of derivative encountered in Proposition \ref{th:enhancedcomp}, see also Remark \ref{rem:loss}. Then, in analogy with the previous subsection, we present a weighted estimate that follows by the control of a suitable energy functional.  Theorem \ref{th:NSnoloss} is a consequence of the weighted estimate. 
	
	\subsubsection*{The key Fourier multiplier}
	In the inviscid case and in the previous subsection it was crucial to properly symmetrize the system by weighting the density and the divergence with some negative powers of the Laplacian, see \eqref{def:hZ} and \eqref{def:ZiNScomp1} and recall that $p$ is the symbol associated to $-\Delta_L$. This was essential to balance the growth given by the term $\dt p/p$ present in the equation for the divergence, see also Remark \ref{rem:trandec}. However, in the viscous case it is possible to balance the growth given by $\dt p/p$ by using the dissipation.  To explain how, we consider a toy model introduced by Bedrossian, Germain and Masmoudi in \cite{BGM15III}. In particular, consider the following
	\begin{equation}
		\label{eq:ftoy}
		\dt f=\frac{\dt p}{p}f-\nu pf,
	\end{equation}
	which is clearly a relevant toy model also in our case, since the first two terms in the right-hand side of \eqref{eq:NShA} have exactly this structure. First of all, for $t\leq\eta/k$ we know that $\dt p=-2k(\eta-kt)\leq 0$, hence we may ignore this term in an energy estimate. Instead, for $t\geq \eta/k$ we have $\dt p\geq0$ leading to a growth on $f$ that it is not balanced by the dissipative term, indeed near the critical times $t=\eta/k$ one has $\nu p\approx \nu k^2$. More precisely, we cannot hope to have a uniform estimate like $\dt p/p\lesssim \nu p$ for $t\in[\eta/k,\eta/k+C_{\nu}]$ for some $C_{\nu}$. If we are sufficiently far away from the critical times dissipation overcomes the growth, namely for any $\beta>0$ one has
	\begin{align}
		\label{bd:trivnup}\nu p(t,k,\eta)\geq \ \beta^2 \nu^{\frac13}, \qquad &\text{if }  |t-\frac{\eta}{k}|\geq \beta \nu^{-\frac13},\\ 
		\label{bd:p'pnu}\frac{\dt p}{p}(t,k,\eta)\leq \frac{2}{\sqrt{1+(\frac{\eta}{k}-t)^2}}\leq \ 2\beta^{-1}\nu^{\frac13} \qquad& \text{if }  |t-\frac{\eta}{k}|\geq \beta \nu^{-\frac13},
	\end{align}
	so that for $\beta >2$ we see that $\dt p/p\leq \nu p/4$ if $|t-\eta/k|\geq \beta \nu^{-\frac13}$.
	
	In order to control the growth near the critical times, for a fixed $\beta>2$ to be specified later we introduce the following Fourier multiplier 
	\begin{equation}
		\label{def:w}
		\begin{split}
			(\dt w)(t,k,\eta)=&\begin{cases}
				0 \qquad &\text{if } t\notin[\frac{\eta}{k},\frac{\eta}{k}+\beta \nu^{-\frac13}] \\
				\displaystyle \left(\frac{\dt p}{p}w\right)(t,k,\eta) \qquad &\text{if }  t\in[\frac{\eta}{k},\frac{\eta}{k}+\beta \nu^{-\frac13}]
			\end{cases}\\
			w(0,k,\eta)=&1,
		\end{split}
	\end{equation} 
	which is explicitly given by
	\begin{equation}
		\label{def:explw}
		w(t,k,\eta)= \begin{cases}
			1 &\qquad \text{if } \eta k\geq 0 \text{ and }0\leq t\leq \frac{\eta}{k},\\
			1 &\qquad \text{if } \eta k< 0, \ |\frac{\eta}{k}|\geq \beta\nu^{-\frac13} \text{ and } t\geq 0,\\
			\displaystyle\frac{p(t,k,\eta)}{k^2} &\qquad \text{if } \frac{\eta}{k}\leq t\leq \frac{\eta}{k}+\beta \nu^{-\frac13},\\
			1+\beta^2\nu^{-\frac23} &\qquad \text{in all the other cases}.
		\end{cases}
	\end{equation} 
	This multiplier has been used also in \cite{BGM15III,liss2020sobolev}. Let us state some properties of this multiplier. 
	\begin{lemma}
		\label{lemma:propw}
		Let $w$ and $m$ be the Fourier multipliers defined in \eqref{def:explw} and \eqref{def:m} respectively. Then, for any $t\geq 0$, $\eta \in \R$ and $k\in \mathbb{Z}\setminus\{0\}$ the following inequalities holds:
		\begin{align}
			\label{bd:w}
			1\leq w(t,k,\eta)\leq& \beta^2\nu^{-\frac23},\\
			\label{bd:wp-1}(wp^{-1})(t,k,\eta)\leq &\frac{1}{k^2}.
		\end{align}
		In addition, for any $ \max\{2(\beta(\beta^2-1))^{-1},4\beta^{-1}\}<\delta_\beta\leq 1$ one has
		\begin{align}
			\label{bd:dtwdeltanup}
			\left(\delta_\beta(\frac{\dt m}{m}+\nu p)+\frac{\dt w}{w}-\frac{\dt p}{p}\right)(t,k,\eta)\geq& \ \delta_\beta \nu^{\frac13},\\
			\label{bd:dtwdeltanu13}\left(\delta_\beta(\frac{\dt m}{m}+\nu^\frac13)+\frac{\dt w}{w}-\frac{\dt p}{p}\right)(t,k,\eta)\geq&\ \frac{\delta_\beta}{2} \nu^{\frac13},
		\end{align} 
	\end{lemma} 
	
	Observe that the bound \eqref{bd:w} is exactly the maximal growth expected by solving explicitly \eqref{eq:ftoy}. Indeed, solving \eqref{eq:ftoy} one has 
	\begin{equation}
		(p^{-1}f)(t,k,\eta)=\frac{1}{\jap{k,\eta}^2}e^{-\nu\int_0^tp(\tau,k,\eta)d\tau}f^{in}(k,\eta),
	\end{equation}
	so that, since $p\leq \jap{t}^2\jap{k,\eta}^2$, by using \eqref{bd:propenheat} we get
	\begin{equation}
		\norm{f}_{H^s}=\norm{p(p^{-1}f)}_{H^s}\leq \jap{t}^2e^{-\frac{1}{12}\nu^\frac13 t}\norm{f^{in}}_{H^s}\lesssim \nu^{-\frac23}e^{-\frac{1}{24}\nu^{\frac13}t}\norm{f^{in}}_{H^s}.
	\end{equation}
	Whereas if we multiply \eqref{eq:ftoy} by $m^{-2}w^{-2}f$ we obtain 
	\begin{align*}
		\frac12 \Dt|w^{-1}m^{-1}f|^2=&-\left(\frac{\dt m}{m}+\frac{\dt w}{w}+\nu p-\frac{\dt p}{p}\right)|w^{-1}m^{-1}f|^2\\
		\leq&-\nu^{\frac13}|w^{-1}m^{-1}f|^2,
	\end{align*}
	where in the last line we have used \eqref{bd:dtwdeltanup}. Hence, one has $\norm{w^{-1}m^{-1}f}_{H^s}\leq e^{-\nu^\frac13t}\norm{f^{in}}_{H^s}$. Then, since $m\approx 1$, see \eqref{def:m}, from \eqref{bd:w} we infer  
	\begin{equation*}
		\norm{f}_{H^s}=\norm{wm(w^{-1}m^{-1}f)}_{H^s}\lesssim \nu^{-\frac23}\norm{w^{-1}m^{-1}f}_{H^s}\lesssim \nu^{-\frac23}e^{-\nu^{\frac13}t}\norm{f^{in}}_{H^s}. 
	\end{equation*}
	\begin{remark}
		As shown in the computations above, we see that by using the weight $p^{-1}$ or $(mw)^{-1}$ we obtain the same asymptotic behaviour. The advantage of the weight $w$ with respect to $p$ it is clearly the uniform bound \eqref{bd:w}, so that we do not have to pay regularity to translate the estimates from weighted to unweighted quantities. Notice also that, in view of \eqref{bd:dtwdeltanu13}, to control the growth given by $\dt p/p$ it is enough to have a dissipative term with constant coefficients. This property will be crucial for the density,  where we can hope to recover a similar dissipative term by exploiting its coupling with the divergence, see for example \eqref{eq:dtE}.
	\end{remark}
	Let us prove Lemma \ref{lemma:propw}
	\begin{proof}
		The proof of \eqref{bd:w} and \eqref{bd:wp-1} readily follows by the definition of $w$ given in \eqref{def:explw}. 
		
		To obtain \eqref{bd:dtwdeltanup} and \eqref{bd:dtwdeltanu13}, when $\dt p\leq 0$, namely $0\leq t \leq \eta/k$, in account of the property \eqref{bd:propm} there is nothing to prove. When $t\in [\eta/k,\eta/k+\beta \nu^{-\frac13}]$ we make use of the definition of $\dt w/w$, see \eqref{def:w}, and the property \eqref{bd:propm}. In all the other cases we have $|t-\eta/k|\geq \beta \nu^{-\frac13}$. Therefore, appealing to \eqref{bd:trivnup} and \eqref{bd:p'pnu}, we infer 
		\begin{align*}
			&\delta_\beta \nu p-\frac{\dt p}{p}\geq  \nu^{\frac13}(\delta_\beta \beta^2-2\beta^{-1})\geq \delta_\beta\nu^{\frac13}, \\
			& \delta_{\beta}\nu^{\frac13}-\frac{\dt p}{p}\geq  \nu^{\frac13}(\delta_\beta -2\beta^{-1})\geq \frac{\delta_\beta}{2}\nu^{\frac13}, 
		\end{align*}
		where we have also used that $\beta>2$ and  $ \max\{2(\beta(\beta^2-1))^{-1},4\beta^{-1}\}<\delta_\beta\leq 1$, hence the proof is over.
	\end{proof}
	We are now ready to introduce the weighted energy functional.
	\subsubsection{The weighted estimate and the proof of Theorem \ref{th:NSnoloss}}
	Having defined the weight $w$ in \eqref{def:explw}, in the following proposition we present the weighted estimate which allow us to prove Theorem \ref{th:NSnoloss}.
	\begin{proposition}
		\label{prop:bdnoloss}
		Let $s\geq 0$, $\mu \leq 1/2, \ M>0$ be such that $M\leq \min \{\mu^{-\frac12},\nu^{-\frac13}\}$. If $\rho^{in}\in H^{s+1}(\mathbb{T}\times \mathbb{R})$ and $\alpha^{in},\omega^{in}\in H^{s}(\mathbb{T}\times \mathbb{R})$ then
		\begin{equation}
			\label{bd:RAnoloss}
			\begin{split}
				\frac{1}{M}&\norm{(w^{-\frac34}p^{\frac12}{\hR})(t)}_{H^s}+\norm{(w^{-\frac34}{\hA})(t)}_{H^s}+\norm{(w^{-\frac34}(\hXi-\nu M^2\hA))(t)}_{H^s}\\
				&\lesssim e^{-\frac{1}{64}\nu^\frac13 t}(\norm{\nabla \rho^{in}}_{H^s}+\norm{\alpha^{in}}_{H^s}+\norm{\rho^{in}+\omega^{in}-\nu M^2\alpha^{in}}_{H^s}).
			\end{split}
		\end{equation}
	\end{proposition} 
	Notice that with respect to Proposition \eqref{th:enhancedcomp} we have replaced the weight $p^{-\frac34}$ with $w^{-\frac34}$ for $A$ and $\Xi-\nu M^2A$. Instead, for the density, in view of \eqref{bd:wp-1} the same asymptotic behaviour is expected since $w^{-\frac34}p^{\frac12}\gtrsim p^{-\frac14}$.
	
	Appealing to Proposition \ref{prop:bdnoloss}, we first prove Theorem \ref{th:NSnoloss}. Then, we present the proof of Proposition \eqref{prop:bdnoloss}.
	
	\begin{proof}[Proof of Theorem \ref{th:NSnoloss}]
		To prove \eqref{bd:arhoomen}, by the change of coordinates $X=x-yt,\ Y=y$ we have 
		\begin{align*}
			\norm{\alpha(t)}_{L^2}+\frac{1}{M}\norm{\nabla \rho(t)}_{L^2}+\norm{\omega(t)}_{L^2}=&\norm{A(t)}_{L^2}+\frac{1}{M}\norm{(\nabla_L R)(t)}_{L^2}+\norm{\Omega(t)}_{L^2}\\
			=&\norm{\hA(t)}_{L^2}+\frac{1}{M}\norm{(p^{\frac12} \hR)(t)}_{L^2}+\norm{\hOmega(t)}_{L^2}.
		\end{align*}
		Then, observe that
		\begin{equation}
			\label{bd:trivOm}
			|\hOmega|\leq |\hXi-\nu M^2\hA|+|\hR|+\nu M^2|\hA|\leq |\hXi-\nu M^2\hA|+M\frac{1}{M}p^{\frac12}|\hR|+\nu M^2|\hA|.
		\end{equation} 
		Hence we get 
		\begin{align*}
			\norm{\alpha(t)}_{L^2}+&\frac{1}{M}\norm{\nabla \rho(t)}_{L^2}+\norm{\omega(t)}_{L^2}\\
			\lesssim&\norm{\hA(t)}_{L^2}+\frac{1}{M}\norm{(p^{\frac12} \hR)}_{L^2}+\norm{(\hXi-\nu M^2\hA)(t)}_{L^2}\\
			\lesssim &\nu^{-\frac12}\bigg(\norm{(w^{-\frac34}\hA)(t)}_{L^2}+\frac{1}{M}\norm{(w^{-\frac34}p^{\frac12} \hR)(t)}_{L^2}\\
			&\qquad +\norm{(w^{-\frac34}(\hXi-\nu M^2\hA))(t)}_{L^2}\bigg),
		\end{align*}
		where in the last line we have used \eqref{bd:w}, namely $w^{\frac34}\lesssim \nu^{-\frac12}$. Therefore, \eqref{bd:arhoomen} follows by combining the bound above with \eqref{bd:trivOm} and \eqref{bd:RAnoloss}.
		
		The inequality \eqref{bd:vrhoen} instead is obtained as follows. By the Helmholtz decomposition and the change of variable \eqref{def:movframe}, we first observe that
		\begin{align*}
			\norm{\vv(t)}_{L^2}+\frac{1}{M}\norm{\rho(t)}_{L^2}\leq& \norm{(\nabla_L\Delta_L^{-1}A)(t)}_{L^2}+\norm{(\nabla^\perp_L\Delta_L^{-1}\Omega)(t)}_{L^2}+\frac{1}{M}\norm{R(t)}_{L^2}\\
			\leq &\norm{(p^{-\frac12}\hA)(t)}_{L^2}+\norm{(p^{-\frac12}\hOmega)(t)}_{L^2}+\frac{1}{M}\norm{(p^{-\frac12}p^{\frac12}\hR)(t)}_{L^2}.
		\end{align*}
		Then, by using \eqref{bd:w}-\eqref{bd:wp-1} we have
		\begin{align*}
			\norm{\vv(t)}_{L^2}+\frac{1}{M}\norm{\rho(t)}_{L^2}\leq&\norm{(w^{\frac14}w^\frac12p^{-\frac12}(w^{-\frac34}\hA))(t)}_{L^2}+\norm{(w^{\frac14}w^\frac12p^{-\frac12}(w^{-\frac34}\hOmega))(t)}_{L^2}\\
			&+\frac{1}{M}\norm{(w^{\frac14}w^\frac12p^{-\frac12}(w^{-\frac34}p^\frac12\hR))(t)}_{L^2}\\
			\lesssim&\nu^{-\frac16} \bigg(\norm{(w^{-\frac34}\hA)(t)}_{L^2}+\norm{( w^{-\frac34}\hOmega)(t)}_{L^2}\\
			&\qquad \qquad + \norm{(w^{-\frac34}p^{\frac12}\hR)(t)}_{L^2}\bigg), 
		\end{align*}
		whence concluding the proof by combining the bound above with \eqref{bd:trivOm} and \eqref{bd:RAnoloss}.
	\end{proof} 
	It thus remain to prove Proposition \eqref{prop:bdnoloss}. We do not present the proof in detail since it will be similar to the one of Proposition \eqref{th:enhancedcomp}.
	\begin{proof}[Proof of Proposition \eqref{prop:bdnoloss}]
		Consider the weight defined in \eqref{def:m}, we introduce the following weighted variables
		\begin{equation}
			\label{def:ZiNS2}
			\begin{split}
				&Z^w_1(t)=\frac{1}{M}\langle k,\eta \rangle^s(m^{-1}w^{-\frac34}p^{\frac12}\widehat{R})(t), \quad Z^w_2(t)=\langle k,\eta \rangle^s(m^{-1}w^{-\frac34}\widehat{A})(t), \\
				&Z^w_3(t)=\langle k,\eta \rangle^s(m^{-1}w^{-\frac34}(\Xi-\nu M^2A))(t).
			\end{split}
		\end{equation}
		Notice that with respect to \eqref{def:ZiNScomp1}, for $Z^w_2$ and $Z^w_3$ we have replaced $p^{-\frac34}$ with $w^{-\frac34}$.  Then, in account of \eqref{eq:NShR}, \eqref{eq:NShA} and \eqref{eq:dtXinuAphys1}, we observe that 
		\begin{align}
			\label{eq:dtZw1}\dt Z^w_1=&-\left(\frac{\dt m}{m}+\frac34(\frac{\dt w}{w}-\frac{\dt p}{p})\right)Z^w_1-\frac14 \frac{\dt p}{p}Z_1^w-\frac{1}{M}p^\frac12Z^w_2,\\
			\label{eq:dtZw2}\dt Z^w_2=&-\left(\frac{\dt m}{m}+\mu p+\frac34(\frac{\dt w}{w}-\frac{\dt p}{p})\right)Z^w_2+\frac14\frac{\dt p}{p}Z^w_2\\
			\notag &+\left(\frac{1}{M}p^{\frac12}+\frac{2Mk^2}{p^{\frac32}}\right)Z^w_1-2\frac{ k^2}{p}Z^w_3-2\nu M^2\frac{k^2}{p}Z^w_2,\\
			\label{eq:dtZw3}\dt Z^w_3=&-\left(\frac{\dt m}{m}+\nu p\right)Z^w_3-\frac34\frac{\dt w}{w}Z^w_3+\nu(\mu-\nu)M^2pZ^w_2\\
			\notag &-\nu M^2\frac{\dt p}{p}Z^w_2-2\nu M^3 \frac{k^2}{p^{\frac32}}Z^w_1+2\nu M^2\frac{k^2}{p}Z^w_3+2\nu^2M^4 \frac{k^2}{p}Z^w_2.
		\end{align}
		Besides the first term on the left-hand side, the equations \eqref{eq:dtZw1}-\eqref{eq:dtZw2} and \eqref{eq:Z1}-\eqref{eq:Z2} have the same structure. The only difference between the equation \eqref{eq:dtZw3} and \eqref{eq:Z3} is that in \eqref{eq:dtZw3} we have $\displaystyle -\frac34 \frac{\dt w}{w}$ whereas in \eqref{eq:Z3} there is $\displaystyle  -\frac34 \frac{\dt p}{p}$. Hence, we define the energy functional as done in \eqref{def:Et1}-\eqref{def:Et2}, namely
		\begin{align}
			\label{def:Ewt1}	E^w(t)=\frac12 \bigg(&\left(1+M^2\frac{(\dt p)^2}{p^3}\right)|Z^w_1|^2(t)+|Z^w_2|^2(t)+ |Z^w_3|^2(t)\\
			\label{def:Ewt2}&+(\frac{M}{2} \frac{\dt p}{p^{\frac32}}\Re(\bar{Z}^w_1Z^w_2))(t)-\frac{M\nu^\frac13}{2}( p^{-\frac12}\Re(\bar{Z}^w_1Z^w_2))(t)\bigg),
		\end{align}
		which is  clearly coercive and satisfy the same bounds given in \eqref{bd:coerc1}-\eqref{bd:coerc2}.
		By analogous computations done to obtain \eqref{eq:dtE}, we have that
		\begin{equation}
			\label{eq:dtEw}\begin{split}
				\Dt E^w(t)=&-\left(\frac{\dt m}{m}+\mu p+\frac34(\frac{\dt w}{w}-\frac{\dt p}{p})\right)|Z^w_2|^2\\
				&-\left(\frac{\dt m}{m}+\frac{\nu^\frac13}{4}\big(1+2M^2\frac{k^2}{p^2}\big)+\frac34(\frac{\dt w}{w}-\frac{\dt p}{p})\right)|w_1|^2\\
				&-\left(\frac{\dt m}{m}+\nu p\right)|Z_3|^2+\sum_{i=1}^5\mathcal{D}^w_i+\sum_{i=1}^6\mathcal{I}^w_i,
			\end{split}
		\end{equation}
		where $\mathcal{D}_i^w$, for $i=1,\dots,5$, are defined as in \eqref{def:D1}-\eqref{def:D2} by replacing $Z_j$ with $Z_j^w$ for $j=1,2,3$. Analogously, $\mathcal{I}^w_i$, for $i=1,\dots, 6$ and $i\neq 3$, are defined as in \eqref{def:I11}-\eqref{def:I6}. Instead, the term $\mathcal{I}_3^w$ is given by 
		\begin{equation}
			\label{def:Iw3}
			\mathcal{I}^w_3=-\frac34 \frac{\dt w}{w}|Z_3^w|^2+2\nu M^2\frac{k^2}{p}|Z_3^w|^2.
		\end{equation}
		In particular, with respect to \eqref{def:I3}, there is the great advantage that $\dt w/w \geq 0$, meaning that we can bound $\mathcal{I}^w_3$ just with the last term in the right-hand side of \eqref{def:Iw3}. Therefore, thanks to \eqref{def:Iw3}, by making the same estimates given in \eqref{bd:D2}-\eqref{bd:D5} and \eqref{bd:I1}-\eqref{bd:I6} we infer 
		\begin{align}
			\label{eq:dtEw1}
			\Dt E^w(t)\leq &-\left(\frac{1}{16}(\frac{\dt m}{m}+\mu p)+\frac34(\frac{\dt w}{w}-\frac{\dt p}{p})\right)|Z^w_2|^2\\
			&-\left(\frac{1}{16}\left(\frac{\dt m}{m}+\nu^\frac13\right)+\frac34(\frac{\dt w}{w}-\frac{\dt p}{p})\right)|Z_1^w|^2\\
			\label{eq:dtEw3}&-\frac{\nu^\frac13}{8}\frac{M^2 (\dt p)^2}{p^3}|Z_1^w|^2\\
			\label{eq:dtEw4}&-\frac{1}{16}\left(\frac{\dt m}{m}+\nu p\right)|Z^w_3|^2+4\left(C_M\frac{k^2}{p}+\frac{1}{2}\frac{\dt m}{m}\right)E^w(t),
		\end{align}
		where to obtain \eqref{eq:dtEw3} we have used \eqref{bd:dtpp3}, whereas to get the last term in \eqref{eq:dtEw4} we have used the coercivity properties of the functional, see \eqref{bd:coerc1}-\eqref{bd:coerc2}. We now have to exploit the properties \eqref{bd:dtwdeltanup}-\eqref{bd:dtwdeltanu13}. In particular, in our case we have $\delta_\beta=1/12$, hence choosing $\beta>4$ in the definition of $w$, see \eqref{def:explw}, we get
		\begin{align*}
			\label{eq:dtEw11}
			\Dt E^w(t)\leq& -\frac{1}{16}\nu^\frac13|Z^w_2|^2-\frac{1}{32}\nu^\frac13|Z_1^w|^2-\frac{\nu^\frac13}{8}\frac{M^2 (\dt p)^2}{p^3}|Z_1^w|^2\\
			&-\frac{1}{16}\nu^\frac13|Z^w_3|^2+4\left(C_M\frac{k^2}{p}+\frac{1}{2}\frac{\dt m}{m}\right)E^w(t).\\
			\leq& -\frac{1}{32}\nu^\frac13E^w(t)+4\left(C_M\frac{k^2}{p}+\frac{1}{2}\frac{\dt m}{m}\right)E^w(t),
		\end{align*} 
		where in the last line we have used \eqref{bd:coerc2}. Therefore, by applying Gr\"onwall's Lemma we obtain 
		\begin{equation}
			\label{bd:EwGron}
			E^w(t)\lesssim e^{-\frac{1}{32}\nu^\frac13t }E^w(0).
		\end{equation}
		Then, in view of the definition \eqref{def:ZiNS2} and the coercivity of the functional, we also know that 
		\begin{align*}
			\sum_k \int E^w(t)d\eta \approx& 	\frac{1}{M}\norm{(w^{-\frac34}p^{\frac12}{\hR})(t)}_{H^s}^2+\norm{(w^{-\frac34}{\hA})(t)}_{H^s}^2\\
			&+\norm{(w^{-\frac34}(\hXi-\nu M^2\hA))(t)}_{H^s}^2,
		\end{align*}
		hence, thanks to \eqref{bd:EwGron}, the proof is over.
	\end{proof}
	\subsection*{Acknowledgements} We would like to thank J. Bedrossian, M. Coti Zelati and T. Gallay for useful comments and suggestions about this work. 
	
	PA, MD were partially supported through the INdAM-GNAMPA project \virg{\textit{Esistenza, limiti singolari e comportamento asintotico per equazioni Eulero/Navier–Stokes–Korteweg}}. MD was partially supported by the Royal Society through the (URF/R1/191492). PM acknowledges partial support by the PRIN-MIUR project 2015YCJY3A\_003 \virg{\textit{Hyperbolic Systems of Conservation Laws and Fluid Dynamics: Analysis and Applications}}.
	\bibliographystyle{siam}
	\bibliography{bibliohuge}

\begin{bibdiv}
\begin{biblist}

\bib{antonelli2020linear}{article}{
      author={Antonelli, Paolo},
      author={Dolce, Michele},
      author={Marcati, Pierangelo},
       title={Linear stability analysis for 2d shear flows near {Couette} in
  the isentropic compressible euler equations},
        date={2020},
     journal={arXiv preprint arXiv:2003.01694},
}

\bib{arnold1965conditions}{incollection}{
      author={Arnold, Vladimir~I},
       title={Conditions for non-linear stability of stationary plane
  curvilinear flows of an ideal fluid},
        date={1965},
   booktitle={Vladimir i. arnold-collected works},
   publisher={Springer},
       pages={19\ndash 23},
}

\bib{bakas2009mechanisms}{article}{
      author={Bakas, Nikolaos~A},
       title={Mechanisms underlying transient growth of planar perturbations in
  unbounded compressible shear flow},
        date={2009},
     journal={Journal of Fluid Mechanics},
      volume={639},
       pages={479\ndash 507},
}

\bib{BGM15III}{article}{
      author={Bedrossian, Jacob},
      author={Germain, Pierre},
      author={Masmoudi, Nader},
       title={On the stability threshold for the 3{D} {C}ouette flow in
  {S}obolev regularity},
        date={2017},
     journal={Ann. of Math. (2)},
      volume={185},
      number={2},
       pages={541\ndash 608},
}

\bib{bedrossian2019stability}{article}{
      author={Bedrossian, Jacob},
      author={Germain, Pierre},
      author={Masmoudi, Nader},
       title={Stability of the {Couette} flow at high reynolds numbers in two
  dimensions and three dimensions},
        date={2019},
     journal={Bulletin of the American Mathematical Society},
      volume={56},
      number={3},
       pages={373\ndash 414},
}

\bib{bedrossian2020he}{article}{
      author={Bedrossian, Jacob},
      author={He, Siming},
       title={Inviscid damping and enhanced dissipation of the boundary layer
  for 2{D} {N}avier-{S}tokes linearized around {C}ouette flow in a channel},
        date={2020},
        ISSN={0010-3616},
     journal={Comm. Math. Phys.},
      volume={379},
      number={1},
       pages={177\ndash 226},
         url={https://doi.org/10.1007/s00220-020-03851-9},
      review={\MR{4152270}},
}

\bib{bedrossian2015inviscid}{article}{
      author={Bedrossian, Jacob},
      author={Masmoudi, Nader},
       title={Inviscid damping and the asymptotic stability of planar shear
  flows in the 2d euler equations},
        date={2015},
     journal={Publications math{\'e}matiques de l'IH{\'E}S},
      volume={122},
      number={1},
       pages={195\ndash 300},
}

\bib{bedrossian2016landau}{article}{
      author={Bedrossian, Jacob},
      author={Masmoudi, Nader},
      author={Mouhot, Cl{\'e}ment},
       title={Landau damping: paraproducts and gevrey regularity},
        date={2016},
     journal={Annals of PDE},
      volume={2},
      number={1},
       pages={4},
}

\bib{bedrossian2016enhanced}{article}{
      author={Bedrossian, Jacob},
      author={Masmoudi, Nader},
      author={Vicol, Vlad},
       title={Enhanced dissipation and inviscid damping in the inviscid limit
  of the {N}avier--{S}tokes equations near the two dimensional {C}ouette flow},
        date={2016},
     journal={Archive for Rational Mechanics and Analysis},
      volume={219},
      number={3},
       pages={1087\ndash 1159},
}

\bib{bianchini2020linear}{article}{
      author={Bianchini, Roberta},
      author={Coti~Zelati, Michele},
      author={Dolce, Michele},
       title={Linear inviscid damping for shear flows near {Couette} in the 2d
  stably stratified regime},
        date={2020},
     journal={arXiv preprint arXiv:2005.09058},
}

\bib{blumen1975shear}{article}{
      author={Blumen, W},
      author={Drazin, PG},
      author={Billings, DF},
       title={Shear layer instability of an inviscid compressible fluid. part
  2},
        date={1975},
     journal={Journal of Fluid Mechanics},
      volume={71},
      number={2},
       pages={305\ndash 316},
}

\bib{blumen1970shear}{article}{
      author={Blumen, William},
       title={Shear layer instability of an inviscid compressible fluid},
        date={1970},
     journal={Journal of Fluid Mechanics},
      volume={40},
      number={4},
       pages={769\ndash 781},
}

\bib{bodo2005spiral}{article}{
      author={Bodo, G},
      author={Chagelishvili, G},
      author={Murante, G},
      author={Tevzadze, A},
      author={Rossi, P},
      author={Ferrari, Attilio},
       title={Spiral density wave generation by vortices in keplerian flows},
        date={2005},
     journal={Astronomy \& Astrophysics},
      volume={437},
      number={1},
       pages={9\ndash 22},
}

\bib{caglioti1998time}{article}{
      author={Caglioti, E},
      author={Maffei, C},
       title={Time asymptotics for solutions of vlasov--poisson equation in a
  circle},
        date={1998},
     journal={Journal of statistical physics},
      volume={92},
      number={1-2},
       pages={301\ndash 323},
}

\bib{chagelishvili1994hydrodynamic}{article}{
      author={Chagelishvili, GD},
      author={Rogava, AD},
      author={Segal, IN},
       title={Hydrodynamic stability of compressible plane {{Couette}} flow},
        date={1994},
     journal={Physical Review E},
      volume={50},
      number={6},
       pages={R4283},
}

\bib{chagelishvili1997linear}{article}{
      author={Chagelishvili, GD},
      author={Tevzadze, AG},
      author={Bodo, G},
      author={Moiseev, SS},
       title={Linear mechanism of wave emergence from vortices in smooth shear
  flows},
        date={1997},
     journal={Physical review letters},
      volume={79},
      number={17},
       pages={3178},
}

\bib{chen2018transition}{article}{
      author={Chen, Qi},
      author={Li, Te},
      author={Wei, Dongyi},
      author={Zhang, Zhifei},
       title={Transition threshold for the 2-d {Couette} flow in a finite
  channel},
        date={2020},
     journal={Archive for rational mechanics and analysis},
      volume={238},
       pages={125\ndash 183},
}

\bib{chen2020transition}{article}{
      author={Chen, Qi},
      author={Li, Te},
      author={Wei, Dongyi},
      author={Zhang, Zhifei},
       title={Transition threshold for the 3d {Couette} flow in a finite
  channel},
        date={2020},
     journal={arXiv preprint arXiv:2006.00721},
}

\bib{chen2019linear}{article}{
      author={Chen, Qi},
      author={Wei, Dongyi},
      author={Zhang, Zhifei},
       title={Linear stability of pipe poiseuille flow at high reynolds number
  regime},
        date={2019},
     journal={arXiv preprint arXiv:1910.14245},
}

\bib{cotizelati2020separation}{article}{
      author={Coti~Zelati, Michele},
      author={Dolce, Michele},
       title={Separation of time-scales in drift-diffusion equations on
  {$\Bbb{R}^2$}},
        date={2020},
        ISSN={0021-7824},
     journal={J. Math. Pures Appl. (9)},
      volume={142},
       pages={58\ndash 75},
         url={https://doi.org/10.1016/j.matpur.2020.08.001},
      review={\MR{4149684}},
}

\bib{zelati2019stochastic}{article}{
      author={Coti~Zelati, Michele},
      author={Drivas, Theodore~D},
       title={A stochastic approach to enhanced diffusion},
        date={2019},
     journal={arXiv preprint arXiv:1911.09995},
}

\bib{cotizelati2020enhancedNS}{article}{
      author={Coti~Zelati, Michele},
      author={Elgindi, Tarek~M.},
      author={Widmayer, Klaus},
       title={Enhanced dissipation in the {N}avier-{S}tokes equations near the
  {P}oiseuille flow},
        date={2020},
        ISSN={0010-3616},
     journal={Comm. Math. Phys.},
      volume={378},
      number={2},
       pages={987\ndash 1010},
         url={https://doi.org/10.1007/s00220-020-03814-0},
      review={\MR{4134940}},
}

\bib{deng2020stability}{article}{
      author={Deng, Wen},
      author={Wu, Jiahong},
      author={Zhang, Ping},
       title={Stability of {Couette} flow for 2d boussinesq system with
  vertical dissipation},
        date={2020},
     journal={arXiv preprint arXiv:2004.09292},
}

\bib{drazin1977shear}{article}{
      author={Drazin, PG},
      author={Davey, A},
       title={Shear layer instability of an inviscid compressible fluid. part
  3},
        date={1977},
     journal={Journal of Fluid Mechanics},
      volume={82},
      number={2},
       pages={255\ndash 260},
}

\bib{drazin2004hydrodynamic}{book}{
      author={Drazin, Philip~G},
      author={Reid, William~Hill},
       title={Hydrodynamic stability},
   publisher={Cambridge university press},
        date={2004},
}

\bib{duck1994linear}{article}{
      author={Duck, Peter~W},
      author={Erlebacher, Gordon},
      author={Hussaini, M~Yousuff},
       title={On the linear stability of compressible plane {{Couette}} flow},
        date={1994},
     journal={Journal of Fluid Mechanics},
      volume={258},
       pages={131\ndash 165},
}

\bib{eckart1963extension}{article}{
      author={Eckart, Carl},
       title={Extension of howard's circle theorem to adiabatic jets},
        date={1963},
     journal={The Physics of Fluids},
      volume={6},
      number={8},
       pages={1042\ndash 1047},
}

\bib{farrell2000transient}{article}{
      author={Farrell, BF},
      author={Ioannou, PJ},
       title={Transient and asymptotic growth of two-dimensional perturbations
  in viscous compressible shear flow},
        date={2000},
     journal={Physics of Fluids},
      volume={12},
      number={11},
       pages={3021\ndash 3028},
}

\bib{gallay2020enhanced}{article}{
      author={Gallay, Thierry},
       title={Enhanced dissipation and axisymmetrization of two-dimensional
  viscous vortices},
        date={2018},
        ISSN={0003-9527},
     journal={Arch. Ration. Mech. Anal.},
      volume={230},
      number={3},
       pages={939\ndash 975},
         url={https://doi.org/10.1007/s00205-018-1262-0},
      review={\MR{3851053}},
}

\bib{glatzel1989linear}{article}{
      author={Glatzel, W},
       title={The linear stability of viscous compressible plane {Couette}
  flow},
        date={1989},
     journal={Journal of Fluid Mechanics},
      volume={202},
       pages={515\ndash 541},
}

\bib{goldreich1965gravitational}{article}{
      author={Goldreich, Peter},
      author={Lynden-Bell, D},
       title={I. {G}ravitational stability of uniformly rotating disks},
        date={1965},
     journal={Monthly Notices of the Royal Astronomical Society},
      volume={130},
      number={2},
       pages={97\ndash 124},
}

\bib{goldreich1965ii}{article}{
      author={Goldreich, Peter},
      author={Lynden-Bell, D},
       title={{II}. {S}piral arms as sheared gravitational instabilities},
        date={1965},
     journal={Monthly Notices of the Royal Astronomical Society},
      volume={130},
      number={2},
       pages={125\ndash 158},
}

\bib{grenier2020landau}{article}{
      author={Grenier, Emmanuel},
      author={Nguyen, Toan~T},
      author={Rodnianski, Igor},
       title={Landau damping for analytic and gevrey data},
        date={2020},
     journal={arXiv preprint arXiv:2004.05979},
}

\bib{guo2012decay}{article}{
      author={Guo, Yan},
      author={Wang, Yanjin},
       title={Decay of dissipative equations and negative sobolev spaces},
        date={2012},
     journal={Communications in Partial Differential Equations},
      volume={37},
      number={12},
       pages={2165\ndash 2208},
}

\bib{hanifi1996transient}{article}{
      author={Hanifi, Ardeshir},
      author={Schmid, Peter~J},
      author={Henningson, Dan~S},
       title={Transient growth in compressible boundary layer flow},
        date={1996},
     journal={Physics of Fluids},
      volume={8},
      number={3},
       pages={826\ndash 837},
}

\bib{hau2015comparative}{article}{
      author={Hau, Jan-Niklas},
      author={Chagelishvili, George},
      author={Khujadze, George},
      author={Oberlack, Martin},
      author={Tevzadze, Alexander},
       title={A comparative numerical analysis of linear and nonlinear
  aerodynamic sound generation by vortex disturbances in homentropic constant
  shear flows},
        date={2015},
     journal={Physics of Fluids},
      volume={27},
      number={12},
       pages={126101},
}

\bib{holm1985nonlinear}{article}{
      author={Holm, Darryl~D},
      author={Marsden, Jerrold~E},
      author={Ratiu, Tudor},
      author={Weinstein, Alan},
       title={Nonlinear stability of fluid and plasma equilibria},
        date={1985},
     journal={Physics reports},
      volume={123},
      number={1-2},
       pages={1\ndash 116},
}

\bib{hu1998linear}{article}{
      author={Hu, Sean},
      author={Zhong, Xiaolin},
       title={Linear stability of viscous supersonic plane {Couette} flow},
        date={1998},
     journal={Physics of Fluids},
      volume={10},
      number={3},
       pages={709\ndash 729},
}

\bib{ionescu2020nonlinear}{article}{
      author={Ionescu, Alexandru~D},
      author={Jia, Hao},
       title={Nonlinear inviscid damping near monotonic shear flows},
        date={2020},
     journal={arXiv preprint arXiv:2001.03087},
}

\bib{jia2020linear}{article}{
      author={Jia, Hao},
       title={Linear inviscid damping in gevrey spaces},
        date={2020},
     journal={Archive for Rational Mechanics and Analysis},
      volume={235},
      number={2},
       pages={1327\ndash 1355},
}

\bib{kagei2011asymptotic}{article}{
      author={Kagei, Yoshiyuki},
       title={Asymptotic behavior of solutions of the compressible
  navier--stokes equation around the plane {Couette} flow},
        date={2011},
     journal={Journal of Mathematical Fluid Mechanics},
      volume={13},
      number={1},
       pages={1\ndash 31},
}

\bib{kagei2012asymptotic}{article}{
      author={Kagei, Yoshiyuki},
       title={Asymptotic behavior of solutions to the compressible
  navier--stokes equation around a parallel flow},
        date={2012},
     journal={Archive for Rational Mechanics and Analysis},
      volume={205},
      number={2},
       pages={585\ndash 650},
}

\bib{kelvin1887stability}{article}{
      author={Kelvin, Lord},
       title={Stability of fluid motion: rectilinear motion of viscous fluid
  between two parallel plates},
        date={1887},
     journal={Phil. Mag},
      volume={24},
      number={5},
       pages={188\ndash 196},
}

\bib{lees1946investigation}{article}{
      author={Lees, Lester},
      author={Lin, Chia~Chiao},
       title={Investigation of the stability of the laminar boundary layer in a
  compressible fluid},
        date={1946},
}

\bib{li2017stability}{article}{
      author={Li, Hai-Liang},
      author={Zhang, Xingwei},
       title={Stability of plane couette flow for the compressible
  navier--stokes equations with navier-slip boundary},
        date={2017},
     journal={Journal of Differential Equations},
      volume={263},
      number={2},
       pages={1160\ndash 1187},
}

\bib{li2020pseudo}{article}{
      author={Li, Te},
      author={Wei, Dongyi},
      author={Zhang, Zhifei},
       title={Pseudospectral and spectral bounds for the {O}seen vortices
  operator},
        date={2020},
        ISSN={0012-9593},
     journal={Ann. Sci. \'{E}c. Norm. Sup\'{e}r. (4)},
      volume={53},
      number={4},
       pages={993\ndash 1035},
      review={\MR{4157106}},
}

\bib{lin2011inviscid}{article}{
      author={Lin, Zhiwu},
      author={Zeng, Chongchun},
       title={Inviscid dynamical structures near {Couette} flow},
        date={2011},
     journal={Archive for rational mechanics and analysis},
      volume={200},
      number={3},
       pages={1075\ndash 1097},
}

\bib{liss2020sobolev}{article}{
      author={Liss, Kyle},
       title={On the sobolev stability threshold of 3d {Couette} flow in a
  uniform magnetic field},
        date={2020},
     journal={Communications in Mathematical Physics},
       pages={1\ndash 50},
}

\bib{makita2000two}{article}{
      author={Makita, Makoto},
      author={Miyawaki, Kenji},
      author={Matsuda, Takuya},
       title={Two-and three-dimensional numerical simulations of accretion
  discs in a close binary system},
        date={2000},
     journal={Monthly Notices of the Royal Astronomical Society},
      volume={316},
      number={4},
       pages={906\ndash 916},
}

\bib{malik2008linear}{article}{
      author={Malik, M},
      author={Dey, J},
      author={Alam, Meheboob},
       title={Linear stability, transient energy growth, and the role of
  viscosity stratification in compressible plane {Couette} flow},
        date={2008},
     journal={Physical Review E},
      volume={77},
      number={3},
       pages={036322},
}

\bib{masmoudi2020stability}{article}{
      author={Masmoudi, Nader},
      author={Said-Houari, Belkacem},
      author={Zhao, Weiren},
       title={Stability of couette flow for 2d boussinesq system without
  thermal diffusivity},
        date={2020},
     journal={arXiv preprint arXiv:2010.01612},
}

\bib{masmoudi2020nonlinear}{article}{
      author={Masmoudi, Nader},
      author={Zhao, Weiren},
       title={Nonlinear inviscid damping for a class of monotone shear flows in
  finite channel},
        date={2020},
     journal={arXiv preprint arXiv:2001.08564},
}

\bib{matsumura1983initial}{article}{
      author={Matsumura, Akitaka},
      author={Nishida, Takaaki},
       title={Initial boundary value problems for the equations of motion of
  compressible viscous and heat-conductive fluids},
        date={1983},
     journal={Communications in Mathematical Physics},
      volume={89},
      number={4},
       pages={445\ndash 464},
}

\bib{MV11}{article}{
      author={Mouhot, Cl{\'e}ment},
      author={Villani, C{\'e}dric},
       title={On {L}andau damping},
        date={2011},
     journal={Acta Math.},
      volume={207},
      number={1},
       pages={29\ndash 201},
         url={http://dx.doi.org/10.1007/s11511-011-0068-9},
}

\bib{romanov1973stability}{article}{
      author={Romanov, Valerii~Aleksandrovich},
       title={Stability of plane-parallel {Couette} flow},
        date={1973},
     journal={Functional analysis and its applications},
      volume={7},
      number={2},
       pages={137\ndash 146},
}

\bib{subbiah1990stability}{article}{
      author={Subbiah, M},
      author={Jain, RK},
       title={Stability of compressible shear flows},
        date={1990},
     journal={Journal of mathematical analysis and applications},
      volume={151},
      number={1},
       pages={34\ndash 41},
}

\bib{trefethen1993hydrodynamic}{article}{
      author={Trefethen, Lloyd~N},
      author={Trefethen, Anne~E},
      author={Reddy, Satish~C},
      author={Driscoll, Tobin~A},
       title={Hydrodynamic stability without eigenvalues},
        date={1993},
     journal={Science},
      volume={261},
      number={5121},
       pages={578\ndash 584},
}

\bib{wei2018linear}{article}{
      author={Wei, Dongyi},
      author={Zhang, Zhifei},
      author={Zhao, Weiren},
       title={Linear inviscid damping for a class of monotone shear flow in
  sobolev spaces},
        date={2018},
     journal={Communications on Pure and Applied Mathematics},
      volume={71},
      number={4},
       pages={617\ndash 687},
}

\bib{WZZ17}{article}{
      author={Wei, Dongyi},
      author={Zhang, Zhifei},
      author={Zhao, Weiren},
       title={Linear inviscid damping and vorticity depletion for shear flows},
        date={2019},
     journal={Ann. PDE},
      volume={5},
      number={1},
       pages={Art. 3, 101},
}

\bib{yaglom2012hydrodynamic}{book}{
      author={Yaglom, Akiva~M},
       title={Hydrodynamic instability and transition to turbulence},
   publisher={Springer Science \& Business Media},
        date={2012},
      volume={100},
}

\bib{yang2018linear}{article}{
      author={Yang, Jincheng},
      author={Lin, Zhiwu},
       title={Linear inviscid damping for {Couette} flow in stratified fluid},
        date={2018},
     journal={Journal of Mathematical Fluid Mechanics},
      volume={20},
      number={2},
       pages={445\ndash 472},
}

\bib{zillinger2016linear}{article}{
      author={Zillinger, Christian},
       title={Linear inviscid damping for monotone shear flows in a finite
  periodic channel, boundary effects, blow-up and critical sobolev regularity},
        date={2016},
     journal={Archive for Rational Mechanics and Analysis},
      volume={221},
      number={3},
       pages={1449\ndash 1509},
}

\bib{zillinger2020enhanced}{article}{
      author={Zillinger, Christian},
       title={On enhanced dissipation for the boussinesq equations},
        date={2020},
     journal={arXiv preprint arXiv:2004.08125},
}

\bib{zillinger2020boussinesq}{article}{
      author={Zillinger, Christian},
       title={On the boussinesq equations with non-monotone temperature
  profiles},
        date={2020},
     journal={arXiv preprint arXiv:2011.02316},
}

\end{biblist}
\end{bibdiv}
\end{document}